\documentclass[10pt,a4paper]{scrartcl}%

\pdfoutput=1

\usepackage[T1]{fontenc}
\usepackage{epsfig}
\usepackage{graphicx}
\usepackage{amssymb}
\usepackage{amsmath}
\usepackage[utf8]{inputenc}
\usepackage{amsfonts}
\usepackage{amsthm}
\usepackage{stmaryrd}

\usepackage{tikz}
\usepackage{tikz-qtree}
\usetikzlibrary{fit,calc,positioning,decorations.pathreplacing,matrix,trees,patterns}
\usepackage{xcolor}

\usepackage{xcolor, soul}
\sethlcolor{red}

\usepackage{geometry}
\usepackage{hyperref}
\usepackage[numbers]{natbib}

\makeatletter

\@addtoreset{equation}{section}
\makeatother

\theoremstyle{plain}
\newtheorem{theorem}{Theorem}[section]
\newtheorem{corollary}[theorem]{Corollary}
\newtheorem{lemma}[theorem]{Lemma}

\newtheorem{proposition}[theorem]{Proposition}
\newtheorem{conj}[theorem]{Conjecture}
\newtheorem{claim}[theorem]{Claim}
\theoremstyle{definition}
\newtheorem{definition}[theorem]{Definition}
\newtheorem{prop_def}[theorem]{Proposition-definition}
\newtheorem{convention}[theorem]{Convention}
\theoremstyle{remark}
\newtheorem{example}[theorem]{Example}
\newtheorem{remark}[theorem]{Remark}

\DeclareMathOperator*{\argmin}{argmin}

\newcommand{\DSF}{\text{DSF}}

\newcommand{\CFD}{\text{CFD}}
\newcommand{\MBD}{\text{MBD}}
\newcommand{\CO}{\text{CO}}
\newcommand{\Cyl}{\text{Cyl}}
\newcommand{\Slice}{\text{Slice}}
\newcommand{\EmptySlice}{\text{EmptySlice}}

\newcommand{\Approx}{\text{Approx}}
\newcommand{\FC}{\text{FC}}
\newcommand{\Bi}{\text{BI}}
\newcommand{\Rad}{\text{Rad}}
\newcommand{\Leb}{\text{Leb}}
\newcommand{\N}{\mathbb{N}}

\newcommand{\tvc}[1]{#1}

\makeindex

\title{The Directed Spanning Forest in the hyperbolic space}
\author{Lucas Flammant \footnote{Univ. de Valenciennes, CNRS, EA 4015 - LAMAV, F-59313 Valenciennes Cedex 9, France}
\footnote{Univ. Lille, CNRS, UMR 8524 - Laboratoire P. Painlevé, F-59000 Lille, France, E-mail: lucas.flammant@univ-lille.fr}}
\date{\today}

\begin{document}

\maketitle

\begin{abstract}

The Euclidean Directed Spanning Forest is a random forest in $\mathbb{R}^d$ introduced by Baccelli and Bordenave in 2007 and we introduce and study here the analogous tree in the hyperbolic space. The topological properties of the Euclidean DSF have been stated for $d=2$ and conjectured for $d \ge 3$ (see further): it should be a tree for $d \in \{2,3\}$ and a countable union of disjoint trees for $d \ge 4$. Moreover, it should not contain bi-infinite branches whatever the dimension $d$. In this paper, we construct the Hyperbolic Directed Spanning Forest (HDSF) and we give a complete description of its topological properties, which are radically different from the Euclidean case. Indeed, for any dimension, the hyperbolic DSF is a tree containing infinitely many bi-infinite branches, whose asymptotic directions are investigated. The strategy of our proofs consists in exploiting the Mass Transport Principle, which is adapted to take advantage of the invariance by isometries. Using appropriate mass transports is the key to carry over the hyperbolic setting ideas developed in percolation and for spanning forests. This strategy provides an upper-bound for horizontal fluctuations of trajectories, which is the key point of the proofs. To obtain the latter, we exploit the representation of the forest in the hyperbolic half space.
\end{abstract}

\bigbreak
\bigbreak
\textbf{Acknowledgments.} This work has been supervised by David Coupier and Viet Chi Tran, who helped a lot and finalize the manuscript. It has been supported by the LAMAV (Université Polytechnique des Hauts de France) and the Laboratoire P. Painlevé (Lille). It has also benefitted from the GdR GeoSto 3477, the Labex CEMPI (ANR-11-LABX-0007-01) and the ANR PPPP (ANR-16-CE40-0016).


\section{Introduction}

Many random objects present radically different behaviors depending on whether they are considered in an Euclidean or hyperbolic setting. With the dichotomy of recurrence and transience for symmetric random walks \cite{markvorsen}, one of the most emblematic example is given by continuum percolation models. Indeed, the Poisson-Boolean model contains at most one unbounded component in $\mathbb{R}^d$ \cite{continuum} whereas it admits a non-degenerate regime with infinitely many unbounded components in the hyperbolic plane \cite{tykesson}. The difference is mainly explained by the fact that the hyperbolic space is non-amenable, i.e. the measure of the boundary of a large subset is not negligible with respect to its volume. For this reason, arguments based on comparison between volume and surface, such as the Burton and Keane argument \cite{burton}, fail in hyperbolic geometry. For background in hyperbolic geometry, the reader may refer to \tvc{\cite{cannon,paupert}, or to \cite{chavel,iversen,ratcliffe} for more exhaustive information}.

Hence there is a growing interest for the study of random models in a hyperbolic setting. Let us cite the work of Benjamini \& Schramm about the Bernoulli percolation on regular tilings and Voronoï tesselation in the hyperbolic plane \cite{benjamini}, and the work of Calka \& Tykesson about asymptotic visibility in the Poisson-Boolean model \cite{visibility}. Mean characteristics of the Poisson-Voronoï tesselation have also been studied in a general Riemannian manifold by Calka et. al. \cite{calka2018mean}. In addition, huge differences between amenable and non-amenable spaces are well known in a discrete context \cite{percolationbeyondZd, lyons2017probability, pete2014probability}.

It is in order to highlight new behaviors that we investigate the study of the hyperbolic counterpart of the Euclidean Directed Spanning Forest (DSF) defined in $\mathbb{R}^d$ by \cite{baccelli}. To our knowledge, this is the first study of a spanning forest in the hyperbolic space.

\bigbreak

Geometric random trees are well studied in the literature since it interacts with many other fields, such as communication networks, particles systems or population dynamics. We can cite the work of Norris and Turner \cite{norristurner} establishing some scaling limits for a model of planar aggregation. In addition, hyperbolic random graphs are well-fitted to modelize social networks (\textit{e.g.} \cite{socialnetworks} or \cite{vanderhoorn}).

The Euclidean DSF is a random forest whose introduction has been motivated by applications for communication networks.  The set of vertices is given by a homogeneous Poisson Point Process (PPP) $N$ of intensity $\lambda$ in $\mathbb{R}^d$. For any unit vector $u \in \mathbb{R}^d$, the (Euclidean) DSF with direction $u$ is the graph obtained by connecting each point $z \in N$ to the closest point to $z$ among all points $z' \in N$ that are further in the direction $u$ (i.e. such that $\langle z'-z,u \rangle>0$). \tvc{The latter point is called the parent or ancestor of $z$, and denoted by $A(z)$. Because the parent of a point always exists, it is possible to construct from any $z\in N$ a sequence of Poisson points $(z_n)_{n\in \N}$ such that $z_1=z$ and $z_{n+1}=A(z_n)$ for all $n\in \N$. This sequence constitutes a branch of the DSF, which is by construction infinite in the direction $u$. The branch is said to be bi-infinite if every point is the parent of another point of $N$. In this case, the bi-infinite branch can be represented by a sequence $(z_n)_{n\in \mathbb{Z}}$ where for each $n\in \mathbb{Z}$, $z_{n+1}$ is the parent of $z_n$.}

The topological properties of the Euclidean DSF are now well-understood. Coupier and Tran showed in 2010 that, in dimension $2$, it is a tree that does not contain bi-infinite branches \cite{dsf}. Their proof used a Burton \& Keane argument, so it cannot be carried over the hyperbolic case. In addition, Coupier, Saha, Sarkar \& Tran developed tools to split trajectories in i.i.d. blocks \cite{dsftobw}, and these tools may permit to show that the Euclidean DSF is a tree in dimension $2$ and $3$ but not in dimension $4$ and more (see \cite[Remark 18, p.35]{dsftobw}). This dichotomy and the absence of bi-infinite branches for any dimension $d$  have been proved for similar models defined on lattices and presenting less geometrical dependencies \cite{roy2016random, gangopadhyay, athreya, ferrari}. \tvc{Indeed, compared with these models, the DSF exhibits complex geometrical dependencies: given a Poisson point $z \in N$, knowing the position of its parent $A(z)$ implies that some region is empty in the direction $u$, which affects the future evolution of trajectories (Figure \ref{fig:constructionhyper}), and thus destroys nice Markov properties available for the models on lattices mentioned above. For instance, if $u$ is the upward direction, the knowledge of $z$ and $A(z)$ indicates that the upper part of a hyperbolic ball centered at $z$ and having $A(z)$ on its boundary is empty of Poisson point.}

The hyperbolic space is a homogeneous space with constant negative curvature, that can be chosen equal to $-1$ without loss of generality. It can be represented by several models, all related by isometries. We will work in the $(d+1)$-dimensional upper half-space $H:=\{z=(x_1,...,x_d,y) \in \mathbb{R}^{d+1},~y>0\}$  \cite[p.69]{cannon}
endowed with the metric
\begin{eqnarray*}
ds_H^2:=\frac{dx_1^2+...+dx_d^2+dy^2}{y^2}.
\end{eqnarray*}

This representation is well adapted to our problem as explained in Section \ref{S:generalsettings}. Now, let us define the hyperbolic DSF. The set of vertices is given by a homogeneous PPP $N$ of intensity $\lambda>0$ in $(H,ds_H^2)$. Given a point $x \in N$, choosing its closest vertex according to a given direction can be interpreted in different ways in the hyperbolic space. Hence several hyperbolic DSF could be considered. We choose to connect each point $z=(x_1,...,x_d,y) \in N$ to the closest point to $z$ \tvc{(for the hyperbolic distance)}  among all points $z'=(x'_1,...,x'_d,y') \in N$ with $y'>y$ (called the parent or ancestor of $z$, $z$ being the descendant of $z'$). An equivalent and more intrinsic definition of this model using horodistances is given in the core of the article. The main interest of this definition is the preservation of the link between the DSF and the Radial Spanning Tree (RST) existing in the Euclidean setting. The (Euclidean) RST, also defined by \cite{baccelli}, is a random tree whose set of vertices is given by a homogeneous PPP $N$ plus the origin $0$ and defined by connecting each point $z \in N$ to the closest point to $z$ among all points $z' \in N \cup \{0\}$ such that $\|z'\|<\|z\|$. In the Euclidean setting, the DSF approximates locally the RST in distribution far from the origin. This remains true in the hyperbolic setting for our definition of hyperbolic DSF. \tvc{The results established in this paper are of great importance for the study of the hyperbolic RST that is not considered here, but in \cite{flammant-RST}}. A simulation of the hyperbolic DSF is given in Figure \ref{fig:simuhypdsf}.

In this paper, we give a complete description of the topological properties of the hyperbolic DSF which present huge differences with the Euclidean case : whatever the dimension $d$, the hyperbolic DSF is a.s. a tree (Theorem \ref{Thm:coalescence}) and admits infinitely many bi-infinite branches (Theorem \ref{Thm:biinfinite}).

\begin{figure}[!h]
    \centering
    \includegraphics[scale=0.37]{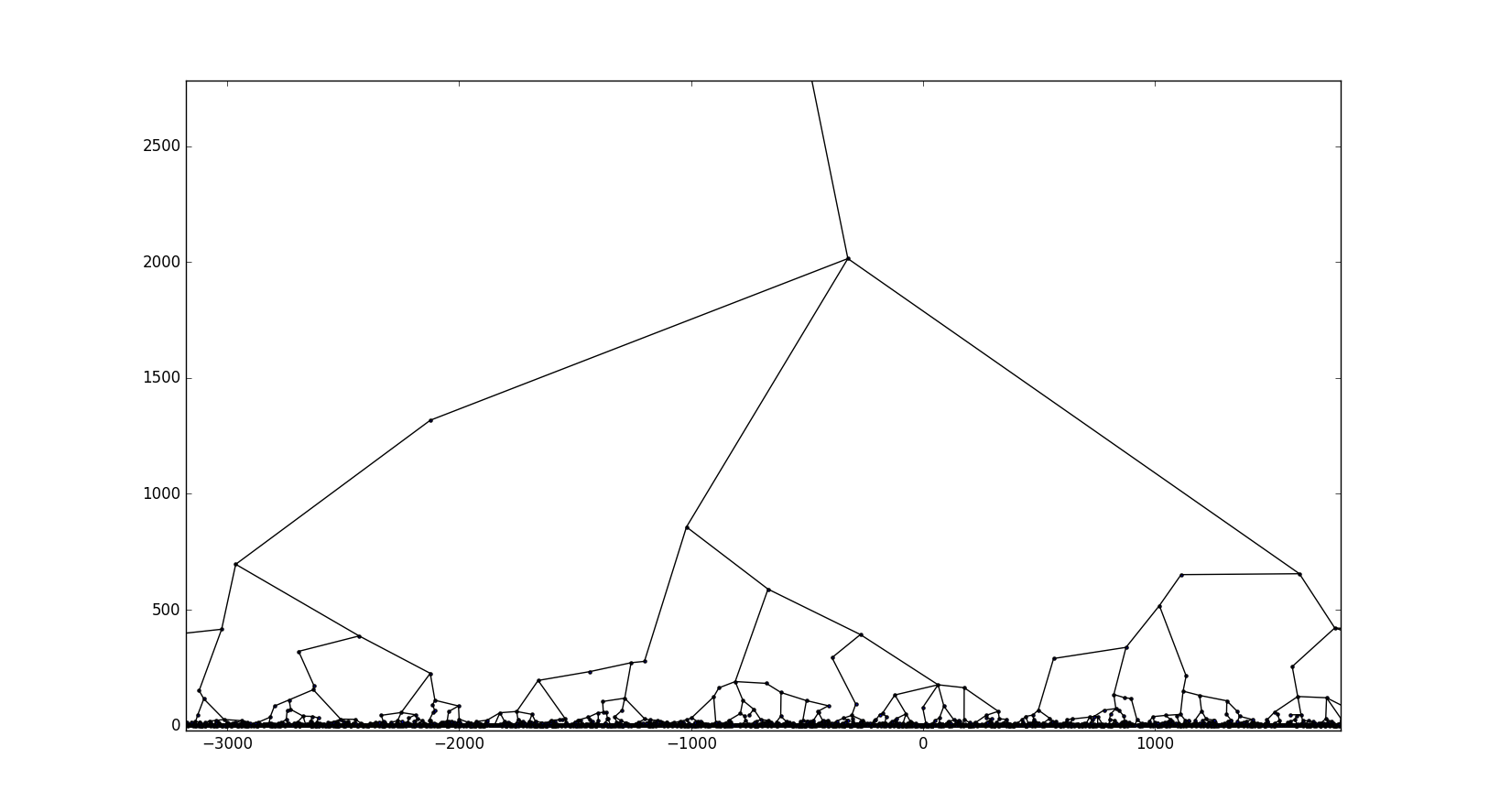}
    \includegraphics[scale=0.37]{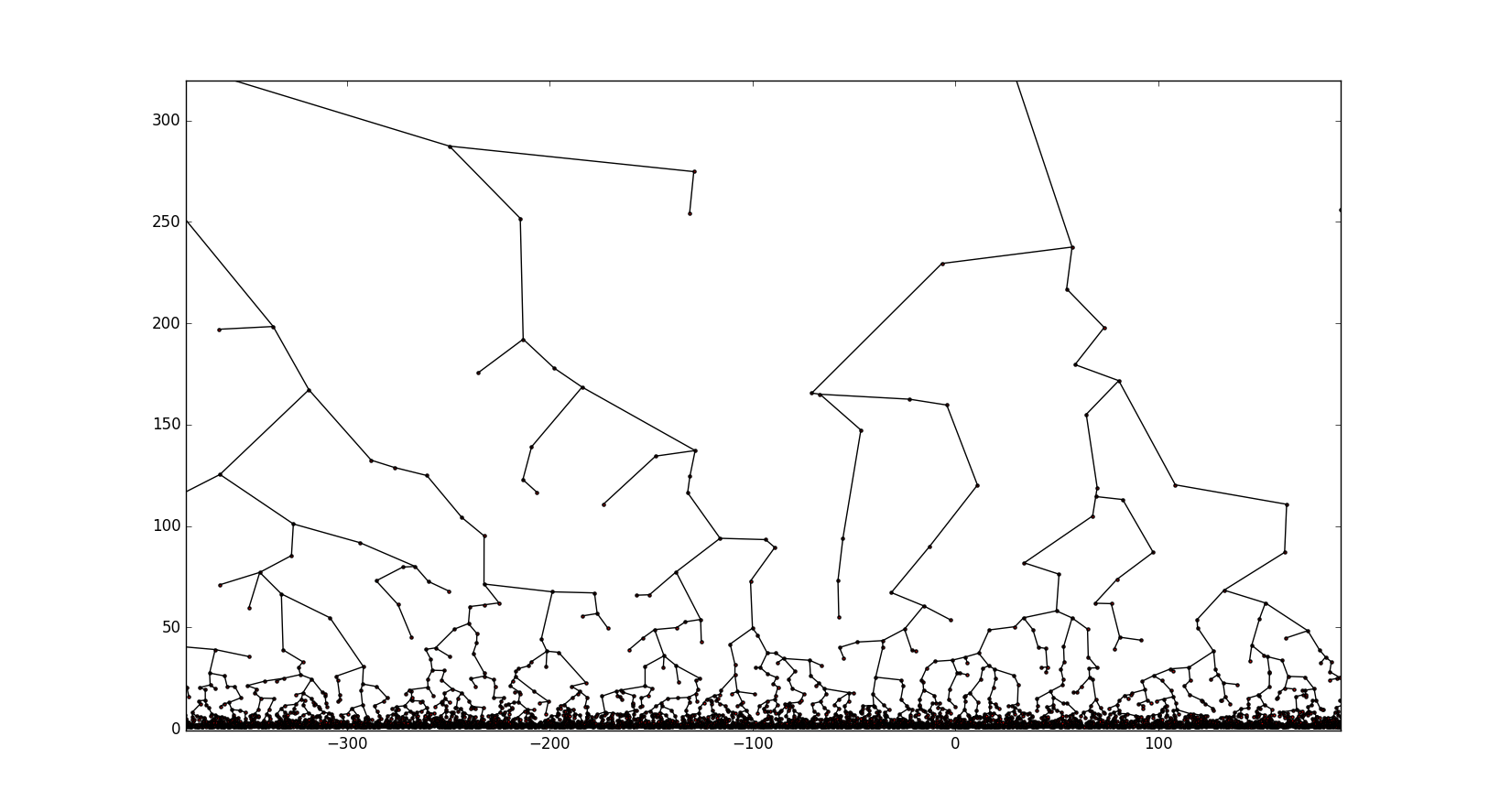}
    \caption{Simulation of the DSF in the half-space model, with  $\lambda=1$ (to the top) and $\lambda=10$ (to the bottom). The local behavior of the hyperbolic DSF depends on the intensity $\lambda$ because the space is curved. For instance the average number of \tvc{descendants}  is larger for $\lambda=1$ than for $\lambda=10$. But its topological properties do not depend on $\lambda$ (see Theorems \ref{Thm:coalescence} and \ref{Thm:biinfinite}). \tvc{Notice that throughout the paper, the $(d+1)$-th coordinate $y$ is seen as the time; the future is upward and the past is downward. Thus, coalescence events being encountered when advancing in $y$, the tree structure is such that ancestors (in the tree) are in the `future' of descendants (i.e. with higher $y$-coordinate).}}
    \label{fig:simuhypdsf}
\end{figure}

For the DSF, being a tree means that all branches eventually coalesce, i.e. any two points $z,z' \in N$ have a common ancestor somewhere in the DSF. For any bounded measurable subset $A \in \mathbb{R}^d$, we can define its \emph{coalescing height} $\tau_A$ as the smallest $\tau \ge 0$ such that every branches passing through $A \times \{e^0\}$ have merged below ordinate $e^\tau$ (see Definition \ref{Def:coalesceheight}). Here is our first main result:

\begin{theorem}
\label{Thm:coalescence}
For all $d \ge 1$ and for all intensity $\lambda>0$, the hyperbolic DSF in dimension $d+1$ is a.s. a tree.
Moreover, if $d=1$, for all $a>0$, the coalescing height $\tau_{[-a,a]}$ admits exponential tail decay: for any $t>0$,
\begin{eqnarray*}
\mathbb{P}\left[ \tau_{[-a,a]}>t \right] \le 2\alpha_0ae^{-t},
\end{eqnarray*}
where the positive constant $\alpha_0$ will be specified later (in Proposition \ref{Prop:finiteintensity}).
\end{theorem}

The coalescence in every dimension is specific to the hyperbolic case, since in the Euclidean case, it is expected that the DSF is a tree in dimension $2$ and $3$ only. \tvc{For $d=1$, the coalescing height admits exponential tail decay in the hyperbolic case whereas $\mathbb{P}[\tau_{[-a,a]}>t] \sim Ca/\sqrt{t}$ when $t \to \infty$ and for a constant $C>0$ in the Euclidean case \cite{dsftobw} (heuristically, it can be compared to the coalescing height of two Brownian motions starting from $(-a,0)$ and $(a,0)$ and directed to the top).} The coalescence of all trajectories can be heuristically explained by the fact that two trajectories starting from the ordinate $e^0$ almost remain in a cone: their typical horizontal deviations at ordinate $e^t$ are of order $e^t$. So, roughly speaking, they remain at the same hyperbolic horizontal distance from each other as they go up, implying that they must coalesce. This behavior is due to the hyperbolic metric and does not occur in $\mathbb{R}^d$.

\tvc{Recall that a bi-infinite branch can be identified with a sequence of Poisson points $(z_n)_{n \in \mathbb{Z}} \in N^\mathbb{Z}$ such that $z_{n+1}$ is the parent of $z_n$ for all $n \in \mathbb{Z}$. In the upper half-space representation, the limit $\lim_{n\rightarrow -\infty} z_n$, when it exists, is necessarily a point of the boundary hyperplane $\mathbb{R}^d \times \{0\}$. The points of $\mathbb{R}^d\times \{0\}$ are `points at infinity' for the hyperbolic geometry. We say that a bi-infinite branch has an asymptotic direction if there exists $(x,0) \in \mathbb{R}^d \times \{0\}$ such that the path converges towards the past to this point at infinity: $\lim_{n\rightarrow -\infty} z_n=(x,0)$.}  Our second main result concerns bi-infinite branches and their asymptotic directions. \tvc{The $(d+1)$-th coordinate $y$ is seen as the time; the future is upward and the past is downward.}

\begin{theorem}
\label{Thm:biinfinite}
For all $d \ge 1$ and for all intensity $\lambda>0$, the following assertions hold outside a set of probability zero: \\
(i) The hyperbolic DSF admits infinitely many bi-infinite branches. \\
(ii) Every bi-infinite branch of the hyperbolic DSF converges toward the past. \\
(iii) \tvc{For every $x \in \mathbb{R}^d$, there exists a bi-infinite branch of the hyperbolic DSF that converges to $(x,0)$ toward the past.} \\
(iv) \tvc{Such a branch is unique for almost every $x \in \mathbb{R}^d$. The set of $x \in \mathbb{R}^d$ for which there is no uniqueness is dense in $\mathbb{R}^d$.} It is moreover countable in the bi-dimensional case (i.e. if $d=1$).

Moreover, for any deterministic $x \in \mathbb{R}^d$, the bi-infinite branch converging to $(x,0)$ toward the past is unique a.s.
\end{theorem}

This result is specific to the hyperbolic case since the Euclidean DSF does not admit bi-infinite branches \cite{dsf}.

The existence of bi-infinite branches can be suggested by the following heuristic. In the half-space representation, because of the hyperbolic metric, the density of points decreases with the height, implying that a typical point will have a mean number of descendants larger than 1. Thus the tree of descendants of a typical point could be compared to a supercritical Galton-Watson tree and then should be infinite with positive probability. According to this heuristic, the hyperbolic DSF should admit infinitely many bi-infinite branches. On the contrary, in the Euclidean DSF, a typical point has a mean number of descendants equal to $1$ (it can be seen by the Mass Transport Principle discussed later). Hence the corresponding analogy leads to a critical Galton-Watson tree which is finite a.s., which suggests that the Euclidean DSF does not admit bi-infinite branches.

The key point of the proofs is to upper-bound horizontal fluctuations of trajectories, both forward (i.e. upward) and backward (i.e. downward). Roughly speaking, we establish that a typical trajectory almost remains in a forward cone. Controlling the fluctuations of trajectories is a common technique to obtain the existence of infinite branches and to control their asymptotic directions: it is done for the RST in \cite{baccelli}, and also by Howard \& Newman in the context of first passage percolation \cite{howard}.

To do it, we proceed in two steps. We first use a percolation argument to upper-bound fluctuations on a small vertical distance. Then we generalise the bound on an arbitrary vertical distance by a new technique based on the Mass Transport Principle (Theorem \ref{Thm:masstranport}). This principle roughly says that for a given mass transport with isometries invariance properties (Definition \ref{Def:diaginv}), the incoming mass is equal to the outgoing mass. Most models in hyperbolic space studied in the literature are invariant by the group of all isometries, which is unimodular (i.e. the left-invariant Haar measure is also right-invariant), and the Mass Transport in the hyperbolic space \cite[pp. 13-14]{benjamini} is well-adapted for these models. However, the hyperbolic DSF is only invariant by the group of isometries that fix a particular point at infinity, and this group is not unimodular \tvc{(the invariance properties are explained in Section} \ref{S:properties}). For this reason, the Mass Transport Principle cannot be used in the same way. Instead, we introduce a slicing of $H$ into levels $\mathbb{R}^d \times \{e^t\}$ for $t \in \mathbb{R}$, and we typically consider appropriate mass transports from $\mathbb{R}^d \times \{e^{t_1}\}$ to $\mathbb{R}^d \times \{e^{t_2}\}$ with $t_1 \le t_2$, in order to obtain useful equalities by identifying the incoming mass and the outgoing mass.

\bigbreak

The rest of the paper is organized as follows. In Section \ref{S:hyperbolicgeometry}, we set some reminders on hyperbolic geometry. We also define the hyperbolic DSF in more details and we give its basic properties.

In Section \ref{S:tools}, we state some technical results derived from the Mass Transport Principle in $\mathbb{R}^d$. These results are well fitted to take advantage of the translation invariance of the model in distribution. 

In Section \ref{S:prooffluctuations}, we establish upper-bounds for horizontal fluctuations of forward (i.e. upward) and backward (i.e. downward) trajectories, which is the key point of the proofs. In particular, we show that a typical trajectory almost stays in a forward cone. A block control argument is used to upper-bound the fluctuations on a small vertical distance, and Mass Transport arguments are used to deduce the general bound.

In Section \ref{S:coalescencealld}, we exploit the control of horizontal fluctuations to prove the coalescence in any dimensions (Theorem \ref{Thm:coalescence}). The idea behind it is that, since two trajectories almost stay in cone, they roughly stay at the same hyperbolic horizontal distance to each other as they go up, thus they must coalesce. We also give a simpler proof of coalescence in the bi-dimensional case based on planarity.

In Section \ref{S:proofofinfinitebranches}, we prove Theorem \ref{Thm:biinfinite}. We use a second moment technique to show the existence of bi-infinite branches, based on the control of forward horizontal fluctuations. We exploit the control of fluctuations backward to prove the results concerning asymptotic directions.

\section{Definition of the hyperbolic DSF and general settings}
\label{S:generalsettings}

We denote by $\mathbb{N}$ the set of non-negative integers and by $\mathbb{N}^*$ the set of positive integers. After recalling some facts on the hyperbolic space $\mathbb{H}^{d+1}$ (Section \ref{S:hyperbolicgeometry}), we consider an homogeneous PPP on $\mathbb{H}^{d+1}$ and construct the hyperbolic DSF (Section \ref{S:defDSF}).

\subsection{The hyperbolic space and the half-space model}
\label{S:hyperbolicgeometry}

For $d \in \mathbb{N}^*$, the $(d+1)$-dimensional hyperbolic space, denoted by $\mathbb{H}^{d+1}$, is a $(d+1)$-dimensional Riemannian manifold, homogeneous and isotropic, and of constant negative curvature equal to $-1$. The reader may refer to \tvc{\cite[Section 4.6]{ratcliffe}} for background on hyperbolic geometry. The space $\mathbb{H}^{d+1}$ can be described with several isometric models and we will work in the \emph{half-space model}, defined as the upper half-space $H:=\{(x_1,...,x_d,y) \in \mathbb{R}^{d+1},~y>0\}$ endowed with the metric 
\begin{eqnarray*}
ds_H^2:=\frac{dx_1^2+...+dx_d^2+dy^2}{y^2}.
\end{eqnarray*}
The metric $ds_H^2$ naturally gives a volume measure $\mu$ on $(H,ds_H^2)$, given by \tvc{(see \cite[Th. 4.6.7]{ratcliffe})}
\begin{eqnarray*}
d\mu=\frac{dx_1...dx_ddy}{y^{d+1}}.
\end{eqnarray*}
Note that the last coordinate $y$ plays a special role with respect to the other ones. The metric becomes smaller as we get closer to the boundary hyperplane $\partial H=\mathbb{R}^d \times \{0\}$, and this boundary is at infinite hyperbolic distance from any point of $H$. In the following, we will identify the point $(x_1,...,x_{d+1}) \in H$ with the couple $(x,y) \in \mathbb{R}^d \times \mathbb{R}_+^*$ with $x:=(x_1,...,x_d)$ and $y:=x_{d+1}$. The coordinate $x$ is referred as the \emph{abscissa} and $y$ as the \emph{ordinate}. For $z=(x,y) \in H$, we denote by $\pi_x(z)=x$ the abscissa of $z$ and by $\pi_y(z)=y$ its ordinate. We also define its $\emph{height}$ as $h(z):=\ln(y)$. \tvc{All the points of same height $h$ constitute the \textit{level} $h$, \textit{i.e.} it is the hyperplane $\mathbb{R}\times \{h\}$ of all points of $H$ with ordinate $e^h$. The height can be positive or negative depending on whether $y \ge 1$ or $y \le 1$. All along the paper, we will use the level $0$, of height $0$ and corresponding to the ordinate $y=e^0$ as a reference.}

\begin{figure}
    \centering
    \includegraphics[scale=0.8]{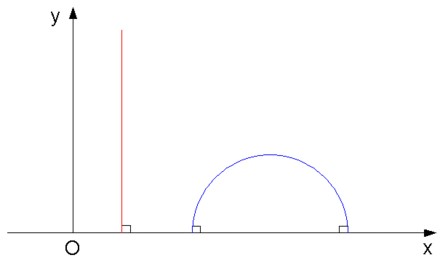}
    \caption{The geodesics on $(H,ds_H^2)$ are of two types: on the one hand, the vertical straight lines $\{x=a\}$ for any $a \in \mathbb{R}^d$ (in \textcolor{red}{red}) and, on the other hand, the semi-circles contained in $H$ and centered at a point of the boundary hyperplane $\partial H$ (in \textcolor{blue}{blue}).}
    \label{fig:halfspace}
\end{figure}

Let us set some general notation. We denote by $d(\cdot,\cdot)$ the hyperbolic distance in $(H,ds_H^2)$ and by $\|\cdot\|$ the Euclidean norm in $\mathbb{R}^d \cup \{\infty\}$, with the convention $\|\infty\|=\infty$. For $z_1,z_2 \in H$, we denote by $[z_1,z_2]$ the geodesic between $z_1$ and $z_2$ and by $[z_1,z_2]_{eucl}$ the Euclidean segment between $z_1$ and $z_2$. For $z \in H$ and $\rho>0$, let $B_H(z,\rho):=\{z' \in H,~d(z,z')<\rho \}$ be the hyperbolic ball centered at $z$ of radius $\rho$. For $x \in \mathbb{R}^d$ and $r>0$, let $B_{\mathbb{R}^d}(x,r):=\{x' \in \mathbb{R}^d,~\|x'-x\|<r\}$ be the Euclidean ball centered at $x$ of radius $r$. If there is no ambiguity, we will replace the notations $B_H(\cdot,\cdot)$ and $B_{\mathbb{R}^d}(\cdot,\cdot)$ with $B(\cdot,\cdot)$. Finally, for $z=(x,y) \in H$ and $\rho>0$, we define the \emph{upper semi-ball}
\begin{eqnarray*}
B_+(z,\rho):=B_H(z,\rho) \cap (\mathbb{R}^d \times (y,\infty)).
\end{eqnarray*}
It is the part of the (hyperbolic) ball $B_H(z,\rho)$ that is above the hyperplane $\mathbb{R}^d \times \{y\}$ containing $z$. This hyperplane is a curved subspace of $(H,ds_H^2)$, so it does not split $B_H(z,\rho)$ in two isometric pieces.

We now state some useful facts about the half-space model. In $(H,ds_H^2)$, hyperbolic spheres are also Euclidean spheres. Moreover, the Euclidean center and the hyperbolic center belong to the same vertical line, but they do not coincide \cite[Fact 1, p.86]{cannon}. Hence, if $z_1,z_2,z_3 \in H$ are aligned in this order for the Euclidean metric (i.e. $z_2 \in [z_1,z_3]_{eucl}$), then $d(z_1,z_2) \le d(z_1,z_3)$. We will use the following distance formula to do precise calculations in the half space model:

\begin{proposition}[Distance formula]
\label{Prop:disthalfspace}
Let $z_1=(x_1,y_1) \in H$ and $z_2=(x_2,y_2) \in H$. Let $\kappa=\|x_1-x_2\|/y_1$ and $v=y_2/y_1$. Then
\begin{eqnarray}
\label{E:distanceformula}
d(z_1,z_2)=2\tanh^{-1}\left(\sqrt{\frac{\kappa^2+(v-1)^2}{\kappa^2+(v+1)^2}}\right)=\Phi\left( \frac{\kappa^2+(v+1)^2}{v} \right)
\end{eqnarray}
where $\Phi:[4,+\infty) \rightarrow \mathbb{R}_+$ is increasing and defined as
\begin{eqnarray*}
\Phi(t)=2\tanh^{-1}\left(\sqrt{1-4/t}\right)=\ln\left( \frac{1+\sqrt{1-4/t}}{1-\sqrt{1-4/t}} \right).
\end{eqnarray*}
\end{proposition}

\begin{remark}
\label{Rem:disthalfplane}
Given the ratio $v=y_2/y_1$, the distance $d(z_1,z_2)$ increases with $\kappa$. In particular, when $y_1$, $y_2$ are fixed, the distance $d(z_1,z_2)$ is minimal when $x_1=x_2$.
\end{remark}

The proof of Proposition \ref{Prop:disthalfspace} is given in the Supplementary materials (Section \ref{S:hyperbolicgeometry}). We now discuss some particular cases of the distance formula. For two points on a same vertical straight line,  $z_1=(x,y_1)$ and  $z_2=(x,y_2)$, their distance is $d(z_1,z_2)=|\ln(y_2/y_1)|=|h(z_1)-h(z_2)|.$ This shows that the notion of height is compatible with the hyperbolic distance, this justifies the relevance of this notion. In particular, for $z=(x,e^t)$ and $\rho>0$; consider the hyperbolic (closed) ball $\overline{B}_H(z,\rho)$. Then the \emph{top} (i.e. the point with the highest ordinate) of $\overline{B}_H(z,\rho)$ is precisely $(x,e^{t+\rho})$, and the \emph{bottom} (i.e. the point with the lowest ordinate) of $\overline{B}_H(z,\rho)$ is $(x,e^{t-\rho})$.

For two points on the same horizontal hyperplane, $z_1=(x_1,y)$ and $z_2=(x_2,y)$, denoting by $R=\|x_1-x_2\|$ their horizontal Euclidean distance, their hyperbolic distance can be estimated when $R \to \infty$ by $d(z_1,z_2)=2\ln(R/y)+o(1)$. Moreover, for any $R>0$, $d(z_1,z_2) \le R/y$.

The hyperbolic space $\mathbb{H}^{d+1}$ is equipped with a set of points at infinity, denoted by $\partial \mathbb{H}^{d+1}$. In the half-space model $(H,ds_H^2)$, the points at infinity are identified by the boundary hyperplane $\partial H=\mathbb{R}^d \times \{0\}$, plus an additional point at infinity in all directions, obtained by compactification of the closed half-space $\mathbb{R}^d \times \mathbb{R}_+$. This particular point at infinity will be denoted by $\infty$. 

\subsection{Definition of the hyperbolic DSF}
\label{S:defDSF}

\subsubsection{Poisson point processes}

Let $E=\mathbb{R}^d$ or $\mathbb{H}^d$. For any measurable subset $A \subset E$, we denote by $|A|$ its volume (it is either $\Leb(A)$ in the Euclidean case or $\mu(A)$ in the hyperbolic case). Let us denote by $\mathcal{N}_S$ the space of locally finite subsets of $E$, and for $A \subset E$ measurable, let $\mathcal{N}_S(A)$ be the space of locally finite subsets of $A$. The spaces $\mathcal{N}_S$ and $\mathcal{N}_S(A)$ are equipped with the $\sigma$-algebra generated by counting applications (i.e. of the form $\eta \mapsto \#(\eta \cap K)$ for any compact set $K$).

\begin{definition}[Homogeneous Poisson point process (PPP)]
For $\lambda>0$, a point process $N$ is called \emph{homogeneous Poisson point process of intensity $\lambda$} if for any measurable set $A \subset E$, $\#(N \cap A)$ is distributed according to the Poisson law with parameter $\lambda |A|$.
\end{definition}

It can be shown that there is a unique probability measure on $\mathcal{N}_S$ satisfying this condition. Moreover, if $N$ is a homogeneous PPP and $A_1,..,.A_n \subset E$ are pairwise disjoint measurable subsets, then $N \cap A_1,...,N \cap A_n$ are mutually independent \cite{stochastic}.

\subsubsection{Horodistance}

In $\mathbb{H}^{d+1}$, the \emph{horodistance} formalizes the notion of "distance from a point at infinity".

\begin{definition}[Horodistance functions]
\label{def:horodistance}
Let $z_0 \in \mathbb{H}^{d+1}$ be an arbitrary point, considered as the origin. Given a point at infinity $I \in \partial \mathbb{H}^{d+1} $, the \emph{horodistance function} $\mathcal{H}_I:\mathbb{H}^{d+1} \to \mathbb{R}$ is defined as
\begin{eqnarray}
\label{E:defhorodistance}
\mathcal{H}_I(z):=\lim_{z' \to I} d(z,z')-d(z_0,z').
\end{eqnarray}
\end{definition}

The existence of the limit (\ref{E:defhorodistance}) is proved in \cite[Appendix B]{version_longue}. Any change of the origin point $z_0$ only affects the function $\mathcal{H}_I$ up to an additive constant. So $\mathcal{H}_I$ is naturally defined modulo an additive constant.
 
The level sets of $\mathcal{H}_I$, i.e. the sets of points at the same horodistance to $I$, are called \emph{horospheres} (centerered at $I$). Horospheres in $(H,ds_H^2)$ are represented in Figure \ref{fig:horospheres}.

\begin{proposition}
\label{prop:horosphereshalfplane}
Consider $(H,ds_H^2)$ and recall that the boundary point $\infty$ has been defined in Section \ref{S:hyperbolicgeometry}. The horodistance function $\mathcal{H}_\infty$ is (modulo an additive constant):
\begin{eqnarray*}
\mathcal{H}_\infty((x,y))=-\ln(y).
\end{eqnarray*}
\end{proposition}

We refer to \cite[Appendix B]{version_longue} for a proof.

\begin{figure}[!h]
    \centering
    \begin{tikzpicture}[scale=1.7]

\draw[>=stealth,->] (1.5,-0.8)--(3.5,-0.8);
\draw[>=stealth,->] (1.7,-1)--(1.7,1);
\draw (1.7,-0.8) node[below right] {$(0,0)$};
\filldraw[gray] (2.5,-0.8) circle (2pt);
\draw (2.5,-0.8) node[below right] {$I$};
\draw[dashed] (2.5,-0.6) circle (0.2);
\draw[dashed] (2.5,-0.4) circle (0.4);
\draw[dashed] (2.5,-0.2) circle (0.6);
\draw[dashed] (2.5,0) circle (0.8);

\draw[>=stealth,->] (4.2,-0.8)--(6.2,-0.8);
\draw[>=stealth,->] (4.4,-1)--(4.4,1);
\draw (4.4,-0.8) node[below right] {$(0,0)$};
\draw (5.2,1) node {$I=\infty$};
\draw[dashed] (4.2,-0.6)--(6.2,-0.6);
\draw[dashed] (4.2,-0.4)--(6.2,-0.4);
\draw[dashed] (4.2,-0.2)--(6.2,-0.2);
\draw[dashed] (4.2,0)--(6.2,0);
\end{tikzpicture}

\caption{Horospheres centered at $I$ in the half-space representation.}
\label{fig:horospheres}
\end{figure}

\subsubsection{The DSF in hyperbolic space} 

We now introduce our model of DSF in $\mathbb{H}^{d+1}$. Fix $\lambda>0$ and let $N$ be a homogeneous PPP of intensity $\lambda$ in $\mathbb{H}^{d+1}$. Consider a point at infinity $I \in \partial \mathbb{H}^{d+1}$, devoted to be the direction of the DSF or the \emph{target point}. The choice of $I$ is analogous to the choice of the direction vector $u$ in the Euclidean case. In $\mathbb{R}^d$, each $z \in N$ is connected to its closest Poisson point among those that are "further" than $z$ in some direction $u$. Similarly, in the hyperbolic case, we connect each point $z \in N$ to the closest Poisson point among those that are "further in direction $I$", where being "further in direction $I$" is formalized by the notion of horodistance (Definition \ref{def:horodistance}).

\begin{definition}[Directed Spanning Forest in $\mathbb{H}^{d+1}$]
\label{Def:dsfhyp}
Let $I \in \partial \mathbb{H}^{d+1}$. We call \emph{Directed Spanning Forest (DSF)} in $\mathbb{H}^{d+1}$ of direction $I$ the oriented graph whose set of vertices is $N$ and obtained by connecting each $z \in N$ to its \emph{parent} $A(z)$ defined as $$A(z):=\argmin_{z' \in N,~\mathcal{H}_I(z')<\mathcal{H}_I(z)} d(z'-z).$$
\end{definition}

A sketch of the construction is given in Figure \ref{fig:constructionhyper}. The choice of the direction $I$ only affects the law of the DSF up to an isometry. Indeed, for any two points at infinity $I,I' \in \partial \mathbb{H}^{d+1}$, there exists an isometry that sends $I$ on $I'$. In the following, we will work in the half-space representation $(H,ds_H^2)$ (Section \ref{S:hyperbolicgeometry}) and we set the direction $I=\infty$ for convenience. Indeed, the horodistance function $\mathcal{H}_\infty$ only depends on the ordinate (Proposition \ref{prop:horosphereshalfplane}), and $\mathcal{H}_\infty(z_1)<\mathcal{H}_\infty(z_2)$ if and only if $y_1>y_2$. Thus, the parent of $z$ is its closest Poisson point among those having higher ordinate than $z$.

Definition \ref{Def:dsfhyp} does not specify the shape of edges, but the results announced in Theorems \ref{Thm:coalescence} and \ref{Thm:biinfinite} only concern the graph structure of the hyperbolic DSF, so their veracity do not depend on the way points are connected. Here, we choose to connect each point $z \in N$ to its parent $A(z)$ by the Euclidean segment $[z,A(z)]_{eucl}$. It is more natural to represent edges with hyperbolic geodesics, but the choice of Euclidean segments will appear more convenient for the proofs. The main reason of this choice is that we want that the $y$-coordinate increases along a given edge, and it is not the case using geodesics. Thus, we define the random subset $\DSF \subset H$ as the union of all Euclidean segments $[z,A(z)]_{eucl}$ for $z \in N$: $DSF=\bigcup_{z \in N}[z,A(z)]_{eucl}$.

\begin{remark}
For $z \in N$, by definition of the parent $A(z)$, the upper semi-ball $B_+(z,d(z,A(z)))$ contains no points of $N$.
\end{remark}

\begin{convention}
This (random) upper semi-ball $B_+(z,d(z,A(z)))$ will be more simply denoted by $B_+(z)$.
\end{convention}

\begin{proposition}
\label{Prop:DSFisforest1}
The DSF in $\mathbb{H}^{d+1}$ is a forest a.s.
\end{proposition}

\begin{proof}
Suppose that the hyperbolic DSF contains a cycle $(z_0,...z_{k-1})$. Consider the point of the cycle with the lowest ordinate. Then, by construction, both neighbors of $z_i$ in the cycle \tvc{must be} parents of $z_i$, but $z_i$ has only one parent, this is a contradiction. Therefore, the DSF does not contain cycles, it is a forest.
\end{proof}

\begin{proposition}
\label{Prop:firstprop}
Almost surely, the DSF is non-crossing and has finite degree.
\end{proposition}

Proposition \ref{Prop:firstprop} is show in the Appendix (Section \ref{S:firstprop})

\begin{figure}[!h]
    \centering
\begin{tikzpicture}

\draw [>=stealth,->] (-1,0)--(10,0);
\draw [>=stealth,->] (0,0)--(0,6);
\draw (0,0) node[below] {$0$};
\draw (10,0) node[below] {$x$};
\draw (0,6) node[left] {$y$};
\filldraw [gray] (2,1) circle (2pt);
\filldraw [gray] (3,3) circle (2pt);
\filldraw [gray] (5,4) circle (2pt);
\filldraw [gray] (8,1.5) circle (2pt) node[below left, black] {$z$};
\filldraw [gray] (6,2.5) circle (2pt) node[below left, black] {$A(z)$};

\draw [dashed] (-1,1)--(10,1);
\draw [dashed] (-1,3)--(10,3);
\draw [dashed] (-1,4)--(10,4);
\draw [dashed] (-1,1.5)--(10,1.5);
\draw [dashed] (-1,2.5)--(10,2.5);

\draw  [>=stealth,->] (2,1)--(3,3);
\draw  [>=stealth,->] (3,3)--(5,4);
\draw  [>=stealth,->] (8,1.5)--(6,2.5);
\draw  [>=stealth,->] (6,2.5)--(5,4);

\draw (0.7090055512641942,1) arc (212.84213041441177:-32.842130414411756:1.536590742882148);
\draw (1.06350832689629127,3) arc ( 197.88741291941264:-17.887412919412647:2.0348525745124637);
\draw (6.267949192431123,1.5) arc (210:-30:2);
\draw (4.574780718626077, 2.5) arc (195.90987591526664:-15.909875915266635:1.4819882126724222);
\fill [pattern=north east lines] (10,2.5)--(6,2.5) arc (180:0:2);
\draw (1,6.5) node {$I=\infty$};
\end{tikzpicture}
    \caption{Sketch of construction of the hyperbolic DSF. This picture illustrates dependence phenomenons existing in a single trajectory and between trajectories. Given a Poisson point $z \in N$, knowing the position of its parent $A(z)$ implies the knowledge that some region above $A(z)$, the upper part of a hyperbolic ball centered at $z$ (the crosshatched area) is empty of Poisson points, which affects the future evolution of trajectories, and thus destroys nice Markov properties in the hyperbolic DSF.}
    \label{fig:constructionhyper}
\end{figure}
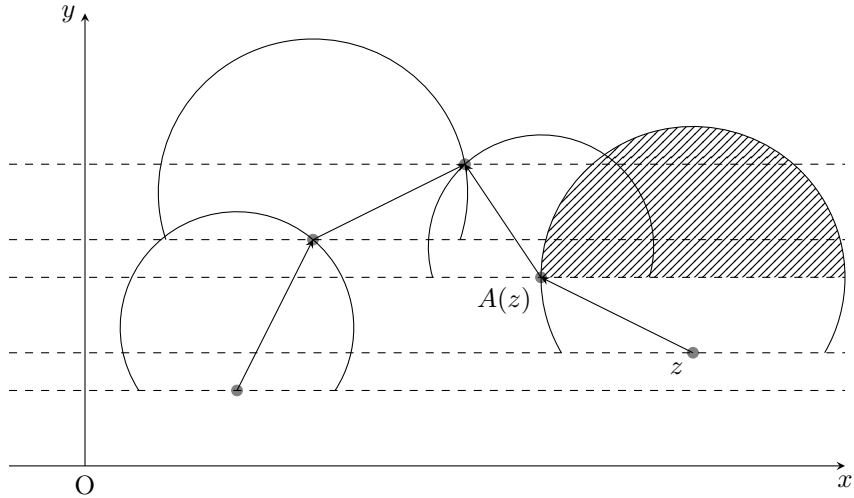

\subsubsection{General notations}

If $X_1,...,X_k$ are random variables, we denote by $\sigma(X_1,...,X_k)$ the $\sigma$-algebra generated by $X_1,...,X_k$. If a random variable $X$ is measurable w.r.t. $\sigma(N)$, then for $\eta \in \mathcal{N}_S$, we denote by $X(\eta)$ the value of $X$ when $N=\eta$. Let us also denote by $T_s:x \mapsto x+s$ the \tvc{translation by} $s$ in $\mathbb{R}^d$.

Let $z \in \DSF$. Since the DSF is non-crossing \cite[Section 3.1]{version_longue}, there exists a unique $z_0 \in N$ such that $z \in [z_0,A(z_0))_{eucl}$. 
Then we define
\begin{eqnarray}
\label{E:updownarrow}
z_\downarrow:=z_0, \quad z_\uparrow:=A(z_0).
\end{eqnarray}

For $z \in \DSF$, we define the $\textit{trajectory from $z$}$ as $
[z,z_\uparrow)_{eucl} \cup \bigcup_{n \in \mathbb{N}} [A^{(n)}(z_\uparrow),A^{(n+1)}(z_\uparrow))_{eucl}$, where $A^{(n)}:=A \circ... \circ A$ $n$ times.

\begin{definition}
For all $t \in \mathbb{R}$, we define the \emph{level $t$}, denoted by $\mathcal{L}_t$, as the set of abscissas of points in DSF with height $t$:
\begin{eqnarray*}
\mathcal{L}_t=\{x \in \mathbb{R}^d,(x,e^t) \in \DSF\}.
\end{eqnarray*}
\end{definition}

\begin{definition}
Let $t_1 \le t_2$, and let $x \in \mathcal{L}_{t_1}$. The trajectory from $(x,e^{t_1})$ crosses the level $t_2$ (the hyperplane $\mathbb{R}^d \times \{e^{t_2}\}$) at most at one point. It could \emph{a priori} never cross the level $t_2$, if the $y$-coordinate stays indefinitely below $e^{t_2}$. Thus we define $\mathcal{A}_{t_1}^{t_2}(x)$ as the point $x' \in \mathbb{R}^d$ such that $(x',e^{t_2})$ belongs to the trajectory from $(x,e^{t_1})$ and we set $\mathcal{A}_{t_1}^{t_2}(x)=\infty$ if this trajectory does not cross the level $t_2$. The point $\mathcal{A}_{t_1}^{t_2}(x)$ is called the $\emph{ancestor}$ of $x$ (or $(x,e^{t_1})$) at level $t_2$.
\end{definition}

\begin{center}
\begin{tikzpicture}[scale=0.9]
\draw [>=stealth,->] (0,0)--(6,0) node[right] {$\mathbb{R}^d$}; 
\draw [>=stealth,->] (0,0)--(0,4) node[left] {$\mathbb{R}$};
\draw [dashed] (0,1)--(6,1) node[right] {$e^{t_1}$};

\filldraw (0.2,1) circle (3pt); \draw [black,dashed](0.2,1)--(0.2,0);
\filldraw [gray] (0.2,0)circle (3pt);
\draw (0,0.5)--(0.4,1.5);

\filldraw (1.4,1) circle (3pt); \draw [black,dashed](1.4,1)--(1.4,0);
\filldraw [gray] (1.4,0) circle (3pt);
\draw (0.6,0.8)--(1.8,1.1);

\draw [black,dashed](2.5,1)--(2.5,0);
\filldraw [gray] (2.5,0) circle (3pt);

\filldraw (4.2,1) circle (3pt); \draw [black,dashed](4.2,1)--(4.2,0);
\filldraw [gray] (4.2,0) circle (3pt);
\draw (4.5,0.5)--(3.75,1.75);

\draw [gray] (0,0) node[left] {$\mathcal{L}_{t_1}$};

\draw [dashed] (0,3)--(6,3) node[right] {$e^{t_2}$};
\draw [red]  (2,0.9)--(3.5,1.2)--(2.4,2)--(3.8,2.3)--(1.9,2.6)--(2.9,2.8)--(2.3,3.1);
\filldraw [red] (2.5,1) circle (3pt) node[above] {$(x,e^{t_1})$}
;
\filldraw [red] (2.5,3) circle (3pt) node[red, above] {$(\mathcal{A}_{t_1}^{t_2}(x),e^{t_2})$};

\end{tikzpicture}
\end{center}

Actually, it will be shown later that the $y$-coordinates always goes to infinity a.s.:
\begin{proposition}
\label{Prop:welldefined}
Almost surely, for all $t_1 \le t_2$, for all $x \in \mathcal{L}_{t_1}$, $\mathcal{A}_{t_1}^{t_2}(x) \neq \infty$.
\end{proposition}
This statement is proved in Section \ref{S:proofwelldefined}.

\begin{definition}
Let $t_1 \le t_2$, and let $x \in \mathcal{L}_{t_2}$. We define the sets of \emph{descendants} of $x$, denoted by $\mathcal{D}_{t_1}^{t_2}(x)$, as the set of points $x' \in \mathcal{L}_{t_1}$ such that $(x,e^{t_2})$ belongs to the trajectory from $(x',e^{t_1})$: $\mathcal{D}_{t_1}^{t_2}(x)=\{x' \in \mathcal{L}_{t_1}, \mathcal{A}_{t_1}^{t_2}(x')=x\}$.
\end{definition}

\begin{center}
\begin{tikzpicture}[decoration={brace}]
\pgftext{\includegraphics[width=8cm]{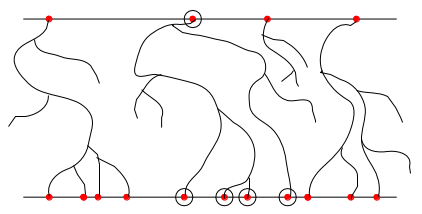}} at (0,0);
\draw (-1,2) node {$(x,e^{t_2})$};
\draw (3.9,1.8)  node {$e^{t_2}$};
\draw (3.9,-1.6) node {$e^{t_1}$};
\draw [decorate,line width=1.5pt] (1.7,-1.95) -- (-0.8,-1.95);
\draw (0.5,-2) node[below] {$\in \mathcal{D}_{t_1}^{t_2}(x) \times \{e^{t_1}\}$};
\end{tikzpicture}
\end{center}

\begin{definition}
\label{Def:coalesceheight}
Let $A \subset \mathbb{R}^d$ be measurable. We define the \emph{coalescing height of $A$}, denoted by $\tau_A$, as
\begin{eqnarray*}
\tau_A=\inf \{t \ge 0, \forall x,x' \in \mathcal{L}_0 \cap A,\  \mathcal{A}_0^t(x)=\mathcal{A}_0^t(x')\} \in \mathbb{R}_+ \cup \{\infty\}.
\end{eqnarray*}
It it the lowest height where all trajectories from points of $(\mathcal{L}_0 \cap A) \times \{e^0\}$ coalesce. 
\end{definition}

The following definition concerns bi-infinite branches, i.e. branches that are infinite in both directions.

\begin{definition}
\label{Def:biinfinite}
Let $f:\mathbb{R} \to \mathbb{R}^d$. We say that $f$ \emph{encodes a bi-infinite branch} if for all $t_1 \le t_2$, $f(t_2)=\mathcal{A}_{t_1}^{t_2}(f(t_1))$. In this case, the subset
$\{(f(t),e^t),~t \in \mathbb{R}\} \subset H$
is called a $\emph{bi-infinite branch}$ of the DSF.\\
We denote by $\Bi$ the random set of functions $f: \mathbb{R} \to \mathbb{R}^d$ that encode a bi-infinite branch.
\end{definition}

\subsection{Preliminary properties}
\label{S:properties}
We will exploit invariance by isometries of the model. The family of \emph{translations} $T^H_s:(x,y) \mapsto (x+s,y)$ for $s \in \mathbb{R}^d$ and the \emph{dilations} $D_\lambda:(x,y) \mapsto(\lambda x,\lambda y)$ are isometries of $(H,ds_H^2)$ ~\cite[p.79]{cannon}, thus they preserve the law of $N$. Moreover, these isometries fix the point $\infty$ (isometries of $(H,ds_H^2)$ are naturally extended to the set of points at infinity). Therefore they also preserve the horodistance function $\mathcal{H}_\infty$ modulo an additive constant, so they preserve the graph structure of the DSF in law.

In addition, these isometries preserve Euclidean segments. Then, they preserve the law of the random subset $\DSF$. A consequence of this translation invariance property is that, for all $t \in \mathbb{R}$, $\mathcal{L}_t$ is a stationary point process.

It will be required to have a control of moments for the number of points of $\mathcal{L}_t$ in a given compact set:

\begin{proposition}
\label{Prop:nop}
We have $\mathbb{E}[\#(\mathcal{L}_t \cap B_{\mathbb{R}^d}(0,R))^p]<\infty$ for all $t \in \mathbb{R}$ and $p,R \ge 0$.
\end{proposition}

We refer to the Appendix (Section \ref{S:nop}) for the proof.

\begin{corollary}
\label{Prop:finiteintensity}
For all $t \in \mathbb{R}$, $\mathcal{L}_t$ has finite intensity $\alpha_0e^{-dt}$, where $\alpha_0$ is the intensity of $\mathcal{L}_0$.
\end{corollary}

\begin{proof}[Proposition \ref{Prop:nop} implies Corollary \ref{Prop:finiteintensity}]
By Proposition $\ref{Prop:nop}$ with $p=1$, $\mathcal{L}_t$ has finite intensity for all $t \in \mathbb{R}$. Then we can define $\alpha_0$ as the intensity of $\mathcal{L}_0$. For $t \in \mathbb{R}$, the dilation $D_{e^t}$ preserves the DSF in distribution, so $\mathcal{L}_t \overset {d}{=} e^t\mathcal{L}_0$. Then $\mathcal{L}_t$ has finite intensity $\alpha_0e^{-dt}$ for any $t \in \mathbb{R}$.
\end{proof}

In the following, we will have to consider the law of DSF conditionally to $\{x \in \mathcal{L}_t\}$, for given $x \in \mathbb{R}^d$ and $t \in \mathbb{R}$. Thus we define the probability measure $\mathbb{P}_{x \in \mathcal{L}_t}[\cdot]$ on $\mathcal{N}_S$ as the Palm distribution of $N$ conditioned on $x \in \mathcal{L}_t$ (and let $\mathbb{E}_{x \in \mathcal{L}_t}[\cdot]$ its associated expectation). The definition of this probability measure follows the standard definition of Palm measures, however it should be a probability measure on all the point process (on $\mathcal{N}_S$) and not only on $\mathcal{L}_t$, that is why we need to re-define properly this probability distribution. 

\begin{prop_def}[Conditional distribution given $\{x \in \mathcal{L}_t\}$]
\label{PropDef:palm}
\!\
\begin{itemize}
\item (Definition) For $\Gamma \subset \mathcal{N}_S$ measurable, we define the measure $\mu_\Gamma$ on $\mathbb{R}^d$ by
\begin{eqnarray}\label{E:defcond1}\mu_\Gamma(A):=\mathbb{E}\left[ \sum_{s \in \mathcal{L}_t \cap A} \mathbf{1}_{T_{x-s}N \in \Gamma} \right]\end{eqnarray}
for all measurable set $A \subset \mathbb{R}^d$. Note that $\mu_\Gamma$ depends on $t$ and $x$. 
\item (Proposition) For all measurable set $\Gamma \subset \mathcal{N}_S$, the measure $\mu_\Gamma$ is invariant by translations and finite on compact sets.
\item (Definition) Then for all measurable set $\Gamma \subset \mathcal{N}_S$, $\mu_\Gamma$ is a multiple of the Lebesgue measure, so we can define:
\begin{eqnarray}\label{E:defcond2}\mathbb{P}_{x \in \mathcal{L}_t}[\Gamma]:=\frac{d\mu_\Gamma}{\alpha_0e^{-dt}d\Leb}.\end{eqnarray}
\item (Proposition) The map $\Gamma \mapsto \mathbb{P}_{x \in \mathcal{L}_t}[\Gamma]$ so defined is a probability measure on $\mathcal{N}_S$. We denote by $\mathbb{E}_{x \in \mathcal{L}_t}$ its associated expectation.
\end{itemize}
\end{prop_def}

We refer to the Supplementary materials (Section 2) for the proof of Proposition-Definition \ref{PropDef:palm}.



\begin{lemma}
\label{Lem:scaleinv}
(Invariance by dilations).
Let $t,t' \in \mathbb{R}$. We have 
\begin{eqnarray*}
\mathbb{P}_{0 \in \mathcal{L}_t}\left[D_{e^{t'-t}}(N) \in \cdot \right]=\mathbb{P}_{0 \in \mathcal{L}_{t'}}\left[ N \in \cdot \right]
\end{eqnarray*}
\end{lemma}

The proof of Lemma \ref{Lem:scaleinv} is also given in the Supplementary materials (Lemma 2 Section 2.1).

\section{The Mass Transport Principle and its consequences}
\label{S:tools}

In this section, we state a main ingredient of the proofs, the Mass Transport Principle and explore some consequences.

\subsection{The Mass Transport Principle}

This theorem is an adaptation of its version on the hyperbolic plane, which is due to Benjamini and Schramm \cite[p.13-14]{benjamini}.

\begin{definition}[Diagonally invariant measure]
\label{Def:diaginv}
Let $\pi$ be some measure on $\mathbb{R}^d \times \mathbb{R}^d$ for the Borel $\sigma$-algebra. We say that $\pi$ is \textit{diagonally invariant} if for all $x \in \mathbb{R}^d$,
\begin{eqnarray*}\pi(A \times B)=\pi(T_xA \times T_xB).\end{eqnarray*}
\end{definition}

\begin{theorem}[Mass Transport Principle]
\label{Thm:masstranport}
Let $\pi$ be some positive diagonally invariant measure on $\mathbb{R}^d \times \mathbb{R}^d$. Then for any measurable set $A \subset \mathbb{R}^d$ with nonempty interior, the following identity holds:
\begin{eqnarray*}\pi(A \times \mathbb{R}^d)=\pi(\mathbb{R}^d \times A),\end{eqnarray*}
these values can be eventually infinite.
\end{theorem}

A proof of Theorem \ref{Thm:masstranport} is given in the Supplementary materials (Section \ref{S:masstransport}). The intuition behind the Mass Transport Principle can be understood as follows. The measure $\pi$ describes a mass transport from $\mathbb{R}^d$ to $\mathbb{R}^d$, that is, $\pi(A \times B)$ corresponds to the amount of mass transported from $A$ to $B$. Then the Mass Transport Principle asserts that the outgoing mass equals the incoming mass.

In the literature, the study of percolation in hyperbolic space mostly concerns models that are invariant under any isometry of $\mathbb{H}^{d+1}$ (see for instance, the Poisson-Boolean model studied in \cite{tykesson} or the Poisson-Voronoï model studied in \cite{benjamini}). Thus it is relevant to use the Mass Transport Principle on $\mathbb{H}^{d+1}$ \cite{benjamini} to study these models. However, our model of DSF is directed, so it is only invariant under isometries that fix the target point. This group of isometries is not unimodular, so \tvc{their} version of the Mass Transport on $\mathbb{H}^{d+1}$ cannot be used for the study of the DSF. Instead of considering mass transports on $\mathbb{H}^{d+1}$, we typically consider mass transport from level $t_1$ to level $t_2$ (for $t_1,t_2 \in \mathbb{R}$), that is why we need the Mass Transport on $\mathbb{R}^d$.

We now state some consequences of the Mass Transport Principle, that play a central role in the control of horizontal fluctuations of trajectories (proofs of Theorems \ref{Thm:controlforward} and \ref{Thm:controlbackward} in Section \ref{S:prooffluctuations}). We first define the concepts of weight function and association function (Section \ref{S:weightassociationfunction}). From these objects, we construct diagonally invariant measures and obtain different equalities by identifying both sides of equality given in the Mass Transport Principle (Section \ref{S:results}). Proofs are given in Section \ref{S:proofresults}.

\subsection{Association functions and weight functions}
\label{S:weightassociationfunction}

Let us introduce a random variable $Y$ independent of $N$, valued in some measurable space $\Upsilon$. In a majority of applications, the extra random variable $Y$ will not be necessary. However, an extra random variable will be used in the proof of (ii) in Theorem \ref{Thm:controlforward}, because some association function using extra randomness will be considered.

\subsubsection{Association functions}

\begin{definition}[Association function, general case]
Let $t \in \mathbb{R}$. We call \emph{level $t$-association function} or more simply \emph{association function} a measurable function $f:\mathbb{R}^d \times \mathcal{N}_S \times \Upsilon \to \mathbb{R}^d$ such that
\begin{itemize}
\item $f$ is valued in $\mathcal{L}_t$, more precisely
\begin{eqnarray*}
\forall x \in \mathbb{R}^d,~f(x,N,Y) \in \mathcal{L}_t \text{ a.s.}
\end{eqnarray*}
\item $f$ is \emph{translation invariant}, in the following sense: for all $\eta \in \mathcal{N}_S$, for all $x,x' \in \mathbb{R}^d$,
\begin{eqnarray*}
f(x+x',T_{x'}\eta,Y) \overset{(d)}{=} f(x,\eta,Y)+x'.
\end{eqnarray*}
\end{itemize}
We set the notation $f(x):=f(x,N,Y)$. An association function can be seen as a (translation invariant) random function from $\mathbb{R}^d$ to $\mathbb{R}^d$.
\end{definition}
For most of applications, $f$ will not depend on $Y$ ($Y$ will be deterministic). This case will be refered as the \emph{non-marked case}. In this case, the notation $f(x,\eta,Y)$ will be replaced by $f(x,\eta)$ for simplicity.

\begin{definition}
\label{Def:cellofpoint}
[Cell of a point] Let $t \in \mathbb{R}$, and let $f$ be a level $t$-association function. For $x \in \mathcal{L}_t$, we define the \emph{cell} of $x$ as the (random) subset of $\mathbb{R}^d$:
\begin{eqnarray*}
\Lambda_f(x):=\{x' \in \mathbb{R}^d,~f(x')=x\}.
\end{eqnarray*}

\end{definition}
\begin{example}
\label{Ex:closest}
The most useful example to keep in mind is the following: $f(x,\eta)$ is defined as the point of $\mathcal{L}_0(\eta)$ the closest to $x$:
\begin{eqnarray*}
f(x,\eta):=\argmin_{x' \in \mathcal{L}_0(\eta)} \|x'-x\|. 
\end{eqnarray*}
Then $f$ is a level $0$-association function independent of $Y$ (the non-marked case). Moreover, for $x \in \mathcal{L}_0$, $\Lambda_f(x)$ is the Voronoï cell of $x$ associated to the point process $\mathcal{L}_0$.
\end{example}

\begin{example}
\label{Ex:poisson}
Suppose that $Y$ is a (homogeneous) PPP on $\mathbb{R}^d$ independent of $N$. Define $f(x,\eta,\xi)$ as the point of $\mathcal{L}_0(\eta)$ the closest to $x$ among all points $x' \in \mathcal{L}_0$ such that the ball $B(x',1)$ contains no points of $\xi$. Then $f$ is a level $0$-association function.

Another association function depending on a extra argument $Y$ will be constructed in Section \ref{S:proofcontrolforward2}.
\end{example}

\subsubsection{Weight functions}

\begin{definition}[Weight function]
We call \emph{weight function} a measurable function $w:\mathbb{R}^d \times \mathcal{N}_S \times \Upsilon \to \mathbb{R}_+$ that is \emph{translation invariant} in the following sense: for all $\eta \in \mathcal{N}_S$ and for all $x,x'\in \mathbb{R}^d$,
\begin{eqnarray*}
w(x,\eta,Y) \overset{(d)}{=} w(x+x',T_{x'}\eta,Y).
\end{eqnarray*}
We set the notation 
\begin{eqnarray*}
w(x):=w(x,N,Y).
\end{eqnarray*}
A weight function can be seen as a random application from $\mathbb{R}^d$ to $\mathbb{R}_+$.
\end{definition}

\tvc{The case where $w$ does not depend on $Y$ will be referred as the \emph{non-marked case}. In this case, we replace the notation $w(x,\eta,Y)$ by $w(x,\eta)$.}

\begin{example}
Consider the function $w(x,\eta):=\mathbf{1}_{x \in \mathcal{L}_0(\eta)}\|\mathcal{A}_0^1(x)(\eta)-x\|$. It is the horizontal deviation between levels $0$ and $1$ of the trajectory from $(x,e^0)$ when $x \in \mathcal{L}_0$. Then $w$ is a weight function in the non-marked case.
\end{example}

\begin{example}
Suppose that $Y$ is a random variable independent of $N$ and valued in $\mathbb{N}^*$. We can define $w(x,\eta,n)$ as the distance ($\| \cdot \|_2$ in $\mathbb{R}^d$) between $x$ and the point of $\mathcal{L}_0$ which is the $n$-th closest to $x$. Then $w$ is a weight function.
\end{example}

\subsubsection{Weighted association functions}

\begin{definition}
Let $t \in \mathbb{R}$. We call \emph{level $t$-weighted association function} (or more simply, \emph{weighted association function}) a couple $(f,w)$, where $f$ is an association function and $w$ is a weight function, such that the \emph{couple} $(f,w)$ is translation invariant in distribution, that is, for all $\eta \in \mathcal{N}_S$ and for all $x,x' \in \mathbb{R}^d$:
\begin{eqnarray}
\label{E:defweighteddiaginv}
\Big(f(x+x',T_{x'}\eta,Y),~w(x+x',T_{x'}\eta,Y) \Big) \overset{(d)}{=} \Big(f(x,\eta,Y)+x',~w(x,\eta,Y) \Big).
\end{eqnarray}
\end{definition}

Note that, in the non-marked case ($f$ and $w$ does not depend of $Y$), the condition (\ref{E:defweighteddiaginv}) is useless.

\begin{example}
Consider the association $f$ introduced in Example \ref{Ex:closest}: $f(x,\eta)$ is the point of $\mathcal{L}_0(\eta)$ the closest to $x$. Let $w(x,\eta):=\|x-f(x,\eta)\|$. Then $w$ is a weight function and $(f,w)$ is a weighted association function (in the non-marked case).
\end{example}

\begin{example}
Consider the association function $f$ introduced in Example \ref{Ex:poisson}. Then define $w(x,\eta,\xi):=\#(\xi \cap B(0,1))$. Then $w$ is a weight function, however the couple $(f,w)$ is not a weighted association function. \tvc{To see this, consider for instance the case $d=1$. The value of $w(x,\eta,\xi)$ depends here only on $\xi$ and is thus constant: it indicates the number of points of $\xi$ in the interval $(-1,1)$. Imagine that this number is zero. Then, the distance of $f(x,\eta,\xi)$ to $x$ is at most $|x|$, and we have loose the translation invariance in distribution.}
\end{example}

\subsection{Results derived from the Mass Transport Principle}
\label{S:results}

\tvc{We now enounce some results using the Mass Transport Principle that will be useful. The proofs are postponed to Section \ref{S:proofresults}.}

\tvc{Let us extend the Palm distribution $\mathbb{P}_{x \in \mathcal{L}_t}$ defined in Section \ref{S:properties} on $\sigma(N)$ to $\sigma(N,Y)$ by setting:
\begin{eqnarray*}
\mathbb{P}_{x \in \mathcal{L}_t}[N \in \Gamma,~Y \in \Gamma']=\mathbb{P}_{x \in \mathcal{L}_t}[N \in \Gamma]\mathbb{P}[Y \in \Gamma']
\end{eqnarray*}
for all $\Gamma \in \mathcal{N}_S$ and $\Gamma' \subset \Upsilon$. We also extend the notation $\mathbb{E}_{x \in \mathcal{L}_t}[X]$ to random variables $X$ measurable w.r.t. $\sigma(N,Y)$.}

\tvc{\begin{lemma}
\label{Lem:technicallemmageneral}
Let $x \in \mathbb{R}^d$ and $t \in \mathbb{R}$. Let $w:\mathbb{R}^d \times \mathcal{N}_S \times \Upsilon \rightarrow \mathbb{R}_+$ be a weight function. Then for all measurable set $A \subset \mathbb{R}^d$,
\begin{eqnarray}
\label{E:technicallemmageneral}
\mathbb{E}\left[ \sum_{s \in \mathcal{L}_t \cap A} w(s)\right]=\alpha_0 e^{-dt}\Leb(A)\mathbb{E}_{0 \in \mathcal{L}_t}[w(0)].
\end{eqnarray}
\end{lemma}
}
\begin{proposition}
\label{Lem:MTlemma1}
Let $t_1,t_2 \in \mathbb{R}$ with $t_1 \le t_2$. Let $w:\mathbb{R}^d \times \mathcal{N}_s \times \Upsilon \rightarrow \mathbb{R}_+$ be a weight function. Then
\begin{eqnarray}
\label{E:MTlemma1}
\mathbb{E}_{0 \in \mathcal{L}_{t_1}}\left[ w(0) \right]=e^{d(t_1-t_2)}\mathbb{E}_{0 \in \mathcal{L}_{t_2}}\left[ \sum_{x \in \mathcal{D}_{t_1}^{t_2}(0)} w(x) \right].
\end{eqnarray}
\end{proposition}

\begin{corollary}[Expected number of descendants]
\label{Cor:expnumdesc}
Let $t_1,t_2 \in \mathbb{R}$ with $t_1 \le t_2$. We have
\begin{eqnarray}
\label{E:expnumdesc}
\mathbb{E}_{0 \in \mathcal{L}_{t_2}}[\#\mathcal{D}_{t_1}^{t_2}(0)]=e^{d(t_2-t_1)}.
\end{eqnarray}
In particular, for all $x \in \mathcal{L}_{t_2}$, $\#\mathcal{D}_{t_1}^{t_2}(x)<\infty$ almost surely.
\end{corollary}

\begin{proof}[Proof of Corollary \ref{Cor:expnumdesc} knowing Proposition \ref{Lem:MTlemma1}]
Applying Proposition \ref{Lem:MTlemma1} to $w \equiv 1$ leads to (\ref{E:expnumdesc}).

Thus, we obtain that $\mathbb{P}_{0 \in \mathcal{L}_{t_2}}[\#\mathcal{D}_{t_1}^{t_2}(0)<\infty]=1$. Then we apply Lemma \ref{Lem:technicallemmageneral} with $A=\mathbb{R}^d$ to the weight function $w'$ defined as
\begin{eqnarray*}
w'(x,\eta)=\mathbf{1}_{x \in \mathcal{L}_{t_2}(\eta),\#\mathcal{D}_{t_1}^{t_2}(x)(\eta)=\infty}.
\end{eqnarray*}
It leads to:
\begin{eqnarray*}
&\mathbb{E}\left[ \#\{x \in \mathcal{L}_{t_2},~\#\mathcal{D}_{t_1}^{t_2}(x)=\infty \} \right]&=\mathbb{E}\left[ \sum_{x \in \mathcal{L}_{t_2}} w'(x) \right]=\infty\times \mathbb{E}_{0 \in \mathcal{L}_{t_2}}[w'(0)] \nonumber\\
&&=\infty\times \mathbb{P}_{0 \in \mathcal{L}_{t_2}}[\#\mathcal{D}_{t_1}^{t_2}(0)=\infty]=0.
\end{eqnarray*}
Thus, for all $x \in \mathcal{L}_{t_2}$, $\#\mathcal{D}_{t_1}^{t_2}(x)<\infty$.
\end{proof}

\begin{proposition}
\label{Lem:MTlemma2}
Let $t \in \mathbb{R}$ and let $(f,w)$ be a level $t$-weighted association function. Then
\begin{eqnarray}
\label{E:MTlemma2}
\mathbb{E}\left[ w(0) \right]=\alpha_0e^{-dt}\mathbb{E}_{0 \in \mathcal{L}_t} \left[ \int_{\Lambda_f(0)} w(x)~dx \right].
\end{eqnarray}
\end{proposition}

\begin{corollary}[Expected volume of a typical cell]
\label{Cor:expectedvolcell}
Let $t \in \mathbb{R}$, $f$ be a level $t$-association function. Applying Proposition \ref{Lem:MTlemma2} with $w \equiv 1$ (it is easy to check that $(f,w)$ is a weighted association function), we obtain:
\begin{eqnarray*}
\mathbb{E}_{0 \in \mathcal{L}_t}[\Leb(\Lambda_f(0))]=\alpha_0^{-1}e^{dt}.
\end{eqnarray*}
\end{corollary}

\begin{proposition}
\label{Lem:technical}
Let $t \in \mathbb{R}$ and $p \ge 1$. Let $f$ be a level $t$-association function. Then
\begin{eqnarray}
\label{E:technical}
\mathbb{E}_{0 \in \mathcal{L}_t}[\Leb(\Lambda_f(0))^{1+p/d}] \le C_{p,d}e^{dt}\mathbb{E}[\|f(0)\|^p],
\end{eqnarray}
where $C_{p,d}$ is a positive constant that only depends on $p$ and $d$.
\end{proposition}

\subsection{Proofs}
\label{S:proofresults}

We first prove Lemma \ref{Lem:technicallemmageneral}.

\begin{proof}[Proof of Lemma \ref{Lem:technicallemmageneral}]
\tvc{We first prove the non-marked case. Let $w:\mathbb{R}^d \times \mathcal{N}_S \to \mathbb{R}_+$ be a weight function in the non-marked case.

For $\eta \in \mathcal{N}_S$, define $g(\eta)=w(x,\eta).$ By translation invariance, for all $r \in \mathbb{R}^d$, $\eta \in \mathcal{N}_S$, $w(r,\eta)=w(x,T_{x-r}\eta)=g(T_{x-r}\eta)$. In particular $w$ is entirely determined by $g$.

Let $\Gamma \subset \mathcal{N}_S$ be measurable. For all measurable set $A \subset \mathbb{R}^d$,
\begin{eqnarray*}
\mathbb{E}\left[ \sum_{s \in \mathcal{L}_t \cap A} w(s,N) \right]=\mathbb{E}\left[ \sum_{s \in \mathcal{L}_t \cap A} w(x,T_{x-s}N) \right]=\mathbb{E}\left[\sum_{s \in \mathcal{L}_t \cap A} g(T_{x-s}N) \right]
\end{eqnarray*}
and
\begin{eqnarray*}
\mathbb{E}_{x \in \mathcal{L}_t}[w(x,N)]=\mathbb{E}_{x \in \mathcal{L}_t}[g(N)].
\end{eqnarray*}
Thus it suffices to prove that the identity
\begin{eqnarray}
\label{E:equa0}
\mathbb{E}\left[\sum_{s \in \mathcal{L}_t \cap A} g(T_{x-s}N) \right]=\alpha_0 e^{-dt}\Leb(A)\mathbb{E}_{x \in \mathcal{L}_t}[g(N)]
\end{eqnarray}
holds for all measurable functions $g:\mathcal{N}_S \rightarrow \mathbb{R}_+$ and all measurable set $A \subset \mathbb{R}$. Let $\Gamma \subset \mathcal{N}_S$ and $A \subset \mathbb{R}^d$ be measurable. We show (\ref{E:equa0}) for $g=\mathbf{1}_\Gamma$:
\begin{eqnarray*}
&\mathbb{E}\left[\sum_{s\in \mathcal{L}_t \cap A} g(T_{x-s}N) \right]&=\mathbb{E}\left[\sum_{s \in \mathcal{L}_t \cap A} \mathbf{1}_{T_{x-s}N \in \Gamma} \right]=\alpha_0 e^{-dt}\Leb(A)\mathbb{P}_{x \in \mathcal{L}_t}[\Gamma] \text{ by (\ref{E:defcond2})}\nonumber\\
&&=\alpha_0 e^{-dt}\Leb(A)\mathbb{E}_{x \in \mathcal{L}_t}[\mathbf{1}_{N \in \Gamma}] =\mathbb{E}_{x \in \mathcal{L}_t}[g(N)],
\end{eqnarray*}
so (\ref{E:equa0}) holds for $g=\mathbf{1}_\Gamma$. Since both sides of equality (\ref{E:equa0}) are linear in $g$, (\ref{E:equa0}) holds for all step functions. Now we pass to the limit to obtain (\ref{E:equa0}) for all measurable function $g$. Let $g:\mathcal{N}_S \rightarrow \mathbb{R}_+$ be measurable, and consider a non-decreasing sequence $(g_n)_{n \in \mathbb{N}}$ of step functions that converges to $g$. By monotone convergence theorem,
\begin{eqnarray*}
\lim_{n \rightarrow \infty} \uparrow \sum_{s \in \mathcal{L}_t \cap A}g_n(T_{x-s}N)=\sum_{s \in \mathcal{L}_t \cap A}g(T_{x-s}N) \text{ a.s.}
\end{eqnarray*}
Then by monotone convergence theorem,
\begin{eqnarray}
\label{E:cv1}
\mathbb{E}\left[ \sum_{s \in \mathcal{L}_t \cap A} g_n(T_{x-s}N) \right] \underset{n \rightarrow \infty}{\longrightarrow}\mathbb{E}\left[ \sum_{s \in \mathcal{L}_t \cap A} g(T_{x-s}N) \right].
\end{eqnarray}
On the other hand, again by monotone convergence,
\begin{eqnarray}
\label{E:cv2}
\mathbb{E}_{x \in \mathcal{L}_t}[g_n(N)] \underset{n \rightarrow \infty}{\longrightarrow}\mathbb{E}_{x \in \mathcal{L}_t}[g(N)].
\end{eqnarray}
By (\ref{E:cv1}) and (\ref{E:cv2}) we obtain (\ref{E:equa0}) for $g$ by passing to the limit.
}
We wove on to show the general case. Let us denote by $\mathbb{P}_\lambda$ the distribution of $N$ (probability measure on $\mathcal{N}_S$) and by $\mathbb{Q}$ the distribution of $Y$ (probability measure on $\Upsilon$). Let $w:\mathbb{R}^d \times \mathcal{N}_S \times \Upsilon \rightarrow \mathbb{R}_+$ be a weight function in the general case. Define $\tilde{w}(x,\eta):=\int_\xi w(x,\eta,\xi) ~\mathbb{Q}(d\xi)$ for all $x \in \mathbb{R}^d$, $\eta \in \mathcal{N}_S$. Then $\tilde{w}$ is a weight function in the non-marked case, so by the non-marked case applies to $\tilde{w}$:
\begin{eqnarray}
\label{E:lemmageq1}
\mathbb{E}\left[ \sum_{s \in \mathcal{L}_t \cap A} \tilde{w}(s)\right]=\alpha_0 e^{-dt}\Leb(A)\mathbb{E}_{x \in \mathcal{L}_t}[\tilde{w}(x)].
\end{eqnarray}
For all $\eta \in \mathcal{N}_S$,
\begin{eqnarray*}
\underset{s \in \mathcal{L}_t(\eta) \cap A}{\sum} \tilde{w}(s,\eta)=\underset{s \in \mathcal{L}_t(\eta) \cap A}{\sum} \int_\xi w(s,\eta,\xi) ~\mathbb{Q}(d\xi)=\int_{\xi} \underset{s \in \mathcal{L}_t(\eta) \cap A}{\sum} w(s,\eta,\xi) ~\mathbb{Q}(d\xi).
\end{eqnarray*}
Therefore
\begin{eqnarray}
\label{E:lemmageq2}
&\mathbb{E}\left[ \sum_{s \in \mathcal{L}_t \cap A}\tilde{w}(s) \right]&=\int_{\eta} \sum_{x \in \mathcal{L}_t(\eta) \cap A} \tilde{w}(s,\eta) ~d\mathbb{P}_\lambda(\eta)=\int_{\eta,\xi} \sum_{s \in \mathcal{L}_t(\eta) \cap A} w(s,\eta,\xi) ~\mathbb{P}_\lambda(d\eta) \otimes \mathbb{Q}(d\xi)\nonumber\\
&&=\mathbb{E}\left[ \sum_{s \in \mathcal{L}_t(\eta) \cap A} w(s) \right]
\end{eqnarray}
since $N$ and $Y$ are independent. On the other hand,
\begin{eqnarray}
\label{E:lemmageq3}
\mathbb{E}_{x \in \mathcal{L}_t}[\tilde{w}(x)]=\int_{\eta,\xi} w(x,\eta,\xi) ~\mathbb{P}_{x \in \mathcal{L}_t}(d\eta) \otimes \mathbb{Q}(d\xi)=\mathbb{E}_{x \in \mathcal{L}_t}[w(x)],
\end{eqnarray}
by definition of $\mathbb{E}_{x \in \mathcal{L}_t}$. Finally, the conclusion is obtained by combining (\ref{E:lemmageq1}), (\ref{E:lemmageq2}) and (\ref{E:lemmageq3}).
\end{proof}

The proofs of Propositions \ref{Lem:MTlemma1} and \ref{Lem:MTlemma2} are based on the Mass Transport Principle. 

\begin{proof}[Proof of Proposition \ref{Lem:MTlemma1}]
Let us define the following measure $\pi$ on $\mathbb{R}^d \times \mathbb{R}^d$:
\begin{eqnarray*}
\pi(E)=\mathbb{E}\left[ \sum_{x \in \mathcal{L}_{t_1}} \mathbf{1}_{\big(x,~\mathcal{A}_{t_1}^{t_2}(x)\big) \in E} w(x)\right]
\end{eqnarray*}
for all measurable set $E \subset \mathbb{R}^d \times \mathbb{R}^d$. \tvc{This measure $\pi$ may be interpreted as the following mass transport: from all point $x \in \mathcal{L}_{t_1}$, we transport a mass $w(x)$ to its ancestor $\mathcal{A}_{t_1}^{t_2}(x)$.} The diagonally invariance of $\pi$ easily follows from the translation invariance property of the model, we refer to \cite[Section 5.4]{version_longue} for the details. By the Mass Transport Principle (Theorem \ref{Thm:masstranport}), for any measurable set $A \subset \mathbb{R}^d$ with non-empty interior, $\pi(A \times \mathbb{R}^d)=\pi(\mathbb{R}^d \times A)$. On the one hand,
\begin{eqnarray}
\label{E:MTlem1eq1}
\pi(A \times \mathbb{R}^d)=\mathbb{E}\left[ \sum_{x \in \mathcal{L}_{t_1}} \mathbf{1}_{x \in A} w(x) \right]=\mathbb{E}\left[ \sum_{x \in \mathcal{L}_{t_1} \cap A} w(x)\right]=\alpha_0e^{-dt_1}\Leb(A)\mathbb{E}_{0 \in \mathcal{L}_{t_1}} \left[ w(0) \right],
\end{eqnarray}
where Lemma \ref{Lem:technicallemmageneral} has been applied to $w$ with $t=t_1$ and $x=0$. On the other hand,
\begin{eqnarray}
\label{E:MTlem1eqoh}
&\pi(\mathbb{R}^d \times A)&=\mathbb{E}\left[ \sum_{x \in \mathcal{L}_{t_1}} \mathbf{1}_{\mathcal{A}_{t_1}^{t_2}(x) \in A} w(x) \right]=\mathbb{E}\left[ \sum_{\substack{x \in \mathcal{L}_{t_1},\\x' \in \mathcal{L}_{t_2} \cap A}} \mathbf{1}_{\mathcal{A}_{t_1}^{t_2}(x)=x'} w(x) \right] \nonumber\\
&&=\mathbb{E}\left[ \sum_{x' \in \mathcal{L}_{t_2} \cap A} \sum_{x \in \mathcal{D}_{t_1}^{t_2}(x')} w(x) \right].
\end{eqnarray}
Consider the level $t_2$-weight function
\begin{eqnarray*}
h(x',\eta,\xi)=\mathbf{1}_{x' \in \mathcal{L}_{t_2}}\sum_{x \in \mathcal{D}_{t_1}^{t_2}(x')(\eta)} w(x,\eta,\xi).
\end{eqnarray*}

The fact that $h$ is a weight function directly follows from translation invariance, precise computations are done in \cite[Section 5.4]{version_longue}. Lemma \ref{Lem:technicallemmageneral} applied to $h$ with $t=t_2$ and $x=0$ gives,
\begin{eqnarray}
\label{E:MTlem1eq2}
\pi(\mathbb{R}^d \times A) \overset{(\ref{E:MTlem1eqoh})}{=}\mathbb{E}\left[ \sum_{x' \in \mathcal{L}_{t_2} \cap A} \sum_{x \in \mathcal{D}_{t_1}^{t_2}(x')} w(x) \right]=\alpha_0e^{-dt_2}\Leb(A)\mathbb{E}_{0 \in \mathcal{L}_{t_2}}\left[ \sum_{x \in \mathcal{D}_{t_1}^{t_2}(0)} w(x) \right].
\end{eqnarray}
By combining (\ref{E:MTlem1eq1}), (\ref{E:MTlem1eq2}) and the Mass Transport Principle with some open set $A \subset \mathbb{R}^d$ verifying $\Leb(A)<\infty$, we obtain (\ref{E:MTlemma1}).
\end{proof}

\bigbreak

\begin{proof}[Proof of Proposition \ref{Lem:MTlemma2}]
Let us consider the measure on $\mathbb{R}^d \times \mathbb{R}^d$ defined as
\begin{eqnarray*}
\pi(E):=\mathbb{E}\left[ \int_{\mathbb{R}^d} \mathbf{1}_{(x,f(x)) \in E}~w(x)~dx \right]
\end{eqnarray*}
for all $E \in \mathbb{R}^d \times \mathbb{R}^d$.  \tvc{This measure $\pi$ may be interpreted as the following Mass Transport: for each point $x \in \mathbb{R}^d$, we transport a mass $w(x,f(x))~dx$ from $x$ to $f(x)$.} The measure $E$ is diagonally invariant, thus the Mass Transport Principle applies. On the one hand:
\begin{eqnarray}
\label{E:mtlem2e1}
\pi(A \times \mathbb{R}^d)=\mathbb{E}\left[ \int_A w(x)~dx \right]=\int_A \mathbb{E}\left[ w(x) \right]~dx=\int_A \mathbb{E}\left[w(0)\right]~dx=\Leb(A)\mathbb{E}[w(0)].
\end{eqnarray}
where the translation invariance of $(f,w)$ was used in the third equality. Indeed, $w(x,N,Y) \overset{(d)}{=} w(0,T_{-x}N,Y) \overset{(d)}{=} w(0,N,Y)$ so $\mathbb{E}\left[ w(x) \right]=\mathbb{E}\left[ w(0) \right]$ for all $x \in \mathbb{R}^d$. On the other hand, since $(f,w)$ is a level $t$-weighted association function,
\begin{eqnarray}
\label{E:mtlem2e2}
&\pi(\mathbb{R}^d \times A)&=\mathbb{E}\left[ \int_{\mathbb{R}^d} \mathbf{1}_{f(x) \in A} ~w(x)~dx \right] =\mathbb{E}\left[ \int_{\mathbb{R}^d} \sum_{x' \in \mathcal{L}_t \cap A} \mathbf{1}_{f(x)=x'} ~w(x)~dx \right] \nonumber\\
&&=\mathbb{E}\left[ \sum_{x' \in \mathcal{L}_t \cap A} \int_{\mathbb{R}^d} \mathbf{1}_{f(x)=x'} ~w(x)~dx \right] =\mathbb{E}\left[ \sum_{x' \in \mathcal{L}_t \cap A} \int_{\Lambda_f(x')} w(x)~dx \right].
\end{eqnarray}
Let $h:\mathbb{R}^d \times \mathcal{N}_S \times \Upsilon \rightarrow \mathbb{R}_+$ be defined as $h(x',\eta,\xi)=\int_{\Lambda_f(x')(\eta)} w(x,\eta,\xi)~dx$. This function $h$ is a weight function (details in \cite[Section 5.4]{version_longue}), so, by Lemma \ref{Lem:technicallemmageneral} applied to $h$ with $x=0$,
\begin{eqnarray}
\label{E:mtlem2e3}
\mathbb{E}\left[ \sum_{x' \in \mathcal{L}_t \cap A} \int_{\Lambda_f(x')} w(x)~dx \right]=\alpha_0e^{-dt}\Leb(A)\mathbb{E}_{0 \in \mathcal{L}_t} \left[ \int_{\Lambda_f(0)} w(x)~dx \right].
\end{eqnarray}
Thus, by (\ref{E:mtlem2e2}) and (\ref{E:mtlem2e3}), we obtain
\begin{eqnarray}
\label{E:mtlem2e4}
\pi(\mathbb{R}^d \times A)=\alpha_0e^{-dt}\Leb(A)\mathbb{E}_{0 \in \mathcal{L}_t} \left[ \int_{\Lambda_f(0)} w(x)~dx \right].
\end{eqnarray}
Finally, we obtain (\ref{E:MTlemma2}) by combining (\ref{E:mtlem2e1}), (\ref{E:mtlem2e4}) and the Mass Transport Principle for some open set $A \subset \mathbb{R}^d$ verifying $\Leb(A)<\infty$.
\end{proof}

\bigbreak

\begin{proof}[Proof of Proposition \ref{Lem:technical}]
Let us consider the function $w$ defined as $w(x,\eta,\xi)=\|f(x,\eta,\xi)-x\|^p$ for $x \in \mathbb{R}^d$, $\eta \in \mathcal{N}_S$ and $\xi \in \Upsilon$. It follows from translation invariance that the couple $(f,w)$ is a level $t$-weighted association function (details in \cite[Section 5.4]{version_longue}). Proposition \ref{Lem:MTlemma2} applied to $(f,w)$ gives,
\begin{eqnarray}
\label{E:MTlem3eq1}
\mathbb{E}\left[ \|f(0)\|^p \right]=\alpha_0e^{-dt}\mathbb{E}_{0 \in \mathcal{L}_t} \left[ \int_{\Lambda_f(0)} \|f(x)-x\|^p ~dx \right]=\alpha_0e^{-dt}\mathbb{E}_{0 \in \mathcal{L}_t} \left[ \int_{\Lambda_f(0)} \|x\|^p ~dx \right]
\end{eqnarray}
because $f(x)=0$ for all $x \in \Lambda_f(0)$. Suppose for the moment that the following inequality holds $\mathbb{P}_{0 \in \mathcal{L}_t}$-almost surely:
\begin{eqnarray}
\label{E:MTlem3eq2}
\Leb(\Lambda_f(0))^{1+p/d} \le \tilde{C}_{p,d}\int_{\Lambda_f(0)} \|x\|^p ~dx,
\end{eqnarray}
where $\tilde{C}_{p,d}$ is a constant that only depends on $p$ and $d$. Then
\begin{eqnarray*}
\mathbb{E}_{0 \in \mathcal{L}_t}[\Leb(\Lambda_f(0))^{1+p/d}] \overset{(\ref{E:MTlem3eq2})}{\le}\tilde{C}_{p,d}\mathbb{E}_{0 \in \mathcal{L}_t}\left[ \int_{\Lambda_f(0)} \|x\|^p ~dx \right]\overset{(\ref{E:MTlem3eq1})}{=}\frac{\tilde{C}_{p,d}e^{dt}}{\alpha_0}\mathbb{E}\left[ \|f(0)\|^p \right],
\end{eqnarray*} 
so (\ref{E:technical}) holds for $C_{p,d}=\tilde{C}_{p,d}/\alpha_0$. It remains to show that (\ref{E:MTlem3eq2}) holds $\mathbb{P}_{0 \in \mathcal{L}_t}$-almost surely. For $r \ge 0$ we denote by $B_r:=\{x \in \mathbb{R}^d,~\|x\|<r\}$ the Euclidean ball of radius $r$ centered at the origin, and we denote by $\vartheta(d):=\Leb(B_1)$ the volume of the unit ball in $\mathbb{R}^d$. We rewrite $\|x\|^p$ as $\int_0^\infty \mathbf{1}_{r \le \|x\|} pr^{p-1}dr$. On the event $\{0 \in \mathcal{L}_t\}$, \tvc{Fubini's Theorem gives,}
\begin{eqnarray*}
&\int_{\Lambda_f(0)}\|x\|^p~dx&=\int_{\mathbb{R}^d} \int_0^\infty \mathbf{1}_{x \in \Lambda_f(0)}\mathbf{1}_{r \le \|x\|} ~pr^{p-1}~dr~dx=\int_0^\infty \int_{\mathbb{R}^d} \mathbf{1}_{x \in \Lambda_f(0) \backslash B_r} ~pr^{p-1}~dx~dr \nonumber\\
&&=\int_0^\infty pr^{p-1} \Leb(\Lambda_f(0) \backslash B_r)~dr \ge \int_0^{\left(\frac{\Leb(\Lambda_f(0))}{\vartheta(d)}\right)^{1/d}} pr^{p-1} (\Leb(\Lambda_f(0))-\Leb(B_r))~dr \nonumber\\
&&=\int_0^{\left(\frac{\Leb(\Lambda_f(0))}{\vartheta(d)}\right)^{1/d}}pr^{p-1}(\Leb(\Lambda_f(0))-\vartheta(d) r^d)~dr \nonumber\\
&&=\left[ \Leb(\Lambda_f(0))r^p-\frac{\vartheta(d) p}{p+d}r^{p+d} \right]_{r=0}^{\left(\frac{\Leb(\Lambda_f(0))}{\vartheta(d)}\right)^{1/d}}=\left( \frac{d}{p+d}\vartheta(d)^{-p/d} \right)\Leb(\Lambda_f(0))^{1+p/d}.
\end{eqnarray*}
Therefore (\ref{E:MTlem3eq2}) holds for $\tilde{C}_{p,d}=(1+p/d)\vartheta(d)^{p/d}$. This completes the proof.
\end{proof}

\section{Controlling fluctuations of trajectories}
\label{S:prooffluctuations}

In order to show the main results (Theorems \ref{Thm:coalescence} and \ref{Thm:biinfinite}), the key point of the proofs is to upper-bound horizontal fluctuations of trajectories.

\subsection{Cumulative Forward Deviation and Maximal Backward Deviation}
We first define the Cumulative Forward Deviation ($\CFD$) and Maximal Backward Deviation ($\MBD$) that measure horizontal deviations, then we state the results concerning $\CFD$ and $\MBD$.

\begin{definition}[Cumulative Forward Deviation]
\label{Def:CFD}
Let $t_1 \le t_2$. For $x \in \mathcal{L}_{t_1}$, we define the \emph{Cumulative Forward Deviation} for $x$ from level $t_1$ to level $t_2$, denoted by $\CFD_{t_1}^{t_2}(x)$, as
\begin{eqnarray*}
\CFD_{t_1}^{t_2}(x)=\int_{t_1}^{t_2} \left\|\frac{\partial}{\partial s}\mathcal{A}_{t_1}^s(x)\right\|~ds.
\end{eqnarray*}
\end{definition}

The quantity $\CFD_{t_1}^{t_2}(x)$ can be considered as the cumulative horizontal deviations (i.e. projected on $\mathbb{R}^d$) of the trajectory starting from $(x,e^{t_1})$ between level $t_1$ and level $t_2$. If $\mathcal{A}_{t_1}^s(x)=\infty$ for some $s \in [t_1,t_2]$, we set $\CFD_{t_1}^{t_2}(x)=\infty$.

We give an equivalent definition of the quantity $\CFD_{t_1}^{t_2}(x)$. \tvc{We suppose that $\mathcal{A}_{t_1}^{t_2}(x)<\infty$ (otherwise $\CFD_{t_1}^{t_2}(x)=\infty$).} Let us define the points $
z_{start}=(x,e^{t_1})_\downarrow$ and $z_{stop}=(\mathcal{A}_{t_1}^{t_2}(x),e^{t_2})_\downarrow$, where the notations $\downarrow$ and $\uparrow$ have been introduced in \eqref{E:updownarrow}.
Thus $z_{stop}$ is on the trajectory from $z_{start}$, let $n \in \mathbb{N}$ such that $z_{stop}=A^{(n)}(z_{start})$. Let us introduce $(x_0,e^{u_0})=z_0=z_{start}$, $ (x_i,e^{u_i})=z_i=A^{(i)}(z_{start})$ for  $i \in \llbracket 1,n \rrbracket$. In particular, $z_n=z_{stop}$. Then:

\begin{proposition}[Alternative writing of $\CFD$]
\begin{eqnarray*}
\CFD_{t_1}^{t_2}(x)=\left\{
\begin{aligned}[l|l]
&\|\mathcal{A}_{t_1}^{t_2}(x)-x\| \text{ if } n=0,\\
&\|x_1-x\|+\sum_{i=1}^{n-1} \|x_{i+1}-x_i \|+\|\mathcal{A}_{t_1}^{t_2}(x)-x_n\| \text{ if }n \ge 1.
\end{aligned}
\right.
\end{eqnarray*}
\end{proposition}

\begin{proof}
If $n=0$ then $(x,e^{t_1})$ and $\big(\mathcal{A}_{t_1}^{t_2}(x),e^{t_2}\big)$ belong to the same edge, so the function $\frac{\partial}{\partial s} \mathcal{A}_{t_1}^s$ has constant direction. Then
\begin{eqnarray*}
\CFD_{t_1}^{t_2}(x)=\int_{t_1}^{t_2} \left\|\frac{\partial}{\partial s} \mathcal{A}_{t_1}^s(x)\right\|~ds=\left\|\int_{t_1}^{t_2} \frac{\partial}{\partial s} \mathcal{A}_{t_1}^s(x)~ds\right\| =\|\mathcal{A}_{t_1}^{t_2}(x)-x\|.
\end{eqnarray*}
If $n \ge 1$, then
\begin{eqnarray*}
&\CFD_{t_1}^{t_2}(x)&=\int_{t_1}^{u_1} \left\|\frac{\partial}{\partial s} \mathcal{A}_{t_1}^s(x)\right\|~ds+\sum_{i=1}^{n-1} \int_{u_i}^{u_{i+1}} \left\|\frac{\partial}{\partial s} \mathcal{A}_{t_1}^s(x)\right\|~ds+\int_{u_n}^{t_2}\left\|\frac{\partial}{\partial s} \mathcal{A}_{t_1}^s(x)\right\|~ds \nonumber\\
&&=\left\|\int_{t_1}^{u_1}\frac{\partial}{\partial s} \mathcal{A}_{t_1}^s(x)~ds\right\|+\sum_{i=1}^{n-1} \left\|\int_{u_i}^{u_{i+1}} \frac{\partial}{\partial s} \mathcal{A}_{t_1}^s(x)~ds\right\|+\left\|\int_{u_n}^{t_2}\frac{\partial}{\partial s} \mathcal{A}_{t_1}^s(x)~ds\right\| \nonumber\\
&&=\|x_1-x\|+\sum_{i=1}^{n-1} \|x_{i+1}-x_i \|+\|\mathcal{A}_{t_1}^{t_2}(x)-x_n\|.
\end{eqnarray*}
For the second equality, we used the fact that, for each term of the sum, the function $s \to \frac{\partial}{\partial s}\mathcal{A}_{t_1}^s$ has constant direction in the corresponding integration interval. 
\end{proof}

Note that $\CFD$ upperbounds the horizontal deviations, in the following sense: for all $t_1 \le t_2$ and $x \in \mathcal{L}_{t_1}$, 
\begin{eqnarray*}
\|\mathcal{A}_{t_1}^{t_2}(x)-x\| =\left\| \int_{t_1}^{t_2} \frac{\partial}{\partial s} \mathcal{A}_{t_1}^s(x)~ds\right\| \le \int_{t_1}^{t_2} \left\| \frac{\partial}{\partial s} \mathcal{A}_{t_1}^s(x) \right\|~ds=\CFD_{t_1}^{t_2}(x).
\end{eqnarray*}

\begin{definition}[Maximal Backward Deviation]
\label{Def:defofcfd}
We define the \emph{Maximal Backward Deviation} of $x$ from level $t_1$ to level $t_2$, denoted by $\MBD_{t_1}^{t_2}(x)$, as

\begin{eqnarray*}\MBD_{t_1}^{t_2}(x)=
\left\{
\begin{aligned}
&\max_{x' \in \mathcal{D}_{t_1}^{t_2}(x)} \CFD_{t_1}^{t_2}(x') \text{ if }\mathcal{D}_{t_1}^{t_2}(x) \neq \emptyset \\
&0 \text{ otherwise.}
\end{aligned}
\right.\end{eqnarray*}
This is the maximal cumulative horizontal deviations among all trajectories between levels $t_1$ and $t_2$ and ending at $(x,e^{t_2})$ .
\end{definition}

The following theorem controls the cumulative forward deviation (\CFD).

\begin{theorem}(Forward fluctuations control.)
\label{Thm:controlforward}
Let $p \ge 1$. \\
(i) Let $Y$ be a random variable independant of $N$. Let $X_0$ be a random point of $\mathcal{L}_0$ (i.e. a random point of $\mathbb{R}^d$ such that $X_0 \in \mathcal{L}_0$ a.s.), measurable w.r.t. $\sigma(N,Y)$, and such that $\mathbb{E}\left[\|X_0\|^p\right]<\infty$. Then there exists a constant $K>0$ that only depends on $p,d$ and $\lambda$ (but not on the distribution of $X_0$) such that:
\begin{eqnarray*}
\limsup_{t \to \infty}\mathbb{E}\left[\left(e^{-t}\CFD_0^t(X_0)\right)^p \right] \le K.
\end{eqnarray*}
(ii) We have
\begin{eqnarray}
\label{E:conclcontrolbackwardii}
\limsup_{t \to \infty} \mathbb{E}_{0 \in \mathcal{L}_0} \left[ \left( e^{-t}\CFD_0^t(0) \right)^p \right]<\infty.
\end{eqnarray}
\end{theorem}

The intuition behind Theorem \ref{Thm:controlforward} is the following. Let us consider a typical trajectory. Because of the hyperbolic metric on $(H,ds_H^2)$, the horizontal fluctuations increase as the trajectory goes up. \tvc{More precisely, the fluctuations around height $h$ (say between $h$ and $h+1$, i.e. between the ordinates $e^h$ and $e^{h+1}$) are to the order of $e^h$. Then the forward cumulative deviation $\CFD_0^h(X_0)$ is almost determined by the last steps of the trajectory, and it is of order $e^h$. This behavior is typical to hyperbolic geometry. In Euclidean geometry, the fluctuations around height $h$ (say between $h$ and $h+1$) are to the order of $1$ for all $h$.}

The following theorem controls the backward maximal deviation (\MBD).

\begin{theorem}(Backward fluctuations control.)
\label{Thm:controlbackward}
For all $p \ge 1$,
\begin{eqnarray}
\label{E:controlbackward}
\limsup_{h \to \infty} \mathbb{E}_{0 \in \mathcal{L}_0} \left[ \MBD_{-h}^0(0)^p \right]<\infty.
\end{eqnarray}
\end{theorem}

The intuition behind Theorem \ref{Thm:controlbackward} is that horizontal fluctuations decrease toward the past (recall that the fluctuations around height $h$ are of order $e^h$), so the sum of fluctuations between level $-\infty$ and $0$ of a typical trajectory must by bounded.

The rest of this section is devoted to the proofs of Theorems \ref{Thm:controlforward} and \ref{Thm:controlbackward}. It will be also proved that almost surely, the $y$-coordinate goes to infinity along any trajectory (Proposition \ref{Prop:welldefined}).

\subsection{Sketch of the proofs}

The proofs are organized as follows. First, we control horizontal deviations between level $0$ and level $\delta$ for some small $\delta>0$. More precisely, we prove the following proposition:

\begin{proposition}
\label{Prop:smallheight}
Recall that $\CFD$ has been defined in Definition \ref{Def:CFD}. There exists $\delta>0$ such that, for all $p \ge 1$,
\begin{eqnarray}
\label{E:smallheight}
\mathbb{E}_{0 \in \mathcal{L}_0}\left[ \CFD_0^\delta(0)^p \right]<\infty.
\end{eqnarray}
In particular, for all $t \in [0,\delta]$, $\mathbb{P}_{0 \in \mathcal{L}_0}$-a.s., $\mathcal{A}_0^t(0) \neq \infty$.
\end{proposition}
The proof of Proposition \ref{Prop:smallheight}, based on a block control argument, is done in Section \ref{S:proofsmallheight}. In Section \ref{S:proofwelldefined}, we deduce Proposition \ref{Prop:welldefined} from Proposition \ref{Prop:smallheight}.

Then, we prove (i) in Theorem \ref{Thm:controlforward} as follows: we propagate the control up to level $\delta$ given by Proposition \ref{Prop:smallheight} to obtain a control up to level $t$ for all $t \ge 0$. It will be done by induction: from a control up to level $t$, we deduce a new control up to level $t+\delta$, by using Proposition \ref{Prop:smallheight} and mass transport arguments. The proof of (i) is done in Section \ref{S:proofcontrolforward}.

In order to prove (ii), we will apply (i) to some particular $X_0$ measurable w.r.t $\sigma(N,Y)$, where $Y$ is some random variable independent of $N$. The extra randomness that will be used in the definition of $X_0$ is the reason why we introduced the extra random variable $Y$ in Section \ref{S:tools}. The proof of (ii) is done in Section \ref{S:proofcontrolforward2}.

Finally we prove Theorem \ref{Thm:controlbackward} in Section \ref{S:proofcontrolbackward}. The proof is based on (ii) in Theorem \ref{Thm:controlforward} and mass transport arguments.

\subsection{Proof of Proposition \ref{Prop:smallheight}}
\label{S:proofsmallheight}

Let us first introduce some notations.
We pave $\mathbb{R}^d$ with cubes of edge length $R$, where $R>0$ is sufficiently large and will be chosen later.  For $a=(a_1,...,a_d) \in \mathbb{Z}^d$, let us define the cube $K_a:=\prod_{1 \le i \le d} [R(a_i-1/2),R(a_i+1/2))$. Let us also define the bottom and top cells $\Psi_a^b:=K_a \times [e^0,e^\delta)$ and $\Psi_a^t:=K_a \times [e^\delta,e]$, where $0<\delta \le 1/2$ is sufficiently small and will be chosen later. 

For $a \in \mathbb{Z}^d$, we say that $K_a$ is \emph{good} if $\Psi_a^b$ contains no points of $N$, and $\Psi_a^t$ contains at least one point of $N$, i.e. we define the event
\begin{eqnarray*}
\text{Good}(a):=\{N \cap \Psi_a^b=\emptyset\} \cap \{N \cap \Psi_a^t \neq \emptyset\}.
\end{eqnarray*}
Note that the event $\text{Good}(a)$ only depends on $N \cap (\Psi_a^b \cup \Psi_a^t)$ and the cells $(\Psi_a^b \cup \Psi_a^t)$ are disjoint, so the events $\text{Good}(a)$ are mutually independent. Moreover they have the same probability by invariance by horizontal translations.

For $m \in \mathbb{N}$, we say that $K_a$ is $m$-\emph{very good}, if $K_a$ is good and if all cubes at distance at most $m$ of the $K_a$ are good:
\begin{eqnarray*}
\text{VeryGood}_m(a):=\bigcap_{\substack{a' \in \mathbb{Z}^d,\\ \|a'-a\|_{\infty} \le m}} \text{Good}(a').
\end{eqnarray*}

\begin{figure}[!h]
    \centering
    \begin{tikzpicture}[scale=0.7]
        \draw [>=stealth,->] (-3,0)--(-2,0) node[right] {$x$};
        \draw [>=stealth,->] (-3,0)--(-3,1) node[above] {$y$};
        \draw (0,0)--(5,0)--(5,5)--(0,5)--(0,0);
        \draw (0,1)--(5,1);
        \draw (0,0) node[left] {$e^0$};
        \draw (0,1) node[left] {$e^\delta$};
        \draw (0,5) node[left] {$e^1$};
        \filldraw[gray] (2,3) circle (2pt);
        \filldraw[gray] (3,4) circle (2pt);
        \draw[decorate,decoration={brace}] (5.1,1)--(5.1,0);
        \draw[decorate,decoration={brace}] (5.1,5)--(5.1,1);
        \draw (5.2,0.5) node[right] {$\Psi_a^b$};
        \draw (5.2,3) node[right] {$\Psi_a^t$};
        \draw (2.5,-0.1) node[below] {$K_a$};
    \end{tikzpicture}
    \caption{A good cube}
    \label{Fig:goodcube}
\end{figure}
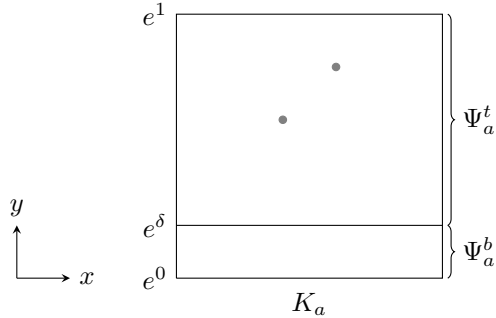

We can consider the random field $V_m:\mathbb{Z}^d \rightarrow \{0,1\}$ defined as $V_m(a)=\mathbf{1}_{\text{VeryGood}_m(a)}$ for all $a \in \mathbb{Z}^d$. We denote by $\Upsilon_m$ the connected component of the subgraph induced by $\{a \in \mathbb{Z}^d,~V_m(a)=0\}$ containing the origin if $V_m(0)=0$ (otherwise we set $\Upsilon_m=\emptyset$). This is the connected component of (indices of) non $m$-very good cubes containing the origin.  Let $\tilde{\Upsilon}_m=\bigcup_{a \in \Upsilon_m} K_a$. We also define $\rho_m:=\sup\{\|a\|,~a \in \Upsilon_m\} \in \mathbb{N} \cup \{\infty\}$ the radius of $\Upsilon_m$, with the convention $\sup\emptyset=0$. Note that those quantities depend on $\delta$ and $R$.\\

In order to prove Proposition \ref{Prop:smallheight}, we will prove that, for $m$ large enough, any trajectory from a $0$-level point in $K_0 \times \{e^0\}$ crosses the level $\delta$ at most "just after" it exits $\tilde{\Upsilon}_m \times [e^0,\infty)$, and that we can choose $R,\delta$ such that $\Upsilon_m$ is small (i.e. its radius admits exponential moment).

More precisely, we will use the two following lemmas that are both proved at the end of the section. The first lemma asserts that, when a trajectory (projected on the $x$-axis hyperplane) crosses a $m$-very good cube for $m$ large enough, then it crosses the level $\delta$ not far from this cube.

\begin{lemma}
\label{Lem:exitsdelta}
There exists $m \in \mathbb{N}$ depending only on $d$ such that, almost surely, for all $R \ge 1$ and for all $\delta \in (0,1/2]$, the following happens: \\
(i) for all $m$-very good cube $K_a$, for all $0 \le t \le t' \le \delta$ and for all $x \in \mathcal{L}_t \cap K_a$,
\begin{eqnarray}
\label{E:exitsdelta}
\|\mathcal{A}_t^{t'}(x)-Ra\| \le mR. 
\end{eqnarray}
(ii) for all $x \in \mathcal{L}_0 \cap K_0$, and for all $t' \in [0,\delta]$, 
\begin{eqnarray}
\label{E:inequastep2}
\|\mathcal{A}_0^{t'}(x)-x\| \le R(\sqrt{d}+m+\rho_m+1).
\end{eqnarray} 
\end{lemma}
The intuition behind \eqref{E:inequastep2} is the following. If $x \in \mathcal{L}_0 \cap K_0$, when the trajectory from $(x,e^0)$ exits the "bad" component $\Upsilon_m$, which has radius $\rho_m$, it crosses a $m$-very good cube. Then by Point (i), the trajectory should exit the strip $\mathbb{R}^d \times [\tvc{e^0},e^\delta]$ at most at distance $mR$ from the center of this cube. 

To prove the Point (i) of Lemma \ref{Lem:exitsdelta}, we will show that the radius admits exponential moment, and use for this a theorem due to Liggett, Schonmann and Stacey \cite[Theorem 0.0, p.75]{liggett} to show that the field $(V_m(a))_{a \in \mathbb{Z}^d}$ is dominated from below by a product random field with density $\rho$ that can arbitrarily close to 1 as $\mathbb{P}[\text{VeryGood}_m(0)]$ is close to 1. \\

The next lemma asserts that, $R$ and $\delta$ can be chosen such that the radius $\rho_m$ of the "bad" component admits exponential moments.

\begin{lemma}
\label{Lem:expdeca}
For all $m \in \mathbb{N}$, there exists $R \ge 1$ and $\delta \in (0,1/2]$ such that for all \tvc{$A \ge 1$,}
\begin{eqnarray}
\label{E:expdeca}
\mathbb{P}[\rho_m>A] \le e^{-CA}.
\end{eqnarray}
\end{lemma}

Let us now end the proof of Proposition \ref{Prop:smallheight}, assuming Lemmas \ref{Lem:exitsdelta} and \ref{Lem:expdeca}. 
Choose $m$ that satisfies Lemma \ref{Lem:exitsdelta}. Then choose $R,\delta>0$ that satisfies Lemma \ref{Lem:expdeca} for the value of $m$ previously chosen. \\

Let $x \in \mathcal{L}_0 \cap K_0$. By Inequality (\ref{E:inequastep2}) proved in Step 2, the trajectory starting from $(x,e^0)$ is entirely contained in the cylinder $\mathcal{C}:=B(0,R(\sqrt{d}+m+\rho_m+1)) \times [e^0,e^\delta]$ before exiting the strip $\mathbb{R}^d \times [e^0,e^\delta]$. Then this portion of trajectory is made of Euclidean segments whose horizontal deviations are upperbounded by $2R(\sqrt{d}+m+\rho_m+1)$. Moreover, the number of segments is (roughly) upperbounded by $1+\#(N \cap \mathcal{C})$. Then
\begin{eqnarray}
\label{E:smheeqn0}
\CFD_0^\delta(x) \le 2R(\sqrt{d}+m+\rho_m+1)(1+\#(N \cap \mathcal{C})).
\end{eqnarray}

By construction, $\rho_m$ admits exponential moments, and $\#(N \cap \mathcal{C})$ admits exponential moments, therefore, $2R(\sqrt{d}+m+\rho_m+1)(1+\#(N \cap \mathcal{C})) \in L^p$ for all $p \ge 1$.

Now, let $p \ge 1$. Lemma \ref{Lem:technicallemmageneral} applied to the weight function $g(x,\eta)=\CFD_0^\delta(x)(\eta)^p$, with $A=[-1/2,1/2]^d$, gives:
\begin{eqnarray}
\label{E:smheeqn1}
&&\alpha_0\mathbb{E}_{0 \in \mathcal{L}_0}\left[\CFD_0^\delta(x)^p \right]=\mathbb{E}\left[ \sum_{x \in [-1/2,1/2]^d} \CFD_0^\delta(x)^p \right] \nonumber\\
&& \overset{(\ref{E:smheeqn0})}{\le} \mathbb{E}\left[ \#(\mathcal{L}_0 \cap [-1/2,1/2]^d) ~ \big[2R(\sqrt{d}+m+\rho_m+1)(1+\#(N \cap \mathcal{C}))\big]^p \right].
\end{eqnarray}

By Proposition \ref{Prop:nop}, $\#(\mathcal{L}_0 \cap [-1/2,1/2]^d) \in L^2$. Moreover, $2R(\sqrt{d}+m+\rho_m+1)(1+\#(N \cap \mathcal{C})) \in L^{2p}$ by the previous discussion. Thus Cauchy-Schwarz gives,
\begin{eqnarray}
\label{E:smheeqn2}
\mathbb{E}\left[ \#(\mathcal{L}_0 \cap [-1/2,1/2]^d) ~\big[ 2R(\sqrt{d}+m+\rho_m+1)(1+\#(N \cap \mathcal{C}))\big]^p \right]<\infty,
\end{eqnarray}
so, combining (\ref{E:smheeqn1}) and (\ref{E:smheeqn2}), we obtain $\mathbb{E}_{0 \in \mathcal{L}_0}\left[ \CFD_0^\delta(0)^p \right]<\infty$, this proves Proposition \ref{Prop:smallheight}. \hfill $\Box$

It remains to prove Lemmas \ref{Lem:exitsdelta} and \ref{Lem:expdeca}.

\begin{proof}[Proof of Lemma \ref{Lem:exitsdelta}] 
We start with the proof of Point (i). Consider $m \in \mathbb{N}$ large enough that will be chosen later. Let $R \ge 1$ and $\delta \in [0,1/2]$. Let $a \in \mathbb{Z}^d$ and suppose that $K_a$ is an $m$-very good cube. Let $0 \le t \le t' \le \delta$ and $x \in \mathcal{L}_t \cap K_a$, define $z_0=(x,e^t)$ and let $z=z_{0\downarrow} \in N$ (recall that the notation $z_{0\downarrow}$ has been defined by \eqref{E:updownarrow}). Let $B=B_{\mathbb{R}^d}(Ra,Rm)$. By definition of a $m$-very good cube, none of the $\Psi^b_{a'}$ when $\|a'-a\|_\infty \le m$ contains points of $N$ and since $B$ is included in the union of $K_{a'}$ for $\|a'-a\|_\infty \le m$, $N \cap (B \times [e^0,e^{t'}])=\emptyset$. Thus $A(z) \notin B \times [e^0,e^{t'}]$. 

Suppose $m \ge R\sqrt{d}/2$. Then $K_a \subset B$, so $z_0=(x,e^t) \in K_a \times \{e^t\} \subset B \times [e^0,e^{t'}]$. So $[z_0,A(z)]_{eucl}$ must cross either $B \times \{e^{t'}\}$, or $\partial B \times [e^0,e^{t'}]$. In the first case, $\mathcal{A}_t^{t'}(x) \in B$, so we are done.

\tvc{It then remains to eliminate the case $[z_0,A(z)]_{eucl} \cap (\partial B \times [e^0,e^{t'}]) \neq \emptyset$. Suppose by contradiction that $[z_0,A(z)]_{eucl}$ crosses $\partial B \times [e^0,e^{t'}]$ and denote by $z_1$ this intersection point. Since $z_0$ and $z_1$ belong to the Euclidean segment $[z,A(z)]_{eucl}$ and by convexity of hyperbolic balls, the inclusion
\begin{equation}
\label{InclusionKa}
B(z_0 , d(z_0,z_1)) \subset B(z , d(z,A(z)))
\end{equation}
holds where $d(\cdot,\cdot)$ denotes the hyperbolic distance. Besides $d(z_0,z_1)$ becomes as large as we want as $m\to\infty$. Hence, we choose $m$ large enough so that $B(z_0 , d(z_0,z_1))$ contains $\Psi_a^t$. Since $K_a$ is a good cube, there is at least a Poisson point in $\Psi_a^t$, say $z_2$. Combining with (\ref{InclusionKa}), we get $z_2 \in B^{+}(z,d(z,A(z)))$ which contradicts the definition of parent.}

\begin{figure}[!ht]
\begin{center}
\includegraphics[width=11cm,height=5.5cm]{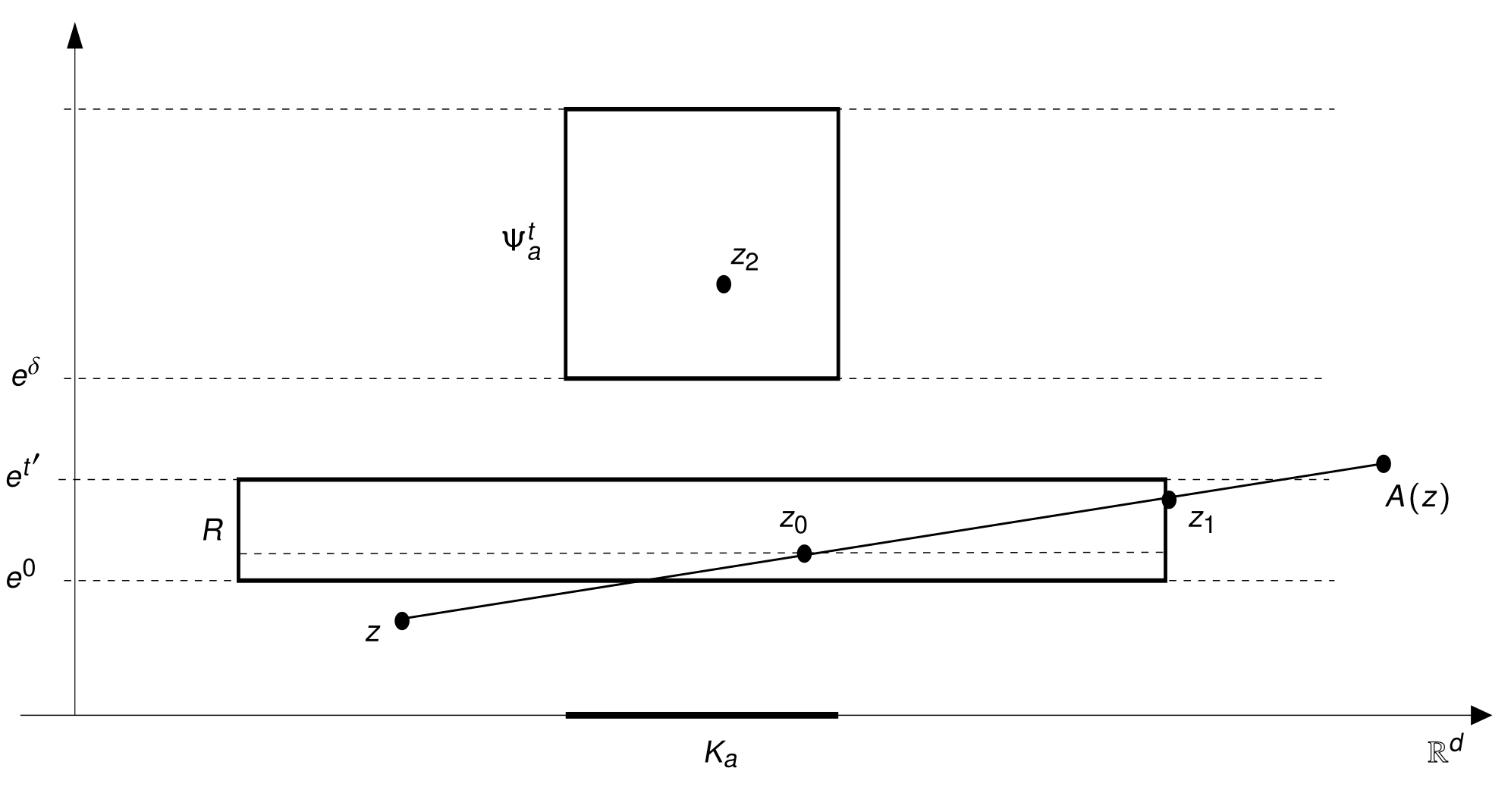}
\caption{The Euclidean segment $[z,A(z)]_{eucl}$ exits the rectangle $R:=B \times [e^0,e^{t'}]$ through $z_1\in \partial B \times [e^0,e^{t'}]$. Whenever $K_a$ is a good cube, this situation does not occur.}
\end{center}
\end{figure}

Let us now prove Point (ii). The conclusion is immediate if $\rho_m=\infty$, so we suppose that $\rho_m<\infty$ in the following. Let $x \in \mathcal{L}_0 \cap K_0$ and $t' \in [0, \delta]$. If $K_0$ is very good, then by the Point (i) of Lemma \ref{Lem:exitsdelta} applied with $t=0$, $\|\mathcal{A}_0^{t'}(x)\| \le mR$, so $\|\mathcal{A}_0^{t'}(x)-x\| \le mR+\|x\| \le R(m+\sqrt{d}/2)$, so we are done.

Suppose that $K_0$ is not $m$-very-good. Then $\Upsilon_m \neq \emptyset$. Let us define the outer-boundary of $\Upsilon_m$ and $\tilde{\Upsilon}_m$ by $\partial_{out}\Upsilon_m=\{a \in \mathbb{Z}^d, a \notin \Upsilon_m \text{ and } \exists a'\in\Upsilon_m, \|a-a'\|=1 \}$.

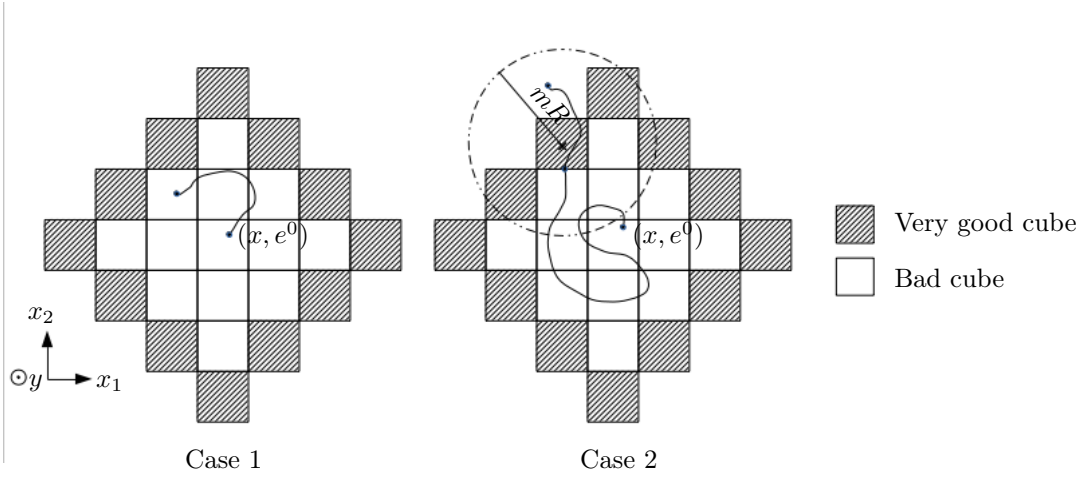
\begin{figure}[!h]
    \centering
    \begin{tikzpicture}
    \node at (0,0)
    {\includegraphics[scale=0.5]{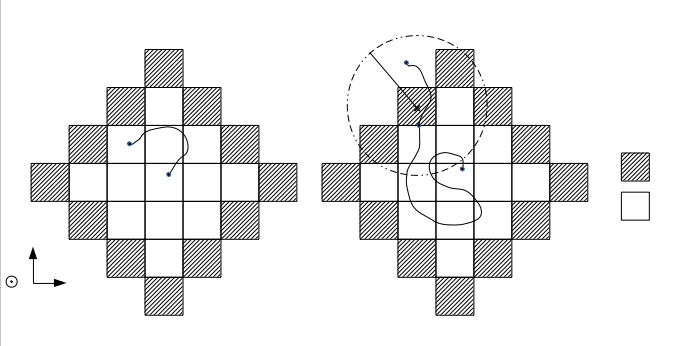}};
    \draw (-4.7,-2) node {$x_1$};
    \draw (-5.6,-1.1) node {$x_2$};;
    \draw (-5.67,-2) node {$y$};
    \draw (5.5,0.1) node[right] {Very good cube};
    \draw (5.5,-0.6) node[right] {Bad cube};
    \draw (-3.2,-3) node {Case 1};
    \draw (2,-3) node {Case 2};
    \draw (-3.15,-0.05) node[right] {$(x,e^0)$};
    \draw (2.05,-0.05) node[right] {$(x,e^0)$};
    \draw (1.1,1.7) node[rotate=-47] {$mR$};
    \end{tikzpicture}

    \caption{Representation of the trajectory from $(x,e^0)$ below level $t'$}
    \label{Fig:exitstrip}
\end{figure}

By definition $\partial_{out}
\Upsilon_m$ is made of (indices of) very good cubes. Moreover, for all $a \in \partial_{out}\Upsilon_m$, $\|a\| \le \rho_m+1$. Since $\rho_m<\infty$, $\tilde{\Upsilon}_m$ is bounded so there are a finite number of points of $N$ in $\tilde{\Upsilon}_m \times [e^0,e^{t'}]$. Then the trajectory starting from $(x,e^0)$ should exit $\tilde{\Upsilon}_m \times [e^0,e^{t'}]$, so by continuity it should cross $\partial(\tilde{\Upsilon}_m \times [e^0,e^{t'}])$. Consider the first time (i.e. the lowest level) when the trajectory crosses $\partial(\tilde{\Upsilon}_m \times [e^0,e^{t'}])$, i.e.
\begin{eqnarray*}
t_{min}=\min\{t>0,(\mathcal{A}_0^t(x),e^t) \in \partial(\tilde{\Upsilon}_m \times [e^0,e^{t'}])\}.
\end{eqnarray*}
The time $t_{min}$ is well-defined since $\partial(\tilde{\Upsilon}_m \times [e^0,e^{t'}])$ is closed. If $t_{min}=t'$ (Case 1 in Figure \ref{Fig:exitstrip}), then for all $t \in [0,t']$, $\mathcal{A}_0^t(x) \in \tilde{\Upsilon}_m$, so $\|\mathcal{A}_0^t(x)-x\| \le \rho_m R+\|x\| \le R(\rho_m+\sqrt{d}/2)+R\sqrt{d}/2=R(\rho_m+\sqrt{d})$, so we are done. Otherwise (Case 2 in Figure \ref{Fig:exitstrip}), $t_{min}<t'$. In this case, $\mathcal{A}_0^{t_{min}}(x) \in \partial\tilde{\Upsilon}_m$ so $\mathcal{A}_0^{t_{min}}(x) \in K_a$ for some $a \in \partial_{out}\Upsilon_m$. Since $a \in \partial\Upsilon_m$, $K_a$ is a very good cube, therefore by Lemma \ref{Lem:exitsdelta}, $\|\mathcal{A}_0^{t'}(x)-Ra\| \le \|\mathcal{A}_{t_{min}}^{t'}(\mathcal{A}_0^{t_{min}}(x))-Ra\| \le mR$. Then $\|\mathcal{A}_0^{t'}(x)-x\| \le \|Ra\|+\|Ra-\mathcal{A}_0^{t'}(x)\|+\|x\| \le R(\rho_m+1)+mR+R\sqrt{d}/2=R(\sqrt{d}+m+\rho_m+1)$, this completes the proof of (\ref{E:inequastep2}).
\end{proof}

\begin{proof}[Proof of Lemma \ref{Lem:expdeca}]
Let $m \in \mathbb{N}$. By translation invariance, $\mathbb{P}[\text{Good}(a)]=\mathbb{P}[\text{Good}(0)]$ for all $a \in \mathbb{Z}^d$. Since the events $\text{Good}(a)$ are mutually independent, for all $a \in \mathbb{Z}^d$, $
\mathbb{P}[\text{VeryGood}_m(a)]=\mathbb{P}[\text{Good}(0)]^{(2m+1)^d}$.
By definition, for $a \in \mathbb{Z}^d$, the event $\text{VeryGood}_m(a)$ only depends on the events $\text{Good}(a')$ with $\|a'-a\|_\infty \le m$. In particular, the events $\text{VeryGood}_m(a)$ are not mutually independent. However, the dependencies are only local. Let $a,a' \in \mathbb{Z}^d$ such that $\|a-a' \|_\infty>2m$. For all $a'' \in \mathbb{Z}^d$, we can't have both $\|a''-a\|_\infty \le m$ and $\|a''-a'\|_\infty \le m$. Therefore, $\text{VeryGood}_m(a)$ is independent of the family of events $(\text{VeryGood}_m(a'))_{a' \in \mathbb{Z}^d, \|a'-a\|_\infty>2m}$. So the field $(V_m(a))_{a \in \mathbb{Z}^d}$ is $2m$-dependant.

Thus Theorem 0.0 of \cite{liggett} tells us that there exists a non-decreasing function $\chi:[0,1] \to [0,1]$ verifying $\lim_{t \to 1} \chi(t)=1$ (and independent of the parameters $R,\delta$) such that, if $(Y_a)_{a \in \mathbb{Z}^d}$ is a product random field of intensity $\chi(\mathbb{P}[V_m(0)=1])$, then $(V_m(a))_{a \in \mathbb{Z}^d} \succeq_{st} (Y_a)_{a \in \mathbb{Z}^d}$ for the product order on $\{0,1\}^{\mathbb{Z}^d}$.

\tvc{Using a Peierls argument, we can choose $\tilde{p}_0>0$ sufficiently small such that for all $p<\tilde{p}_0$, in the product random field $(Y_a)_{a \in \mathbb{Z}^d}$ of density $p$, the radius of the cluster containing the origin admits exponential moments.} Pick $p_0'<1$ such that $\chi(p)>1-\tilde{p}_0$ for all $p>p_0'$ (it is possible since $\lim_{p \to 1} \chi(p)=1$), and set $p_0^*=p_0'^{(2m+1)^{-d}}<1$. It is shown in the next paragraph that $\mathbb{P}[\text{Good}(0)]>p_0^*$ for judiciously chosen $R,\delta$. Then $\mathbb{P}[\text{VeryGood}_m(a)]=\mathbb{P}[\text{Good}(0)]^{(2m+1)^d}>p_0'$. Therefore $\chi\left(\mathbb{P} [\text{VeryGood}_m(a)]\right)>1-\tilde{p}_0$ by our choice of $p_0'$. So the field $(Y_a)_{a \in \mathbb{Z}^d}$ is a product random field with density larger than $1-\tilde{p}_0$. By our choice of $\tilde{p}_0$, it implies that the radius of the component of $\{a \in \mathbb{Z}^d,~Y_a=0\}$ containing the origin admits exponential moments, which implies that $\rho_m$ admits exponential moments by stochastic domination.

It remains to show that we can choose $R \ge 1$ and $\delta \in (0,1/2]$ such that $\mathbb{P}[\text{Good}(0)]>p_0^*$. Since $\Psi_0^b$ and $\Psi_0^t$ are disjoint, by independence
\begin{eqnarray*}
\mathbb{P}[\text{Good}(0)]&=\mathbb{P}[N \cap \Psi_a^b=\emptyset]\mathbb{P}[N \cap \Psi_a^t \neq \emptyset] =\exp\left({-\lambda\mu(\Psi_a^b)}\right)\left(1-\exp(-\lambda\mu(\Psi_a^t))\right).
\end{eqnarray*}
We have
\begin{eqnarray*}
\mu\left( \Psi_a^b \right)=\int_{K_0}\int_1^{e^\delta} \frac{1}{y^d}~dy~dx=\begin{cases}R^d\frac{1-e^{-(d-1)\delta}}{d-1}  & \mbox{ if }d>1\\
\delta R^d & \mbox{ if }d=1.\end{cases}
\end{eqnarray*}
and
\begin{eqnarray*}
&\mu\left( \Psi_a^t \right)&=\int_{K_0}\int_{e^\delta}^e \frac{1}{y^d}~dy~dx=\begin{cases}
R^d\frac{e^{-(d-1)\delta}-e^{-(d-1)}}{d-1} & \mbox{ if } d>1\\
(1-\delta) R^d & \mbox{ if } d=1.
 \end{cases}
\end{eqnarray*}
When $d>1$, let $\kappa:=\frac{e^{-(d-1)/2}-e^{-(d-1)}}{d-1}$. If $\delta<1/2$, then
\begin{eqnarray*}
\mu\left( \Psi_a^t \right) \ge R^d\frac{e^{-(d-1)/2}-e^{-(d-1)}}{d-1}=\kappa R^d.
\end{eqnarray*}
Since $1-\exp\left( -\lambda \kappa R^d \right) \to 1$ when $R \to \infty$, we can pick $R$ large enough such that $1-\exp\left( -\lambda \kappa R^d \right)>\sqrt{p_0^*}$.

Now $R$ is chosen, and at $R$ fixed, $\mu\left( \Psi_a^b \right) \to 0$ when $\delta \to 0$, so $\exp\left({-\lambda\mu(\Psi_a^b)}\right) \to 1$ when $\delta \to 0$. We pick $\delta$ small enough (and also smaller than 1/2) such that $\exp\left(-\lambda\mu(\Psi_a^b)\right)>\sqrt{p_0^*}$. For this choice of $(R,\delta)$,
\begin{eqnarray*}
&\mathbb{P}[\text{Good}(0)]&=\exp\left({-\lambda\mu(\Psi_a^b)}\right)\left(1-\exp(-\lambda\mu(\Psi_a^t))\right) >
p_0^*,
\end{eqnarray*}
this proves Lemma \ref{Lem:expdeca}.\end{proof}

\subsection{Proof of Proposition \ref{Prop:welldefined}}
\label{S:proofwelldefined}
In the following, we consider $\delta>0$ such that Proposition \ref{Prop:smallheight} holds. Let us deduce Proposition \ref{Prop:welldefined} from Proposition \ref{Prop:smallheight}.

\begin{proof}[Proof of Proposition \ref{Prop:welldefined}]

By Proposition \ref{Prop:smallheight}, $\mathbb{P}_{0 \in \mathcal{L}_0}\left[ \mathcal{A}_0^\delta(0)=\infty\right]=0$. Recall that $\mathcal{A}_0^\delta(x)(\eta)$ is the value of $\mathcal{A}_0^\delta(x)$ when $N=\eta$. Define the weight function (in the non-marked case)
$w(x,\eta)=\mathbf{1}_{x \in \mathcal{L}_0(\eta) \text{ and } \mathcal{A}_0^\delta(x)(\eta)=\infty}$
for $x \in \mathbb{R}^d$, $\eta \in \mathcal{N}_S$. Lemma \ref{Lem:technicallemmageneral} applied with $A=\mathbb{R}^d$ gives,
\begin{eqnarray*}
\mathbb{E}[\#\{x \in \mathcal{L}_0,~\mathcal{A}_0^\delta(x)=\infty \}]=\mathbb{E}\left[ \sum_{x \in \mathcal{L}_0} w(x) \right]=\infty \times
 \mathbb{E}_{0 \in \mathcal{L}_0}[w(0)]=\infty\times \mathbb{P}_{0 \in \mathcal{L}_0}\left[ \mathcal{A}_0^\delta(0)=\infty \right]=0.
\end{eqnarray*}
Thus a.s., for all $x \in \mathcal{L}_0$, $\mathcal{A}_0^\delta(x) \neq \infty$.

The \tvc{dilation} invariance property of the model implies that, for all $h \in \mathbb{R}$, a.s., for all $x \in \mathcal{L}_h$, $\mathcal{A}_h^{h+\delta}(x) \neq \infty$. We define
\begin{eqnarray*}
H_0:=\sup\{h \ge 0,~\forall x \in \mathcal{L}_0,~\mathcal{A}_0^h(x) \neq \infty\} \in [0,\infty]
\end{eqnarray*}
Note that $H_0 \ge \delta$ by the previous discussion. Suppose that $\mathbb{P}[H_0<\infty]>0$. Then there exists some (deterministic) $h_0 \ge 0$ such that
$\mathbb{P}[h_0<H_0<h_0+\delta]>0$. On this event, there exists some $x \in \mathcal{L}_0$ such that $\mathcal{A}_0^{h_0}(x) \neq \infty$ but $\mathcal{A}_{h_0}^{h_0+\delta}(\mathcal{A}_0^{h_0}(x))=\mathcal{A}_0^{h_0+\delta}(x)=\infty$. Therefore $x':=\mathcal{A}_0^{h_0}(x)$ satisfies $\mathcal{A}_{h_0}^{h_0+\delta}(x')=\infty$. So $\mathbb{P}\left[ \exists x \in \mathcal{L}_{h_0},~\mathcal{A}_{h_0}^{h_0+\delta}(x)=\infty \right]>0$, which contradicts the previous discussion. So $H_0=\infty$ a.s., i.e. a.s. for all $h \ge 0$ and for all $x \in \mathcal{L}_0$, $\mathcal{A}_0^h(x) \neq \infty$.

By dilations invariance, the same result is true for each level $t \in \mathbb{R}$: for all $t \in \mathbb{R}$, a.s., for all $x \in \mathcal{L}_t$ and for all $h \ge 0$, $\mathcal{A}_t^{t+h}(x) \neq \infty$. Therefore, almost surely,
$\forall t \in \mathbb{Q},~\forall h \in \mathbb{R}_+,~\mathcal{A}_t^{t+h}(x) \neq \infty.$ Every trajectory crosses a rational level $t \in \mathbb{Q}$, since it is the case for every non-horizontal Euclidean segments. Thus we can replace $\mathbb{Q}$ by $\mathbb{R}$ in the above conclusion. This completes the proof.
\end{proof}

\subsection{Proof of (i) in Theorem \ref{Thm:controlforward}}
\label{S:proofcontrolforward}

Let $p \ge 1$. Recall that $\delta>0$ is chosen according to Proposition \ref{Prop:smallheight}. 

The strategy of the proof consists in iterating the control of horizontal deviations up to level $\delta$ given by Proposition \ref{Prop:smallheight} to obtain a control up to level $t$ for all $t>0$. It will be shown that, for all $t \ge 0$,

\begin{eqnarray}
\label{E:propag}
\mathbb{E}\left[ \left( \frac{\|X_0\|+\CFD_0^{t+\delta}(X_0)}{e^{t+\delta}} \right)^p\right]^{1/p} \le \varphi\left(\mathbb{E}\left[ \left( \frac{\|X_0\|+\CFD_0^t(X_0)}{e^t} \right)^p\right]^{1/p}\right),
\end{eqnarray}
where
\begin{eqnarray}
\label{E:defofphi}
\varphi(s)=e^{-\delta}s+C_0s^{\frac{d}{p+d}},
\end{eqnarray}
where $C_0>0$ is a constant that only depends of $p,d,\lambda$.

The key point is that the function $\varphi$ defined in (\ref{E:defofphi}) admits a fixed point. As it will be shown later, the factor $e^{-\delta}$ in the first term of the r.h.s. of (\ref{E:defofphi}) comes from the dilation invariance. Because of the metric of $(H,ds^2)$, the horizontal fluctuations of the lowest part of a trajectory are compressed by rescaling so they have negligible impact on the total cumulated deviations. This is specific to hyperbolic geometry; in Euclidean geometry, the same argument leads to a roughly non-optimal upper-bound of horizontal deviations. 

\paragraph{Step 1:} we prove (i) assuming (\ref{E:propag}).

By assumption, $\mathbb{E}\left[ \left( \frac{\|X_0\|+\CFD_0^0(X_0)}{e^0} \right)^p\right]^{1/p}=\mathbb{E}\left[\|X_0\|^p\right]^{1/p}<\infty$, so by iterating (\ref{E:propag}), since $\varphi$ is  non-decreasing, we get for all $n \in \mathbb{N}$,
\begin{eqnarray}
\label{E:forwinequa0}
\mathbb{E}\left[ \left( \frac{\|X_0\|+\CFD_0^{n\delta}(X_0)}{e^{n\delta}} \right)^p\right]^{1/p} \le \varphi^n\left( \mathbb{E}\left[ \|X_0\|^p\right]^{1/p} \right),
\end{eqnarray}
where $\varphi^n=\varphi \circ...\circ \varphi$ $n$ times. Let $t \ge 0$ and $n=\lceil t/\delta \rceil$ (thus $\delta(n-1)<t\le\delta n$). Then, using the fact that $t \mapsto \CFD_0^t(X_0)$ is non-decreasing,
\begin{eqnarray}
\label{E:forwinequa1}
&&\mathbb{E}\left[ \left(e^{-t} \CFD_0^t(X_0) \right)^p \right]^{1/p} \le \mathbb{E}\left[ \left( \frac{\|X_0\|+\CFD_0^t(X_0)}{e^t} \right)^p\right]^{1/p}=e^{n\delta-t}\mathbb{E}\left[ \left( \frac{\|X_0\|+\CFD_0^t(X_0)}{e^{n\delta}} \right)^p\right]^{1/p} \nonumber\\
&&\le e^\delta \mathbb{E}\left[ \left( \frac{\|X_0\|+\CFD_0^{n\delta}(X_0)}{e^{n\delta}} \right)^p\right]^{1/p} \overset{(\ref{E:forwinequa0})}{\le} e^\delta \varphi^n\left( \mathbb{E}\left[ \|X_0\|^p\right]^{1/p} \right).
\end{eqnarray}

The function $\varphi$ is continuous, non-decreasing and admits a unique positive fixed point $s_0=\left( \frac{C_0}{1-e^{-\delta}} \right)^{1+d/p}$ such that
\begin{eqnarray*}
\left\{
\begin{aligned}[c|l]
\varphi(s)>s &\text{ if } 0<s<s_0\\
\varphi(s)=s &\text{ if } s \in \{0,s_0\}\\
\varphi(s)<s &\text{ if } s>s_0.
\end{aligned}
\right.
\end{eqnarray*}

Therefore, since $\mathbb{E}\left[ \|X_0\|^p\right]^{1/p} \in (0,\infty)$, $\varphi^n\left( \mathbb{E}\left[ \|X_0\|^p\right]^{1/p} \right) \to s_0$ when $n \to \infty$. Combining this with (\ref{E:forwinequa1}), we obtain that (i) holds for $K=e^\delta s_0$.

\bigbreak

\paragraph{Step 2:} we show that (\ref{E:propag}) holds for all $t \ge 0$. Let $t \ge 0$. By Minkowski inequality,
\begin{eqnarray*}
&&\mathbb{E}\left[ \left(\|X_0\|+\CFD_0^{t+\delta}(X_0)\right)^p \right]^{1/p} \nonumber\\
&&\le \mathbb{E}\left[ \left(\|X_0\|+\CFD_0^t(X_0)\right)^p \right]^{1/p}+\mathbb{E}\left[ \left(\CFD_0^{t+\delta}(X_0)-\CFD_0^t(X_0)\right)^p \right]^{1/p},
\end{eqnarray*}
so, multiplying both sides by $e^{-t-\delta}$, we obtain
\begin{eqnarray}
\label{E:forwinequa2}
&&\mathbb{E}\left[ \left(\frac{\|X_0\|+\CFD_0^{t+\delta}(X_0)}{e^{t+\delta}}\right)^p \right]^{1/p} \le e^{-\delta}\mathbb{E}\left[ \left(\frac{\|X_0\|+\CFD_0^t(X_0)}{e^t}\right)^p \right]^{1/p} \nonumber\\
&&+e^{-t-\delta}\mathbb{E}\left[ \left(\CFD_0^{t+\delta}(X_0)-\CFD_0^t(X_0)\right)^p \right]^{1/p}.
\end{eqnarray}
The first term in the r.h.s. of (\ref{E:forwinequa2}) corresponds to $e^{-\delta}s$, with \[s=\mathbb{E}\left[ \left(e^{-t}(\|X_0\|+\CFD_0^t(X_0))\right)^p \right]^{1/p}.\]The factor $e^{-\delta}$ comes from the rescaling and is crucial for the existence of the fixed point.

It remains to upperbound the second term by $C_0s^\frac{d}{p+d}$. We use Proposition \ref{Lem:MTlemma2} to rewrite the quantity $\mathbb{E}\left[ \left(\CFD_0^{t+\delta}(X_0)-\CFD_0^t(X_0)\right)^p \right]$. Let us introduce the level $t$-weighted association function $(f_t,w_t)$ defined as
\begin{eqnarray*}
\left\{
\begin{aligned}[c|l]
f_t(x,\eta,\xi)&=\mathcal{A}_0^t(X_0(T_{-x}\eta,\xi)+x)(\eta) \in \mathcal{L}_t(\eta)\\
w_t\left( x,\eta,\xi) \right)&=\CFD_t^{t+\delta}(f_t( x,\eta,\xi))^p.
\end{aligned}
\right.
\end{eqnarray*}
for all $x \in \mathbb{R}^d$, $\eta \in \mathcal{N}_S$ and $\xi \in \Upsilon$. Checking that $(f_t,w_t)$ is well-defined and is a level $t$-weighted association function is done in \cite[Section 6.6]{version_longue}. Proposition \ref{Lem:MTlemma2} applied to $(f_t,w_t)$ gives,
\begin{eqnarray}
\label{E:resultTM}
\mathbb{E}\left[ w_t(0) \right]=\alpha_0e^{-dt}\mathbb{E}_{0 \in \mathcal{L}_t}\left[ \int_{\Lambda_{f_t}(0)} w_t(x)~dx\right].
\end{eqnarray}
We have
\begin{eqnarray*}
\mathbb{E}\left[w_t(0)\right]=\mathbb{E}\left[ \CFD_t^{t+\delta}\left(\mathcal{A}_0^t(X_0)\right)^p\right]=\mathbb{E}\left[ \left(\CFD_0^{t+\delta}(X_0)-\CFD_0^t(X_0)\right)^p \right]
\end{eqnarray*}
and, since for all $x \in \tvc{\Lambda_{f_t}(0)}$, $f_t(x)=0$,
\begin{eqnarray*}
\mathbb{E}_{0 \in \mathcal{L}_t}\left[ \int_{\tvc{\Lambda_{f_t}(0)}} w_t(x)~dx\right]=\mathbb{E}_{0 \in \mathcal{L}_t}\left[ \int_{\tvc{\Lambda_{f_t}(0)}} \CFD_t^{t+\delta}(0)^p~dx\right]
=\mathbb{E}_{0 \in \mathcal{L}_t}\left[ V_t(0) ~ \CFD_t^{t+\delta}(0)^p\right]
\end{eqnarray*}
where $V_t(0)=\Leb(\Lambda_{f_t}(0))$. Then (\ref{E:resultTM}) can be rewritten as:
\begin{eqnarray}
\label{E:resultTM2}
\mathbb{E}\left[ \left(\CFD_0^{t+\delta}(X_0)-\CFD_0^t(X_0)\right)^p \right]=\alpha_0e^{-dt}\mathbb{E}_{0 \in \mathcal{L}_t}\left[ V_t(0) ~ \CFD_t^{t+\delta}(0)^p\right].
\end{eqnarray}
By Proposition \ref{Lem:technical} applied to $f_t$,
\begin{eqnarray}
\label{E:controlvt}
\mathbb{E}_{0 \in \mathcal{L}_t}\left[ V_t(0)^{1+p/d} \right] \le C_{p,d}e^{dt}\mathbb{E}\left[ \|\mathcal{A}_0^t(X_0)\|^p \right] \le C_{p,d}e^{dt}\mathbb{E}\left[ \left( \|X_0\|+\CFD_0^t(X_0)\right)^p \right]
\end{eqnarray}
since $\|\mathcal{A}_0^t(X_0)\| \le \|X_0\|+\CFD_0^t(X_0)$. Thus Hölder inequality gives,
\begin{eqnarray}
\label{E:applyholder}
&&\mathbb{E}_{0 \in \mathcal{L}_t}\left[ V_t(0) ~\CFD_t^{t+\delta}(0)^p\right] \le \mathbb{E}_{0 \in \mathcal{L}_t}\left[ V_t(0)^{1+p/d} \right]^{\frac{d}{p+d}}\mathbb{E}_{0 \in \mathcal{L}_t}\left[ \CFD_t^{t+\delta}(0)^{p+d} \right]^{\frac{p}{p+d}} \nonumber\\
&&\overset{(\ref{E:controlvt})}{\le} C_{p,d}^{\frac{d}{p+d}}e^{\frac{d^2t}{p+d}}\mathbb{E}\left[ \left( \|X_0\|+\CFD_0^t(X_0)\right)^p \right]^{\frac{d}{p+d}}\mathbb{E}_{0 \in \mathcal{L}_t}\left[ \CFD_t^{t+\delta}(0)^{p+d} \right]^{\frac{p}{p+d}}.
\end{eqnarray}
Since
\begin{eqnarray*}
\CFD_t^{t+\delta}(0)(N)=e^t ~ \CFD_0^\delta(0)(D_{e^{-t}}N),
\end{eqnarray*}
by dilation invariance (Lemma \ref{Lem:scaleinv}),
\begin{eqnarray*}
\mathbb{E}_{0 \in \mathcal{L}_t}\left[ \CFD_t^{t+\delta}(0)^{p+d} \right]=e^{t(p+d)}\mathbb{E}_{0 \in \mathcal{L}_0}\left[ \CFD_0^\delta(0)^{p+d} \right].
\end{eqnarray*}
Then (\ref{E:applyholder}) can be rewritten as
\begin{eqnarray}
\label{E:applyholder2}
&&\mathbb{E}_{0 \in \mathcal{L}_t}\left[ V_t(0)~ \CFD_t^{t+\delta}(0)^p\right]
\nonumber\\
&&\le C_{p,d}^{\frac{d}{p+d}}e^{tp+\frac{d^2t}{p+d}}\mathbb{E}\left[ \left( \|X_0\|+\CFD_0^t(X_0)\right)^p \right]^{\frac{d}{p+d}}\mathbb{E}_{0 \in \mathcal{L}_0}\left[ \CFD_0^\delta(0)^{p+d} \right]^{\frac{p}{p+d}}.
\end{eqnarray}
Then
\begin{eqnarray}
\label{E:forwinequa3}
&&\mathbb{E}\left[ (\CFD_0^{t+\delta}(X_0)-\CFD_0^t(X_0))^p \right] \overset{(\ref{E:resultTM2})}{=}\alpha_0e^{-dt}\mathbb{E}_{0 \in \mathcal{L}_t}\left[ V_t(0)~\CFD_t^{t+\delta}(0)^p\right] \nonumber\\
&&\overset{(\ref{E:applyholder2})}{\le} \alpha_0C_{p,d}^{\frac{d}{p+d}}e^{t(p-d)+\frac{d^2t}{p+d}}\mathbb{E}\left[ \left( \|X_0\|+\CFD_0^t(X_0)\right)^p \right]^{\frac{d}{p+d}}\mathbb{E}_{0 \in \mathcal{L}_0}\left[ \CFD_0^\delta(0)^{p+d} \right]^{\frac{p}{p+d}} \nonumber\\
&&\le \tilde{C}_0e^{\frac{p^2t}{p+d}}\mathbb{E}\left[ \left( \|X_0\|+\CFD_0^t(X_0)\right)^p \right]^{\frac{d}{p+d}}
\end{eqnarray}
where $\tilde{C}_0=\alpha_0C_{p,d}^{\frac{d}{p+d}}\mathbb{E}_{0 \in \mathcal{L}_0}\left[ \CFD_0^\delta(0)^{p+d} \right]^{\frac{p}{p+d}}<\infty$ by Proposition \ref{Prop:smallheight}, and where we used that $p-d+d^2/(p+d)=p^2/(p+d)$. We rewrite (\ref{E:forwinequa3}) as
\begin{eqnarray}
\label{E:forwinequa4}
e^{-t-\delta}\mathbb{E}\left[ \left(\CFD_0^{t+\delta}(X_0)-\CFD_0^t(X_0)\right)^p \right]^{1/p} 
&\le \tilde{C}_0^{1/p}e^{-\delta-\frac{td}{p+d}}\mathbb{E}\left[ \left( \|X_0\|+\CFD_0^t(X_0)\right)^p \right]^{\frac{d}{p(p+d)}} \nonumber\\
&=\tilde{C}_0^{1/p}e^{-\delta}\mathbb{E}\left[ \left(\frac{\|X_0\|+\CFD_0^t(X_0)}{e^t}\right)^p \right]^{\frac{d}{p(p+d)}}.
\end{eqnarray}
Combining (\ref{E:forwinequa2}) and (\ref{E:forwinequa4}) we obtain that (\ref{E:propag}) holds with $C_0=\tilde{C}_0^{1/p}e^{-\delta}$. This completes the proof.

\subsection{A geometrical lemma}
We now prove the following lemma:

\begin{lemma}
\label{Lem:closest}
For all $p \ge 1$, we have
\begin{eqnarray*}
\mathbb{E}\left[ \left(\min_{x \in \mathcal{L}_0} \|x\|\right)^p \right]<\infty.
\end{eqnarray*}
\end{lemma}
This will be used in the proof of (ii) in Theorem \ref{Thm:controlforward}, and several times in the following.

\begin{proof}
We will in fact prove that $\min_{x \in \mathcal{L}_0} \|x\|$ admits exponential moments. Choose $A>0$ large enough such that, for $x_1,x_2 \in \mathbb{R}^d$, if $\|x_1-x_2\| \ge A$ then $d((x_1,e^0),(x_2,e^0)) \ge 6$. For $n \in \mathbb{N}$, define
\begin{eqnarray*}
p_n:=(Ane_1,e^0), \quad B_n^1:=B_H(p_n,1), \quad B_n^3:=B_H(p_n,3).
\end{eqnarray*}
Let us also define
\begin{eqnarray*}
B_n^{1-}=B_n^1 \cap (\mathbb{R}^d \times (0,e^0)), \quad B_n^{1+}=B_n^1 \cap (\mathbb{R}^d \times [e^0,\infty)).
\end{eqnarray*}
For $n \in \mathbb{N}$, we now define the event $E_n$ meaning that there is exactly one point of $N$ in $B_n^{1-}$, exactly one point of $N$ in $B_n^{1+}$ and no more points in $B_n^3$:
\begin{eqnarray*}
E_n:=\{\#(N \cap B_n^{1-})=\#(N \cap B_n^{1+})=1 \text{ and } \#(N \cap B_n^3)=2\}.
\end{eqnarray*}
The event $E_n$ only depend on the process $N$ inside the ball $B_n^3$, and the balls $(B_n^3)_{n \in \mathbb{N}}$ are pairwise disjoint by our choice of $A$, so the events $(E_n)$ are mutually independent. Moreover they all have the same probability $p>0$. It is shown in the next paragraph that, on $E_n$, $\min_{x \in \mathcal{L}_0} \|x\| \le An+3$. Consider $n_{\min}=\min\{n \in \mathbb{N},~E_n \text{ occurs}\}$.
The random variable $n_{\min}$ is distributed according to a geometric distribution so it admits exponential moments. Since $\min_{x \in } \|x\| \le An_{\min}+3$, it implies Lemma \ref{Lem:closest}.

It remains to show that $E_n$ implies $\min_{x \in \mathcal{L}_0} \|x\| \le An+3$. Fix $\eta \in E_n$ and consider $z_1$ (resp. $z_2$) the unique point in $\eta \cap B_n^{1-}$ (resp. $\eta \cap B_n^{1+}$). For any $z \in B^+(z_1)(\eta)$,
\begin{eqnarray*}
d(z,p_n) \le d(z,z_1)+d(z_1,p_n) \le d(z_1,z_2)+d(z_1,p_n) \le 2d(z_1,p_n)+d(p_n,z_2) \le 3,
\end{eqnarray*}
so $B^+(z_1)(\eta) \subset B_n^3$. Since $\eta \cap B_n^3$ contains no more points than $z_1$ and $z_2$, this implies that $B^+(z_1)(\eta)=\emptyset$, so $[z_1,z_2]_{eucl} \in \DSF(\eta)$. Consider the intersection point $z=(x_0,e^0)$ of $[z_1,z_2]_{eucl}$ and the hyperplane $\mathbb{R}^d \times \{1\}$. Then $z \in B_n^3$ so, by the discussion below Proposition \ref{Prop:disthalfspace}, $\|x-Ane_1\| \le d(z,p_n) \le 3$. Thus $\|x\| \le An+3$, and $x \in \mathcal{L}_0(\eta)$, so $\min_{x \in \mathcal{L}_0(\eta)} \|x\| \le An+3$. This completes the proof of Lemma \ref{Lem:closest}.
\end{proof}

\subsection{Proof of (ii) in Theorem \ref{Thm:controlforward}}
\label{S:proofcontrolforward2}

\tvc{
Let $p \ge 1$. Let us assume first that there exists some level $0$-association function $f$ verifying $\mathbb{E}\left[ \|f(0)\|^{2p} \right]<\infty$ and  $\mathbb{E}_{0 \in \mathcal{L}_0}\left[ \Leb(\Lambda_f(0))^{-1} \right]<\infty$. The construction of such an $f$ is done later, at the end of the proof.\\
}

\tvc{Applying part (i) of Theorem \ref{Thm:controlforward} to $X_0:=f(0)$, since $\mathbb{E}\left[ \|f(0)\|^{2p} \right]<\infty$, we obtain
\begin{eqnarray}
\label{E:forw2eq1}
\limsup_{t \to \infty}\mathbb{E}\left[ \left(\frac{\CFD_0^t(f(0))}{e^t} \right)^{2p}\right]<\infty.
\end{eqnarray}
If we set
\begin{eqnarray*}
w(x,\eta,\xi)=\left(\frac{\CFD_0^t(f(x,\eta,\xi))(\eta)}{e^t}\right)^{2p},
\end{eqnarray*}
then $(f,w)$ is a level $0$-weighted association function. So using mass transport, we have by Proposition \ref{Lem:MTlemma2},
\begin{eqnarray}
\label{E:forw2eq2}
&\mathbb{E}\left[ \left( \frac{\CFD_0^t(f(0))}{e^t} \right)^{2p} \right]&=\alpha_0\mathbb{E}_{0 \in \mathcal{L}_0}\left[ \int_{\Lambda_f(0)} \left( \frac{\CFD_0^t(f(x))}{e^t} \right)^{2p}~dx \right] \nonumber\\
&&=\alpha_0\mathbb{E}_{0 \in \mathcal{L}_0}\left[ \Leb(\Lambda_f(0)) ~ \left(\frac{\CFD_0^t(0)}{e^t}\right)^{2p} \right].
\end{eqnarray}
Hence \eqref{E:forw2eq1} and \eqref{E:forw2eq2} provide
\begin{equation}
\limsup_{t\rightarrow \infty} \mathbb{E}_{0 \in \mathcal{L}_0}\left[ \Leb(\Lambda_f(0)) ~ \left(e^{-t}\CFD_0^t(0)\right)^{2p} \right]<\infty. 
\end{equation}We can hence obtain the result (\ref{E:conclcontrolbackwardii}) announced in part (ii) is obtained by Cauchy-Schwarz inequality under
the assumption $\mathbb{E}_{0 \in \mathcal{L}_0}\left[ \Leb(\Lambda_f(0))^{-1} \right]<\infty$:
\begin{multline}
\label{E:forw2CS}
\limsup_{t\rightarrow \infty}\mathbb{E}_{0 \in \mathcal{L}_0}\left[ \left( \frac{\CFD_0^t(0)}{e^t} \right)^p \right] \\
\le \limsup_{t\rightarrow \infty} \mathbb{E}_{0 \in \mathcal{L}_0}\left[ \Leb(\Lambda_f(0))\left(\frac{\CFD_0^t(0)}{e^t}\right)^{2p} \right]^{1/2} \mathbb{E}_{0 \in \mathcal{L}_0}\left[ \Leb(\Lambda_f(0))^{-1} \right]^{1/2}<\infty.
\end{multline}
}
To complete the proof, it remains to show that there exists a level $0$-association function $f$ such that $\mathbb{E}\left[ \|f(0)\|^{2p}\right]<\infty$ and $\mathbb{E}_{0 \in \mathcal{L}_0}\left[ \Leb(\Lambda_f(0))^{-1} \right]<\infty$. Let $\Upsilon=\mathbb{R}^d/\mathbb{Z}^d$ be the $d$-dimensional torus, and let $Y$ be a random variable independent of $N$ and uniformly distributed on $\mathbb{R}^d/\mathbb{Z}^d$.

Now, fix $\eta \in \mathcal{N}_S$ and $\xi \in \mathbb{R}^d/\mathbb{Z}^d$. Let us construct $f(x,\eta,\xi)$ for all $x \in \mathbb{R}^d$ as follows. We pave $\mathbb{R}^d$ by cubes of size $1$ such that (any representative for) $\xi$ is a node of the grid. More precisely, let $u=(u_1,...,u_d) \in \mathbb{R}^d$ be a representative for $\xi$, and define
\begin{eqnarray*}
\mathcal{K}(\xi)=\left\{\prod_{i=1}^d [u_i+a_i,u_i+a_i+1),~a=(a_1,...,a_d) \in \mathbb{Z}^d \right\}.
\end{eqnarray*}
Clearly, this definition does not depend of the choice of the representative $u$, so $\mathcal{K}(\xi)$ is well-defined.
We construct $f(\cdot,\eta,\xi)$ separately on each cube $K \in \mathcal{K}(\xi)$. Let $K \in \mathcal{K}(\xi)$ and let $b=(b_1,...,b_d) \in \mathbb{R}^d$ be the vertex of $K$ with the lowest coordinates (that is, such that $K=\prod_{i=1}^d[b_i,b_i+1)$). Let $n(K):=\#(\mathcal{L}_0(\eta) \cap K)$ be the number of $0$-level points in $K$. If $n(K)=0$, then, for all $x \in K$, we set $f(x,\eta,\xi)$ to be the point of $\mathcal{L}_0(\eta)$ the closest to $x$ (in case of equality pick, say, the smallest for the lexicographical order). Now suppose $n(K) \ge 1$. We divide $K$ into $n(K)$ equal slices: for $1 \le j \le n(K)$, we set
\begin{eqnarray*}
&&\mathcal{S}_j(K)=\prod_{i=1}^{d-1} \big[b_i,b_i+1) \times \left[b_d+\frac{j-1}{n(K)},b_d+\frac{j}{n(K)}\right).
\end{eqnarray*}
Let $x_1,...,x_{n(K)}$ be the $n(K)$ points of $\mathcal{L}_0(\eta) \cap K$ (in, say, the lexicographical order). For $1 \le j \le n(K)$, we send the slice $\mathcal{S}_j(K)$ on $x_j$, i.e. for all $x \in \mathcal{S}_j(K)$, we set $f(x,\eta,\xi)=x_j$. 

We now show that $f$ is a level $0$-association function. First, for all $x \in \mathbb{R}^d$, $\eta \in \mathcal{N}_S$, $\xi \in \Upsilon$, $f(x) \in \mathcal{L}_0$ by construction. We now check the translation invariance. Let $x,x' \in \mathbb{R}^d$ and $\eta \in \mathcal{N}_S$. By construction, for all $u \in \mathbb{R}^d$, $f(x+x',T_{x'}\eta,\overline{u+x'})=f(x,\eta,\overline{u})+\overline{x'}$, where $\overline{x}$ denotes the class of $x$ in $\mathbb{R}^d/\mathbb{Z}^d$. Then, since $Y \overset{(d)}{=} Y+\overline{u}$, $f(x+x',T_{x'}\eta,Y)=f(x,\eta,Y)+x'$, so $f$ is an association function.

We move on to show that $\mathbb{E}\left[ \|f(0)\|^{2p} \right]<\infty$. By construction $f(0)$ is either the point of $\mathcal{L}_0$ the closest to $0$, or a point of the cube $K_0$ containing the origin. Thus, almost surely, $\|f(0)\| \le \min_{x \in \mathcal{L}_0}\|x\| \vee \sqrt{d}$.
By Lemma \ref{Lem:closest}, $\mathbb{E}\left[ (\min_{x \in \mathcal{L}_0}\|x\|)^{2p} \right]<\infty$ therefore $\mathbb{E}\left[ \|f(0)\|^{2p} \right]<\infty$. 

We finally show that $\mathbb{E}_{0 \in \mathcal{L}_0}\left[ \Leb(\Lambda_f(0))^{-1} \right]<\infty$. For $x \in \mathcal{L}_t$, we note $K(x)$ the (random) cube of $\mathcal{K}(Y)$ containing $x$. By construction, almost surely, $\Lambda_f(0)$ contains at least a slice of volume $1/n(K(0))$ (plus eventually additional points contained in empty cubes), therefore
\begin{eqnarray}
\label{E:forw2lebtok}
\Leb(\Lambda_f(0)) \ge \frac{1}{n(K(0))} \implies \mathbb{E}_{0 \in \mathcal{L}_0}\left[ \Leb(\Lambda(0))^{-1} \right] \le \mathbb{E}_{0 \in \mathcal{L}_0}\left[ n(K(0)) \right].
\end{eqnarray}
Let us introduce the weight function $w(x,\eta,\xi):=\mathbf{1}_{x \in \mathcal{L}_0}n(K(0))(\eta,\xi)$. Applying Lemma \ref{Lem:technicallemmageneral} to $w$ with $A=[-1/2,1/2]^d$ leads to
\begin{eqnarray}
\label{E:forw2equahat}
\mathbb{E}\left[ \sum_{x \in \mathcal{L}_0 \cap [-1/2,1/2]^d} n(K(x)) \right]=\alpha_0\mathbb{E}_{0 \in \mathcal{L}_0}\left[ n(K(0)) \right].
\end{eqnarray}
Since a.s. $K(x) \subset [-3/2,3/2]^d$ for all $x \in [-1/2,1/2]^d$,
\begin{eqnarray}
\label{E:forw2final}
&\mathbb{E}\left[ \sum_{x \in \mathcal{L}_0 \cap [-1/2,1/2]^d} n(K(x)) \right] &\le \mathbb{E}\left[\#(\mathcal{L}_0 \cap [-1/2,1/2]^d)~\#(\mathcal{L}_0 \cap [-3/2,3/2]^d) \right] \nonumber\\
&&\le \mathbb{E}\left[ \#(\mathcal{L}_0 \cap [-3/2,3/2]^d)^2 \right] \nonumber\\
&&<\infty \text{ by Proposition \ref{Prop:nop}}.
\end{eqnarray}
Combining (\ref{E:forw2lebtok}), (\ref{E:forw2equahat}) and (\ref{E:forw2final}), we obtain that $\mathbb{E}_{0 \in \mathcal{L}_0}[\Leb(\Lambda_f(0))^{-1}]<\infty$. This completes the proof.

\subsection{Proof of Theorem \ref{Thm:controlbackward}}
\label{S:proofcontrolbackward}

In order to prove Theorem \ref{Thm:controlbackward}, we will use that, for all $a,t \ge 0$ and $p \ge 1$:
\begin{eqnarray}
\label{E:propagback}
\mathbb{E}_{0 \in \mathcal{L}_0}\left[ \MBD_{-t-a}^0(0)^p \right] \le 2^{p-1}e^{a(d-p)}\left(\mathbb{E}_{0 \in \mathcal{L}_0}\left[ \MBD_{-t}^0(0)^p \right]+\mathbb{E}_{0 \in \mathcal{L}_0}\left[ \CFD_0^a(0)^p \right] \right).
\end{eqnarray}

\paragraph{Step 1:} we prove Theorem \ref{Thm:controlbackward} assuming (\ref{E:propagback}).
We can suppose $p>d$ since the result for $p>d$ immediately implies the result for all $p \ge 1$. Then $2^{p-1}e^{a(d-p)} \to 0$ when $a \to \infty$, so we can choose $a_0>0$ such that $2^{p-1}e^{a_0(d-p)}<1$. For $n \in \mathbb{N}$, we define:
\begin{eqnarray*}
S_n:=\sup_{\substack{t \in [na_0,\\(n+1)a_0]}}\mathbb{E}_{0 \in \mathcal{L}_0}\left[ \MBD_{-t}^0(0)^p \right].
\end{eqnarray*}
Then we \tvc{prove that} $\limsup_{n \to \infty} S_n<\infty$.
Let $s \in [0,a_0]$. Applying (\ref{E:propagback}) with $t=0$ and $a=s$ leads to
\begin{eqnarray*}
\mathbb{E}_{\tvc{0 \in \mathcal{L}_0}}\left[ \MBD_{-s}^0(0)^p \right] \le 2^{p-1}e^{s(d-p)}\mathbb{E}_{0 \in \mathcal{L}_0}\left[ \CFD_0^s(0)^p \right]
\end{eqnarray*}
since $\mathbb{E}_{0 \in \mathcal{L}_0}\left[ \MBD_0^0(0)^p \right]=0$. Using the fact that $\mathbb{P}_{0 \in \mathcal{L}_0}$-a.s., the function $s \mapsto \CFD_0^s(0)$ is non-decreasing and that $e^{s(d-p)} \le 1$, we obtain
\begin{eqnarray}
\label{E:backfs}
\forall s \in [0,a_0],~\mathbb{E}_{0 \in \mathcal{L}_0}\left[ \MBD_{-s}^0(0)^p \right] \le 2^{p-1}\mathbb{E}_{0 \in \mathcal{L}_0}\left[ \CFD_0^{a_0}(0)^p \right]<\infty
\end{eqnarray}
since $\mathbb{E}_{0 \in \mathcal{L}_0}\left[ \CFD_0^{a_0}(0)^p \right]<\infty$ by (ii) in Theorem \ref{Thm:controlforward}. Therefore $S_0<\infty$.

Let $n \in \mathbb{N}$. Then
\begin{eqnarray}
\label{E:ntonplusun}
&&S_{n+1}=\sup_{\substack{t \in [na_0,\\(n+1)a_0]}} \mathbb{E}_{0 \in \mathcal{L}_0}\left[ \MBD_{-t-a_0}^0(0)^p \right] \nonumber\\
&&\overset{(\ref{E:propagback})}{\le} \sup_{\substack{t \in [na_0,\\(n+1)a_0]}}2^{p-1}e^{a_0(d-p)}\left(\mathbb{E}_{0 \in \mathcal{L}_0}\left[ \MBD_{-t}^0(0)^p \right]+\mathbb{E}_{0 \in \mathcal{L}_0}\left[ \CFD_0^{a_0}(0)^p \right] \right) \nonumber\\
&&=2^{p-1}e^{a_0(d-p)}\left(\sup_{\substack{t \in [na_0,\\(n+1)a_0]}}\mathbb{E}_{0 \in \mathcal{L}_0}\left[ \MBD_{-t}^0(0)^p \right]+\mathbb{E}_{0 \in \mathcal{L}_0}\left[ \CFD_0^{a_0}(0)^p \right] \right) \nonumber\\
&&=\varphi(S_n)
\end{eqnarray}
where $\varphi:\mathbb{R}_+ \rightarrow \mathbb{R}_+$ is defined as
\begin{eqnarray*}
\varphi(t)=2^{p-1}e^{a_0(d-p)}\big(t+\mathbb{E}_{0 \in \mathcal{L}_0}\left[ \CFD_0^{a_0}(0)^p \right] \big).
\end{eqnarray*}
The function $\varphi$ is well-defined since $\mathbb{E}_{0 \in \mathcal{L}_0}\left[ \CFD_0^{a_0}(0)^p \right]<\infty$. By iterating (\ref{E:ntonplusun}), since $\varphi$ is non-decreasing, we get $S_n \le \varphi^n(S_0)$, where $\varphi^n=\varphi \circ ... \circ \varphi$ $n$ times. Since $2e^{a_0(d-p)}<1$, $\varphi$ is a contraction linear mapping, it admits a finite \tvc{fixed} point $t_0$ and $\varphi^n(S_0) \to t_0$. Therefore
\begin{eqnarray*}
\limsup_{t \to \infty} \mathbb{E}_{0 \in \mathcal{L}_0}\left[ \MBD_{-t}^0(0)^p \right]= \limsup_{n \to \infty} S_n \le \limsup_{n \to \infty} \varphi^n (S_0)=t_0<\infty.
\end{eqnarray*}
This proves Theorem \ref{Thm:controlbackward}. 

\bigbreak

\paragraph{Step 2:} we show (\ref{E:propagback}). \\
Let $a,t \ge 0$ and $p \ge 1$. For $x \in \mathcal{L}_0$, we have
\begin{eqnarray*}
&\MBD_{-t-a}^0(x)&=\max_{x'' \in \mathcal{D}_{-t-a}^0(x)} \CFD_{-t-a}^0(x'') \nonumber\\
&&=\max_{x'' \in \mathcal{D}_{-t-a}^0(x)} \left(\CFD_{-t-a}^{-a}(x'')+\CFD_{-a}^0(\mathcal{A}_{-t-a}^{-a}(x'')\right) \nonumber\\
&&=\max_{x' \in \mathcal{D}_{-a}^0(x)}\max_{x'' \in \mathcal{D}_{-t-a}^{-a}(x')} \left( \CFD_{-t-a}^{-a}(x'')+\CFD_{-a}^0(x') \right) \nonumber\\
&&=\max_{x' \in \mathcal{D}_{-a}^0(x)} \left(\MBD_{-t-a}^{-a}(x')+\CFD_{-a}^0(x')\right).
\end{eqnarray*}
Therefore
\begin{eqnarray}
\label{E:backinequa0}
&&\mathbb{E}_{0 \in \mathcal{L}_0}\left[ \MBD_{-t-a}^0(0)^p \right] \\
&&=\mathbb{E}_{0 \in \mathcal{L}_0}\left[\max_{x \in \mathcal{D}_{-a}^0(0)} \left(\MBD_{-t-a}^{-a}(x)+\CFD_{-a}^0(x) \right)^p \right] \nonumber\\
&&\le \mathbb{E}_{0 \in \mathcal{L}_0}\left[\sum_{x \in \mathcal{D}_{-a}^0(0)} \left(\MBD_{-t-a}^{-a}(x)+\CFD_{-a}^0(x) \right)^p \right] \nonumber\\
&&\le \mathbb{E}_{0 \in \mathcal{L}_0}\left[\sum_{x \in \mathcal{D}_{-a}^0(0)} 2^{p-1}\left( \MBD_{-t-a}^{-a}(x)^p+\CFD_{-a}^0(x)^p\right) \right] \nonumber\\
&&=2^{p-1} \left(\mathbb{E}_{0 \in \mathcal{L}_0}\left[\sum_{x \in \mathcal{D}_{-a}^0(0)} \MBD_{-t-a}^{-a}(x)^p \right]+\mathbb{E}_{0 \in \mathcal{L}_0}\left[\sum_{x \in \mathcal{D}_{-a}^0(0)}\CFD_{-a}^0(x) ^p \right] \right)\nonumber,
\end{eqnarray}
where \tvc{convexity} was used in the second inequality. Now, we use the Mass Transport Principle to rewrite the quantities $\mathbb{E}_{0 \in \mathcal{L}_0}\left[\sum_{x \in \mathcal{D}_{-a}^0(0)} \MBD_{-t-a}^{-a}(x)^p \right]$ and $\mathbb{E}_{0 \in \mathcal{L}_0}\left[\sum_{x \in \mathcal{D}_{-a}^0(0)}\CFD_{-a}^0(x) ^p \right]$. Let us introduce the two weight functions $w_1$ and $w_2$ defined as
\begin{eqnarray*}
w_1(x,\eta)=\mathbf{1}_{x \in \mathcal{L}_{-a}(\eta)}\MBD_{-t-a}^{-a}(x)^p(\eta), \quad w_2(x,\eta)=\mathbf{1}_{x \in \mathcal{L}_{-a}(\eta)}\CFD_{-a}^0(x)^p(\eta).
\end{eqnarray*}
Applying Proposition \ref{Lem:MTlemma1} to $w_1$ with $t_1=-a,t_2=0$ leads to:
\begin{eqnarray}
\label{E:cont1w1}
\mathbb{E}_{0 \in \mathcal{L}_{0}}\left[\sum_{x \in \mathcal{D}_{-a}^0(0)} \MBD_{-t-a}^{-a}(x)^p \right]=e^{ad}\mathbb{E}_{0 \in \mathcal{L}_{-a}}\left[ \MBD_{-t-a}^{-a}(0)^p \right].
\end{eqnarray}
For all $\eta \in \mathcal{N}_S$ such that $0 \in \mathcal{L}_{-a}[\eta]$, we have
\begin{eqnarray*}
\MBD_{-t-a}^{-a}(0)(\eta)=e^{-a}\MBD_{-t}^0(0)(D_{e^a}\eta),
\end{eqnarray*}
so by scale invariance (Lemma \ref{Lem:scaleinv} applied with $t=-a,t'=0$),
\begin{eqnarray}
\label{E:cont2w1}
\mathbb{E}_{0 \in \mathcal{L}_{-a}}\left[ \MBD_{-t-a}^{-a}(0)^p \right]=e^{-ap}\mathbb{E}_{0 \in \mathcal{L}_0}\left[ \MBD_{-t}^0(0)^p \right].
\end{eqnarray}
Combining (\ref{E:cont1w1}) and (\ref{E:cont2w1}), we obtain:
\begin{eqnarray}
\label{E:contfw1}
\mathbb{E}_{0 \in \mathcal{L}_{0}}\left[\sum_{x \in \mathcal{D}_{-a}^0(0)} \MBD_{-t-a}^{-a}(x)^p \right]=e^{a(d-p)}\mathbb{E}_{0 \in \mathcal{L}_0}\left[ \MBD_{-t}^0(0)^p \right].
\end{eqnarray}
The same calculations with $w_2$ lead to
\begin{eqnarray}
\label{E:contfw2}
\mathbb{E}_{0 \in \mathcal{L}_{0}}\left[\sum_{x \in \mathcal{D}_{-a}^0(0)} \CFD_{-a}^0(x)^p \right]=e^{a(d-p)}\mathbb{E}_{0 \in \mathcal{L}_0}\left[ \CFD_0^a(0)^p \right].
\end{eqnarray}
Finally, we rewrite (\ref{E:backinequa0}) using (\ref{E:contfw1}) and (\ref{E:contfw2}), we obtain (\ref{E:propagback}). It completes the proof.

\section{Proof of coalescence}
\label{S:coalescencealld}

In this section we prove Theorem \ref{Thm:coalescence}.

\subsection{A short proof in dimension $1+1$}

We first prove Theorem \ref{Thm:coalescence} in the bi-dimensional case (i.e. $d=1$). It is based on planarity, so it only works for $d=1$. A general (but more complex) proof of coalescence in all dimensions is given after. A useful consequence of planarity we shall need is the following:
\begin{lemma}
\label{Cor:noncrossing}
Suppose $d=1$. Let $t \ge 0$ and $x \in \mathcal{L}_t$. If $x_1,x_2,x_3 \in \mathcal{L}_0$ are such that $x_1<x_2<x_3$, and if $x_1,x_3 \in \mathcal{D}_0^t(x)$, then $x_2 \in \mathcal{D}_0^t(x)$.
\end{lemma}

\begin{proof}[Proof of Lemma \ref{Cor:noncrossing}]
Let $x_1,x_2,x_3 \in \mathcal{L}_0$ such that  $x_1<x_2<x_3$, and $x_1,x_3 \in \mathcal{D}_0^t(x)$. Since the DSF is non-crossing \cite[Section 3.1]{version_longue}, the trajectory from $(x_2,e^0)$ cannot cross the trajectories from $(x_1,e^0)$ and $(x_3,e^0)$. The point $(x,e^h)$ belongs to both trajectories from $(x_1,e^0)$ and $(x_3,e^0)$, so it also belongs to the trajectory from $(x_2,e^0)$, so $x_2 \in \mathcal{D}_0^h(x)$.
\end{proof}

For $t \ge 0$, we define $B_t$ as the set of points of level $t$ that have descendants at level $0$; for $x \in B_t$, we define $M_t(x)$ as the left-most descendant of $x$ at level $0$:
\begin{eqnarray*}
B_t:=\{x \in \mathcal{L}_t,~\mathcal{D}_0^t(x) \neq \emptyset\}, \quad M_t(x)=\inf \mathcal{D}_0^t(x) \in \{-\infty\} \cup \mathbb{R}.
\end{eqnarray*}
We now prove:
\begin{lemma}
\label{Lem:existsmin}
\tvc{Almost surely, for all $x \in B_t$, there exists $y \in \mathcal{D}_0^t (x)$ such that $M t (x) = y$, that is, each point $x \in B_t$ admits a left-most descendant at level $0$.}
\end{lemma}

\begin{proof}[Proof of lemma \ref{Lem:existsmin}]
Since $\mathcal{D}_0^t(x)$ is locally finite and non empty for all $x \in B_t$, it suffices to show that $\inf \mathcal{D}_0^t(x)>-\infty$  for all $x \in B_t$ a.s. Consider the event
\begin{eqnarray*}
A:=\{\exists x \in B_t, \inf \mathcal{D}_0^t(x)=-\infty\}.
\end{eqnarray*}

On $A$, we show that there is a unique $x \in B_t$ such that $\inf \mathcal{D}_0^t(x)=-\infty$. Indeed, suppose that $\inf \mathcal{D}_0^t(x')=-\infty$ for some $x' \in B_t$. Let $\hat{x} \in \mathcal{D}_0^t(x)$. Pick $\hat{x}' \in \mathcal{D}_0^t(x')$ such that $\hat{x}'<\hat{x}$ (such a $\hat{x}'$ exists since $\inf \mathcal{D}_0^t(x)=-\infty$). Then pick $\hat{x}'' \in \mathcal{D}_0^t(x)$ such that $\hat{x}''<\hat{x}'$. By Lemma \ref{Cor:noncrossing}, $\hat{x}' \in \mathcal{D}_0^t(x)$ which implies $x=x'$. 

Suppose that $\mathbb{P}[A]>0$. Then, conditionally to $A$, we can define $X$ as the (random) unique $x \in B_t$ such that $\mathcal{D}_0^t(x)=-\infty$. Since the event $A$ is invariant by translations, the distribution of $N$ conditioned by $A$ is also invariant by translations. Therefore the law of $X$ must be invariant by translations, but there's no probability distribution on $\mathbb{R}$ invariant by translations. This is a contradiction, therefore $\mathbb{P}[A]=0$.
\end{proof}

We call \emph{level $t$-separating points} the points $M_t(x)$ for $x \in B_t$ . We denote by $S_t:=\{M_t(x),~x \in B_t\}$ the set of level $t$-separating points. Let us prove:
\begin{lemma}
\label{Lem:coalesc}
If $S_t \cap [-a,a]=\emptyset$, then $\tau_{[-a,a]} \le t$.
\end{lemma}

\begin{proof}
Suppose that $S_t \cap [-a,a]=\emptyset$. Let $x,x' \in [-a,a]$ with $x<x'$, and suppose that $\mathcal{A}_0^t(x) \neq \mathcal{A}_0^t(x')$. Consider $\hat{x}=M_t(\mathcal{A}_0^t(x'))$. Suppose that $\hat{x} \le x$. Thus $\hat{x} \le x<x'$, and by construction $\hat{x}$ and $x'$ have the same ancestor at level $t$. Then by Lemma \ref{Cor:noncrossing}, $x \in \mathcal{D}_0^t(\mathcal{A}_0^t(x'))$, which contradicts $x \neq x'$. Therefore $\hat{x}>x$. Moreover $\hat{x} \le x'$ by construction, so $x<\hat{x}<x'$, therefore $\hat{x} \in [-a,a]$. But $\hat{x} \in S_t$, this contradicts $S_t \cap [-a,a]=\emptyset$. Thus for all $x,x' \in [-a,a]$, $\mathcal{A}_0^t(x)=\mathcal{A}_0^t(x')$, which implies $\tau_{[-a,a]} \le t$.
\end{proof}

We show that the level $t$- separating points are rare when $t$ is large. We apply the Mass Transport Principle (Theorem \ref{Thm:masstranport}) on $\mathbb{R}$ with the following mass transport: from each point with descendants at level $0$ we transport a unit mass to its left-most descendant.

The following measure $\pi$ expresses this mass transport:
\begin{eqnarray*}
\pi(A \times B)=\mathbb{E}\left[ \sum_{x \in B_t \cap A} \mathbf{1}_{M_t(x) \in B} \right].
\end{eqnarray*}

The measure $\pi$ is diagonally invariant because horizontal translations preserve the DSF's distribution. Then, by the Mass Transport Principle, \tvc{$\pi(A \times \mathbb{R})=\pi(\mathbb{R} \times A)$} for all $A \subset \mathbb{R}$ with non-empty interior. On the one hand,
\begin{eqnarray}
\label{E:mt1}
\pi(A \times \mathbb{R})=\mathbb{E}\left[ \#(B_t \cap A) \right] \le \mathbb{E}\left[\#(\mathcal{L}_h \cap A)\right]=\alpha_0e^{-t}\Leb(A)
\end{eqnarray}
and, on the other hand,
\begin{eqnarray}
\label{E:mt2}
\pi(\mathbb{R} \times A)=\mathbb{E}\left[ \sum_{x \in B_t} \mathbf{1}_{M_t(x) \in A} \right]=\mathbb{E}\left[ \sum_{\substack{x \in B_t, \\ x' \in A}} \mathbf{1}_{x'=M_t(x)}\right]=\mathbb{E}\left[ \#(S_t \cap A) \right].
\end{eqnarray}

Therefore, combining (\ref{E:mt1}) and (\ref{E:mt2}) with $A=[-a,a]$, we obtain $\mathbb{E}[\#(S_t \cap [-a,a])] \le 2a\alpha_0e^{-t}$. Hence
\begin{eqnarray*}
\mathbb{P}[S_t \cap [-a,a] \neq \emptyset] \le \mathbb{E}[\#(S_t \cap [-a,a])] \le 2a \alpha_0e^{-t}.
\end{eqnarray*}

By Lemma \ref{Lem:coalesc}, it implies that $\mathbb{P}[\tau_{[-a,a]}>t] \le 2a \alpha_0 e^{-t}$,
which proves the second statement of Theorem \ref{Thm:coalescence}. The fact that the DSF is almost surely a tree immediately follows. Indeed,
\begin{eqnarray*}
\{\text{The DSF is a tree} \}=\{\forall a \in \mathbb{N},~\exists t \in \mathbb{N},~\tau_{[-a,a]} \le t\}=\bigcap_{a \in \mathbb{N}} \downarrow \bigcup_{t \in \mathbb{N}} \uparrow \{\tau_{[-a,a]} \le t\}.
\end{eqnarray*}
Therefore
\begin{eqnarray*}
\mathbb{P}[\text{The DSF is a tree}]=\lim_{a \in \mathbb{N}} \downarrow \lim_{t \in \mathbb{N}} \uparrow \mathbb{P}[\tau_{[-a,a]} \le t].
\end{eqnarray*}
Since for all $a>0$, $\mathbb{P}[\tau_{[-a,a]} \le t] \ge 1-2a\alpha_0e^{-t} \underset{t \rightarrow \infty}{\longrightarrow} 1$, we obtain
\begin{eqnarray*}\mathbb{P}[\text{The DSF is a tree}]=\lim_{a \in \mathbb{N}} \downarrow 1=1,\end{eqnarray*}
which proves Theorem \ref{Thm:coalescence} for $d=1$.

\subsection{General case: ideas of the proof}

We move on to show the coalescence for all dimensions $d$. Let us consider two trajectories starting from level $0$. The choice of those trajectories will be discussed later. We want to prove that those two trajectories coalesce.

The intuition behind the coalescence can be understood as follows. We can deduce from Theorem \ref{Thm:controlforward} that the two trajectories stay almost in a cone. That is, for $A$ large enough and for all height $h$ large enough, their projection on $\mathbb{R}^d$ at height $h$ are contained in $\mathbb{B}_{\mathbb{R}^d}(0,Ae^h)$ with high probability. That is, they stay close to each other as they go up. Then, at each height, they have a positive probability to coalesce, so they must coalesce.

This is true because the metric of $H$ becomes larger as the ordinate increases, so this behavior is specific to the Hyperbolic geometry. In Euclidean space, the two trajectories move away from each other as they go up, so the same argument cannot be used. Indeed, we expect that the DSF in $\mathbb{R}^d$ with $d \ge 4$ does not coalesce.

\bigbreak

The idea of the proof is the following. We suppose by contradiction that the two trajectories do not coalesce with positive probability. We consider some height $h$ large enough such that, with high probability, the process $N$ below height $h$ almost determines if the two trajectories coalesce or not. \tvc{Then, on the event of non-coalescence and apart from an event of small probability, the probability of coalescence conditionally to the process $N$ below height $h$ is close to $0$.}

On the other hand, with high probability, for some large fixed $A>0$, both trajectories are contained in the cylinder $B_{\mathbb{R}^d}(0,Ae^h) \times (0,e^h)$ (that is, the two trajectories are not too far from each other). Thus we show that we can modify the process above height $h$ in a way that forces the two trajectories to coalesce, and we show that the set of configurations above height $h$ that forces coalescence has probability bounded below independently of $h$. This contradicts the fact that the probability of coalescence knowing the process $N$ below height $h$ \tvc{is} close to $0$ with macroscopic probability.

Some technical difficulties are due to the geometry of the model and the fact that a modification of the point process above height $h$ can affect trajectories below height $h$.

\subsection{Introduction and notations}

Let $d \ge 1$. The following notations are illustrated in Figure \ref{Fig:coalescence}. Let $p_1,p_2 \in \mathbb{Q}^d \times (\mathbb{Q} \cap (0,e^0))$ two points of $H$ below level $0$ with rational coordinates. We define $Z_1$ (resp. $Z_2$) as the (random) point of $N \cap (\mathbb{R}^d \times (0,e^0))$ the closest to $p_1$ (resp. $p_2$):
\begin{eqnarray*}
Z_i=\argmin_{\substack{(x,y) \in N, \\y<1}}d(p_i,Z_i).
\end{eqnarray*}
for $i=1,2$. We will prove that the trajectories from $Z_1$ and $Z_2$ coalesce almost surely, i.e. that a.s. there exists $n \ge 0$ such that $A^n(Z_1)=A^n(Z_2)$, where $A^n=A \circ ... \circ A$ $n$ times (recall that $A(z)$ denotes the parent of $z\in N$). If this is proved, then the result will be true almost surely for all $p_1,p_2 \in \mathbb{Q}^d \times (\mathbb{Q}\cap(0,e^0))$ simultaneously, which implies that the whole DSF coalesces a.s.

\bigbreak

For $t \ge 0$, define $k_i(t)$ as the unique non-negative integer such that $[A^{k_i(t)}Z_i,A^{k_i(t)+1}Z_i]_{eucl}$ crosses the level $t$. It is well defined a.s. because each trajectory starting below the level $t$ crosses the level $t$ a.s.

Let $A,h,M,\delta,\varepsilon>0$ be five parameters that will be chosen later. We define
\begin{eqnarray}
&K(M,h,\delta):&=\{z=(x,y) \in H,~y \ge e^{h-\delta} \text{ and } d(z,(0,e^h))<M\}\nonumber\\
&&=B_H((0,e^h),M) \cap (\mathbb{R}^d \times (e^{h-\delta},\infty)).\label{def:K(Mhdelta)}
\end{eqnarray}
Note that $K(M,h,\delta)=D_{e^h}K(M,0,\delta)$ so $K(M,h,\delta)$ and $K(M,0,\delta)$ are isometric. 

Let $\mathcal{F}_{in}(M,h,\delta)$ (resp. $\mathcal{F}_{out}(M,h,\delta)$) be the $\sigma$-algebra on $\mathcal{N}_S$ generated by the process $N$ inside $K(M,h,\delta)$ (resp. outside $K(M,h,\delta)$). 

Finally, define
\begin{eqnarray*}
\Slice(A,h,\delta):=B_{\mathbb{R}^d}(0,Ae^h) \times (e^{h-\delta},e^h).
\end{eqnarray*}
Notice that $\Slice(A,h,\delta)=D_{e^h}\Slice(A,0,\delta)$, so $\Slice(A,h,\delta)$ and $\Slice(A,0,\delta)$ are isometric.

We now introduce the following events. The event $\CO$ means that trajectories from $Z_1$ and $Z_2$ coalesce:
\begin{eqnarray*}
\CO:=\{\exists n \ge 0,~A^n(Z_1)=A^n(Z_2)\}.
\end{eqnarray*}
The event $\Cyl(A,h)$ means that below the level $h$, trajectories from $Z_1$ and $Z_2$ are entirely contained in the cylinder $B_{\mathbb{R}^d}(0,Ae^h) \times (0,e^h)$:
\begin{eqnarray*}
\Cyl(A,h):=\{\text{For } i=1,2, \text{ for all } 0 \le n \le k_i(h), \ 
A^n(Z_i) \in B_{\mathbb{R}^d}(0,Ae^h) \times (0,e^h)\}.
\end{eqnarray*}
Let us also define
\begin{eqnarray*}
\EmptySlice(A,h,\delta):=\{N \cap \Slice(A,h,\delta)=\emptyset\},
\end{eqnarray*}
\begin{eqnarray*}
\Approx(M,\varepsilon,\delta,h):=\big\{\big|\mathbb{P}[\CO|\mathcal{F}_{out}(M,h,\delta)]-\mathbf{1}_{\CO} \big|<\varepsilon\big\}.
\end{eqnarray*}

\subsection{Heart of the proof of coalescence}
The proof of coalescence is based on the three following lemmas.

\begin{lemma}
\label{Lem:chooseA}
We have
\begin{eqnarray*}
\lim_{A \to \infty}\liminf_{h \to \infty}\mathbb{P}\left[ \Cyl(A,h) \right]=1.
\end{eqnarray*}
\end{lemma}

\begin{lemma}
\label{Lem:chooseh}
Let $M,\varepsilon,\delta>0$. We have
\begin{eqnarray*}
\lim_{h \to \infty} \mathbb{P}\left[ \Approx(M,\varepsilon,\delta,h) \right]=1.
\end{eqnarray*}
\end{lemma}

\begin{lemma}
\label{Lem:chooseMepsilon}
Let $A,\delta>0$. There exist $M,\varepsilon>0$, such that, for all $h>\delta$:
\begin{eqnarray}
\label{E:implic}
\Cyl(A,h) \cap \EmptySlice(A,h,\delta)  \subset \{ \mathbb{P}\left[\CO|\mathcal{F}_{out}(M,h,\delta)\right]>\varepsilon\}.
\end{eqnarray}
\end{lemma}

Lemmas \ref{Lem:chooseA}, \ref{Lem:chooseh} and \ref{Lem:chooseMepsilon} are proved in Sections \ref{S:proofofchooseA}, \ref{S:proofofchooseh} and  \ref{S:proofofchooseMepsilon} respectively. We assume them for the moment and we prove that trajectories from $Z_1$ and $Z_2$ coalesce almost surely, i.e. $\mathbb{P}\left[ \CO \right]=1$. Let us suppose by contradiction that $\mathbb{P}\left[ \CO \right]<1$. We choose the parameters $A,h,M,\delta,\varepsilon$ as follows:

\begin{itemize}
\item We first choose $A$ large enough such that $\liminf_{h \to \infty} \mathbb{P}\left[ \Cyl(A,h) \right]>1-\mathbb{P}[\CO^c]/3$, it is possible by Lemma \ref{Lem:chooseA}. Then we pick $h_0$ large enough such that for all $h \ge h_0$, $\mathbb{P}\left[ \Cyl(A,h) \right]>1-\mathbb{P}[\CO^c]/3$.
\item Then we choose $\delta>0$ small enough such that, for all $h \ge 0$, $\mathbb{P}[\EmptySlice(A,h,\delta)]>1-\mathbb{P}[\CO^c]/3$. It is possible since $\mathbb{P}[\EmptySlice(A,h,\delta)]$ does not depend on $h$ by dilatation invariance and, for all $h>0$ (recall that $\mu$ is the Hyperbolic volume):
\begin{eqnarray*}
\lim_{\delta \to 0}\mu(\Slice(A,h,\delta))=\mu\left(\bigcap_{\delta \downarrow 0} \downarrow \Slice(A,h,\delta) \right)=\mu(B_{\mathbb{R}^d}(0,Ae^h) \times \{e^h\})=0,
\end{eqnarray*}
so
\begin{eqnarray*}
\mathbb{P}[\EmptySlice(A,h,\delta)]=\exp(-\lambda\mu(\Slice(A,h,\delta)) \underset{\delta \to 0}{\to} 1.
\end{eqnarray*}
\item Then we choose $M,\varepsilon>0$ such that, for all $h>\delta$, inclusion (\ref{E:implic}) holds. It is possible by Lemma \ref{Lem:chooseMepsilon}.
\item We finally choose $h$ large enough (and larger than $h_0$ and $\delta$) such that \\ $\mathbb{P}\left[ \Approx(M,\varepsilon,\delta,h) \right]>1-\mathbb{P}[\CO^c]/3$. It is possible by Lemma \ref{Lem:chooseh}.
\end{itemize}

Define
\begin{eqnarray*}
E:=\Cyl(A,h) \cap \EmptySlice(A,h,\delta) \cap \Approx(M,\varepsilon,\delta,h) \cap \CO^c.
\end{eqnarray*}

With this choice of parameters, all the events $\Cyl(A,h)$, $\EmptySlice(A,h,\delta)$ and $\Approx(M,\varepsilon,\delta,h)$ have probability larger than $1-\mathbb{P}[\CO^c]/3$. Therefore
\begin{eqnarray*}
\mathbb{P}[\Cyl(A,h)^c \cup \EmptySlice(A,h,\delta)^c \cup \Approx(M,\varepsilon,\delta,h)^c \cup \CO]<3 \frac{\mathbb{P}[\CO^c]}{3}+\mathbb{P}[\CO]=1.
\end{eqnarray*}
Then $\mathbb{P}[E]>0$. On $E$, since both $\Approx(M,\varepsilon,\delta,h)$ and $\CO^c$ occur, $\mathbb{P}[\CO|\mathcal{F}_{out}(M,h,\delta)]<\varepsilon$; on the other hand, by (\ref{E:implic}), on $E$, $\mathbb{P}[\CO|\mathcal{F}_{out}(M,h,\delta)]>\varepsilon$. This is a contradiction, therefore the assumption $\mathbb{P}[\CO]<1$ is wrong, so  $\mathbb{P}[\CO]=1$. This proves Theorem \ref{Thm:coalescence}.

\subsection{Proof of Lemma \ref{Lem:chooseA}}
\label{S:proofofchooseA}

The proof is based on Theorem \ref{Thm:controlforward} and Markov inequality. For $i \in \{1,2\}$, define
\begin{eqnarray*}
C_i(A,h):=\{\text{For all } 0 \le n \le k_i(h), A^n(Z_i) \in B_{\mathbb{R}^d}(0,Ae^h) \times (0,e^h)\}.
\end{eqnarray*}
Then $\Cyl(A,h)=C_1(A,h) \cap C_2(A,h)$, therefore it suffices to prove that, for $i=1,2$, \[\lim_{A \to \infty}\liminf_{h \to \infty} \mathbb{P}[C_i(A,h)]=1.\]

Let $i \in \{1,2\}$. Let us define $X_i$ to be the (random) point of $\mathcal{L}_0$ such that $(X_i,e^0)$ belongs to the trajectory from $Z_i$. Now, we write $C_i(A,h)$ as the intersection of two events $C_i^-(A,h)$ and $C_i^+(A,h)$, that means respectively that the trajectory is contained in the cylinder below (resp. above) the level 0. We define
\begin{eqnarray*}
C_i^{-}(A,h):=\{\text{For all } 0 \le n \le k_i(0),~A^n(Z_i) \in  B_{\mathbb{R}^d}(0,Ae^h) \times (0,e^0) \}
\end{eqnarray*}
and
\begin{eqnarray*}
C_i^{+}(A,h):=\{\sup_{0 \le t \le h} \|\mathcal{A}_0^t(X_i)\| \le Ae^h\}.
\end{eqnarray*}
Thus $C_i(A,h)=C_i^-(A,h) \cap C_i^+(A,h)$. Clearly, for all $A>0$, $\lim_{h \to \infty}\mathbb{P}[C_i^-]=1$ because the trajectory from $Z_i$ goes above the level $0$. It remains to show that $\lim_{A \to \infty}\liminf_{h \to \infty} \mathbb{P}[C_i^+(A,h)]=1$.

We would like to apply Theorem $\ref{Thm:controlforward}$ to $X_i$ with, say $p=1$, but this would demand showing $\mathbb{E}[\|X_i\|]<\infty$. A workaround is done by using the following trick. Let $B \ge 0$. Define
\begin{eqnarray*}
X_i^B=\left\{
\begin{aligned}[l|ll]
&X_i & \text{ if } \|X_i\| \le B,\\
&\argmin_{x \in \mathcal{L}_0} \|x\| & \text{otherwise}.
\end{aligned}
\right.
\end{eqnarray*}
Using \tvc{Lemma} \ref{Lem:closest} with $p=1$, we get that $\mathbb{E}\left[ \|X_i^B\| \right]<\infty$.
Thus Theorem \ref{Thm:controlforward} applied to $X_i^B$ with $p=1$ gives,
\begin{eqnarray}
\label{E:finitelimsup}
\limsup_{h \to \infty} \mathbb{E}\left[ \frac{\sup_{0 \le t \le h}\|\mathcal{A}_0^t(X_i^B)\|}{e^h} \right] \le \limsup_{h \to \infty}\mathbb{E}\left[\frac{\|X_i^B\|+\CFD_0^h(X_i^B)}{e^h} \right]<\infty.
\end{eqnarray}
By Markov inequality,
\begin{eqnarray}
\label{E:equawithxim}
&\underset{h \to \infty}{\limsup}~ \mathbb{P}[\sup_{0 \le t \le h} \|\mathcal{A}_0^t(X_i^B)\| > Ae^h] &\le \limsup_{h \to \infty} \frac{\mathbb{E}\left[\sup_{0 \le t \le h}\|\mathcal{A}_0^t(X_i^B)\|\right]}{Ae^h} \nonumber\\
&&=A^{-1}\limsup_{h \to \infty} \mathbb{E}\left[e^{-h}\sup_{0 \le t \le h}\|\mathcal{A}_0^t(X_i^B)\|\right] \nonumber\\
&&\to 0 \text{ when } A \to \infty,
\end{eqnarray}
by (\ref{E:finitelimsup}). 

We now need to replace $X_i^B$ by $X_i$ in (\ref{E:equawithxim}). It will be done by taking $B \to \infty$. For $A,B,h>0$,
\begin{eqnarray*}
\left\{\sup_{0 \le t \le h} \|\mathcal{A}_0^t(X_i^B)\| \le Ae^h \text{ and } X_i^B=X_i\right\} \subset \{C_i^+(A,h)\},
\end{eqnarray*}
so
\begin{eqnarray*}
\mathbb{P}[C_i^+(A,h)] \ge 1-\mathbb{P}[\sup_{0 \le t \le h} \|\mathcal{A}_0^t(X_i^B)\| > Ae^h]-\mathbb{P}[X_i \neq X_i^B].
\end{eqnarray*}
Thus, for all $B \ge 0$,
\begin{eqnarray}
\label{E:equawithxim2}
&\underset{A \to \infty}{\lim}\underset{h \to \infty}{\liminf} ~\mathbb{P}[C_i^+(A,h)] &\ge 1-\lim_{A \to \infty}\limsup_{h \to \infty} \mathbb{P}[\sup_{0 \le t \le h} \|\mathcal{A}_0^t(X_i^B)\| > Ae^h]-\mathbb{P}[X_i^B \neq X_i] \nonumber\\
&& \overset{(\ref{E:equawithxim})}{\ge} 1-\mathbb{P}[X_i^B \neq X_i].
\end{eqnarray}
Since $\mathbb{P}[X_i^B \neq X_i] \le \mathbb{P}[X_i>B] \underset{B \to \infty}{\to} 0$, we obtain the wanted result by taking $B \to \infty$ in (\ref{E:equawithxim2}).
This proves Lemma $\ref{Lem:chooseA}$.

\subsection{Proof of Lemma \ref{Lem:chooseh}}
\label{S:proofofchooseh}
Let $M,\varepsilon,\delta>0$. For $h \ge 0$, we denote by $\mathcal{F}_{h-}$ the $\sigma$-algebra generated by the process $N$ on $\mathbb{R}^d \times (0,e^{h-\delta})$. Since $\mathbb{R}^d \times (0,e^{h-\delta}) \subset K(M,h,\delta)^c$, $\mathcal{F}_{h-} \subset \mathcal{F}_{out}(M,h,\delta)$. Since $\bigcup_{h \uparrow \infty} \uparrow \mathcal{F}_{h-}=\sigma(N)$, the martingale convergence theorem gives,
\begin{eqnarray}
\label{E:martconv}
\lim_{h \to \infty} \mathbb{E}[\mathbf{1}_{\CO}|\mathcal{F}_{h-}]=\mathbf{1}_{\CO} \text{ a.s.}
\end{eqnarray}
\tvc{We define
\begin{eqnarray*}
E_1:=\left\{\Big|\mathbb{E}[\mathbf{1}_{\CO}|\mathcal{F}_{h-}]-\mathbf{1}_{\CO}\Big| \ge \frac{\varepsilon}{2} \right\},
\quad
E_2:=\left\{\Big| \mathbb{E}[\mathbf{1}_{\CO}|\mathcal{F}_{out}(M,h,\delta)]-\mathbb{E}[\mathbf{1}_{\CO}|\mathcal{F}_{h-}] \Big| \ge \frac{\varepsilon}{2} \right\}.
\end{eqnarray*}
By (\ref{E:martconv}), $\lim_{h \to \infty}\mathbb{P}[E_1]=0$. Suppose for the moment that $\lim_{h \to \infty}\mathbb{P}[E_2]=0$. Then 
$\lim_{h \to \infty}\mathbb{P}[E_1^c \cap E_2^c]=1$, and on this set
\begin{eqnarray*}
\Big|\mathbb{E}[\mathbf{1}_{\CO}|\mathcal{F}_{out}(M,h,\delta)]-\mathbf{1}_{\CO}\Big|<\varepsilon.
\end{eqnarray*}
Therefore,
$\lim_{h \to \infty}\left[\Approx(M,\varepsilon,\delta,h) \right]=1$,
which proves Lemma \ref{Lem:chooseh}. \\}

It remains to show that $\lim_{h \to \infty}\mathbb{P}[E_2]=0$. Let us define 
\begin{eqnarray*}
X(h):=\mathbb{E}[\mathbf{1}_{\CO}|\mathcal{F}_{out}(M,h,\delta)].
\end{eqnarray*}
Since $\mathcal{F}_{h-} \subset \mathcal{F}_{out}$, we have $\mathbb{E}[\mathbf{1}_{\CO}|\mathcal{F}_{h-}]=\mathbb{E}[X(h)|\mathcal{F}_{h-}]$, so $E_2$ can be rewritten as $\{|X(h)-\mathbb{E}[X(h)|\mathcal{F}_{h-}]| \ge \varepsilon/2\}$. By Markov inequality,
\begin{eqnarray*}
\mathbb{P}\left[\Big|X(h)-\mathbb{E}[X(h)|\mathcal{F}_{h-}]\Big| \ge \frac{\varepsilon}{2} \Big|\mathcal{F}_{h-}\right] \le \frac{2\mathbb{E}\left[ \Big|X(h)-\mathbb{E}[X(h)|\mathcal{F}_{h-}]\Big|~\Big|\mathcal{F}_{h-}\right]}{\varepsilon} \text{ a.s.}
\end{eqnarray*}
So, on the one hand, by triangular inequality and because $X(h) \ge 0$ a.s.,
\begin{eqnarray*}
\mathbb{P}\left[E_2|\mathcal{F}_{h-}\right] \le \frac{2\mathbb{E}\left[X(h)+\mathbb{E}[X(h)|\mathcal{F}_{h-}]\Big|\mathcal{F}_{h-}\right]}{\varepsilon}=\frac{4\mathbb{E}[X(h)|\mathcal{F}_{h-}]}{\varepsilon} \text{ a.s.}
\end{eqnarray*}
On the other hand, again by triangular inequality and because $1-X(h) \ge 0$ a.s.,
\begin{eqnarray*}
\mathbb{P}\left[E_2|\mathcal{F}_{h-}\right] \le \frac{2\mathbb{E}\left[ (1-X(h))+(1-\mathbb{E}[X(h)|\mathcal{F}_{h-}]) \Big|\mathcal{F}_{h-} \right]}{\varepsilon}=\frac{4\mathbb{E}\left[ 1-X(h)|\mathcal{F}_{h-} \right]}{\varepsilon} \text{ a.s.}
\end{eqnarray*}
Thus
\begin{eqnarray*}
\mathbb{P}\left[E_2 \Big|\mathcal{F}_{h-}\right] \le \frac{4\big(\mathbb{E}[X(h)|\mathcal{F}_{h-}] \wedge (1-\mathbb{E}[X(h)|\mathcal{F}_{h-}])\big)}{\varepsilon} \text{ a.s.}
\end{eqnarray*}
Since $\mathbb{E}[X(h)|\mathcal{F}_{h-}]=\mathbb{E}[\mathbf{1}_{\CO}|\mathcal{F}_{h-}] \to \mathbf{1}_{\CO}$ when $h \to \infty$,
\begin{eqnarray*}
\frac{4\big(\mathbb{E}[X(h)|\mathcal{F}_{h-}] \wedge (1-\mathbb{E}[X(h)|\mathcal{F}_{h-}])\big)}{\varepsilon} \le \frac{4\big|\mathbb{E}[X(h)|\mathcal{F}_{h-}]-\mathbf{1}_{\CO}\big|}{\varepsilon} \underset{h \to \infty}{\to} 0 \text{ a.s.}
\end{eqnarray*}
Therefore $\lim_{h \to \infty}\mathbb{P}[E_2|\mathcal{F}_{h-}]=0$ a.s., so dominated convergence theorem gives that $\lim_{h \to \infty} \mathbb{P}[E_2]=0$. This completes the proof of Lemma \ref{Lem:chooseh}.

\subsection{Proof of Lemma \ref{Lem:chooseMepsilon}}
\label{S:proofofchooseMepsilon}

Let us introduce the following notation. For $\eta \in  \mathcal{N}_S$, we define
\begin{eqnarray*}
\eta_{in}:=\eta \cap K(M,h,\delta), \quad \eta_{out}=\eta \cap K(M,h,\delta)^c,
\end{eqnarray*}where we recall that $K(M,h,\delta)$ has been defined in \eqref{def:K(Mhdelta)}.
In particular, $N_{in}=N \cap K(M,h,\delta)$ and $N_{out}=N \cap K(M,h,\delta)^c$, are two independent PPPs of intensity $\lambda$ on $K(M,h,\delta)$ and $K(M,h,\delta)^c$ respectively. 

Let $A,\delta>0$ and consider $\eta \in \Cyl(A,h) \cap \EmptySlice(A,h,\delta)$. The idea of the proof is to build an event $\FC(\eta)$ on $\mathcal{N}_S(K(M,h,\delta))$, of probability bounded below by some $\varepsilon>0$ independent of $\eta$, such that, when we replace the process $\eta$ inside the box $K(M,h,\delta)$ by some $\eta'_{in} \in \FC(\eta)$, then we force trajectories from $Z_1$ and $Z_2$ to coalesce.

For $i=1,2$, consider the point $(x_i,y_i)=A^{k_i(h)}(Z_i)(\eta)$ (i.e. the highest point of $N$ in the trajectory below the level $h$). Notice that $\|x_i\| \le Ae^h$ since $\eta \in \Cyl(A,h)$. We define three balls as follows:
\begin{eqnarray*}
B_i^{down}=B_{H}\left((x_i,e^{h-\delta/2}),\frac{\delta}{2}\right) \text{ for } i=1,2, \quad B^{up}=B_{H}\left((0,e^{h+\delta/2}),\frac{\delta}{2}\right).
\end{eqnarray*}
We now make the choice of $M$. For $h \ge 0$, define
\begin{eqnarray*}
\Xi(h):=\left\{z \in H,~d\left(z,B_{\mathbb{R}^d}(0,Ae^h) \times \{e^{h-\delta/2}\}\right)<\delta/2 \right\}.
\end{eqnarray*}
So $\Xi(h)$ is a \tvc{relatively} compact set and $\Xi(h)=D_{e^h}\Xi(0)$, therefore $\Xi(h)$ and $\Xi(0)$ are isometric. By construction, since $\|x_1\|,\|x_2\| \le Ae^h$, $B_1^{down},B_2^{down} \subset \Xi(h)$. 
Let us pick $M$ large enough such that, for all $z^{down} \in \Xi(0)$ and for all $z^{up} \in B_H((0,e^{\delta/2}),\delta/2)$,
\begin{eqnarray*}
B_H(z^{down},d(z^{down},z^{up})) \subset B_H((0,e^0),M).
\end{eqnarray*}
Since
\begin{eqnarray*}
\Xi(h)=D_{e^h}\Xi(0), \quad B^{up}=D_{e^h}B_H((0,e^0),\delta/2),
\end{eqnarray*}
this choice of $M$ guarantees that, for all $h \ge 0$, for all $z^{down} \in \Xi(h)$ and $z^{up} \in B^{up}$,
\begin{eqnarray}
\label{E:choiceofM}
B_H(z^{down},d(z^{down},z^{up})) \subset B_H((0,e^h),M).
\end{eqnarray}

\begin{figure}[!ht]
\begin{center}
\includegraphics[width=12cm,height=7.5cm]{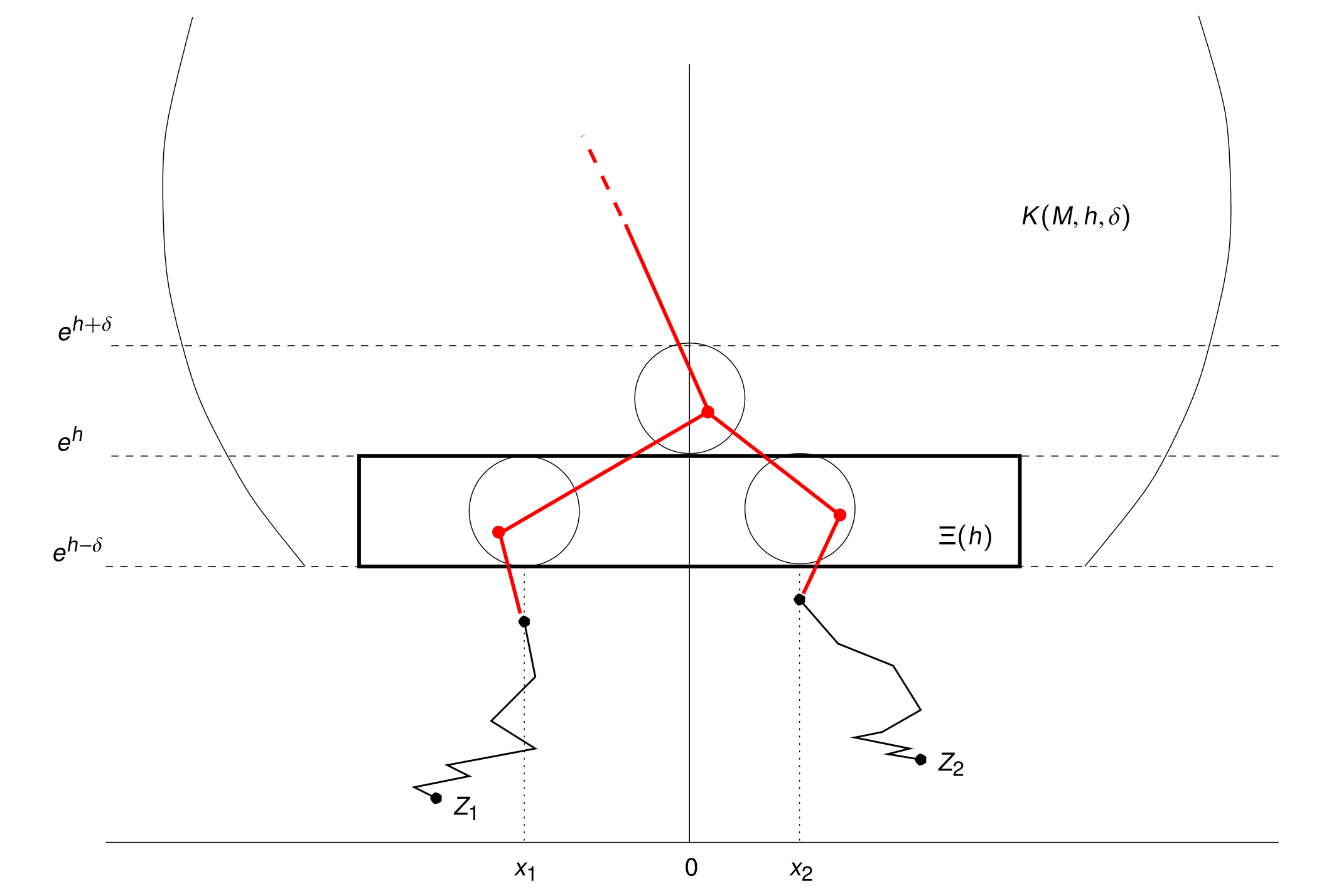}
\caption{Enlever Figure 9 ???}
\end{center}
\end{figure}

We define $\FC(\eta)$ as the event on $\mathcal{N}_S(K(M,h,\delta))$ where there is exactly one point in $B_1^{down}$, exactly one point in $B_2^{down}$, exactly one point in $B^{up}$ and no other point in $K(M,h,\delta)$:
\begin{eqnarray*}
&\FC(\eta)&:=\Big\{\eta'_{in} \in \mathcal{N}_S(K(M,h,\delta)), \nonumber\\
&&\#(\eta'_{in} \cap B_1^{down})=1 \text{ and } \#(\eta'_{in} \cap B_2^{down})=1 \text{ and } \#(\eta'_{in} \cap B^{up})=1 \nonumber\\
&&\text{ and } \eta'_{in} \backslash (B_1^{down} \cup B_2^{down} \cup B^{up})=\emptyset\Big\}. 
\end{eqnarray*}
This defines an event $\FC(\eta)$ on $\mathcal{N}_S(K(M,h,\delta))$ for any $\eta \in \Cyl(A,h)  \cap \EmptySlice(A,h,\delta)$.

We will use the two following claims. 
\begin{claim}
\label{Cl:forcecoalescence}
For any $\eta \in \Cyl(A,h) \cap \EmptySlice(A,h,\delta)$, the event $FC(\eta)$ forces coalescence, that is: for all $\eta'_{in} \in \FC(\eta)$, $\eta':=\eta_{out} \cup \eta'_{in} \in \CO$.
\end{claim}
\begin{claim}
\label{Cl:infeta}
There exists $\varepsilon>0$ independent of $h$ such that $\mathbb{P}[N_{in} \in FC(\eta)]>\varepsilon$ for all $\eta \in \Cyl(A,h) \cap \EmptySlice(A,h,\delta)$.
\end{claim}
Suppose for the moment Claims \ref{Cl:forcecoalescence} and \ref{Cl:infeta}. Choose $\varepsilon$ as in Claim \ref{Cl:infeta}. Then, since, $\mathcal{F}_{in}$ and $\mathcal{F}_{out}$ are independent, for any $\eta \in \Cyl(A,h) \cap \EmptySlice(A,h,\delta)$,
\begin{eqnarray*}
\mathbb{P}[CO|N_{out}=\eta_{out}]=\mathbb{P}[\eta_{out} \cup N_{in} \in \CO] \ge \mathbb{P}[N_{in} \in \FC(\eta)]>\varepsilon,
\end{eqnarray*}
so Lemma \ref{Lem:chooseMepsilon} is proved.

\bigbreak

We now prove Claim \ref{Cl:forcecoalescence}. Let $\eta'_{in} \in FC(\eta)$ and consider $\eta':=\eta_{out} \cup \eta'_{in}$. Define $z_1^{down}$ (resp. $z_2^{down}$) as the unique point of $B_1^{down}$ (resp. $B_2^{down}$) and $z^{up}$ as the unique point of $B^{up}$.

If we change the point process $N$ inside the box $K(M,h,\delta)$, we potentially change trajectories from $Z_1$ and $Z_2$ below the level $h-\delta$, so care is required. However, we will see that this is not a real problem. Since $Z_1,Z_2$ are measurable w.r.t the process $N$ below level 0 (i.e. $N \cap \mathbb{R}^d \times (0,e^0)$), and since $K(M,h,\delta) \in \mathbb{R}^d \times (e^{h-\delta},\infty)$, changing the point process in $K(M,h,\delta)$ does not affect the positions of $Z_1$ and $Z_2$. That is, $Z_1(\eta)=Z_1(\eta')$, $Z_2(\eta)=Z_2(\eta')$. 

For $i=1,2$, we show that, for the realisation $\eta'$, the trajectory from $Z_i$ contains $z^{up}$ (which proves that trajectories from $Z_1$ and $Z_2$ coalesce).

By our assumption on $M$ (\ref{E:choiceofM}) and since $z_i^{down} \in B_i^{down} \subset \Xi(h)$,
\begin{eqnarray*}
B_H(z_i^{down},d(z_i^{down},z^{up})) \subset B_H((0,e^h),M).
\end{eqnarray*}
Thus, since $z^{up}$ is higher than $z_i^{down}$ and the only points in $K(M,h,\delta)$ are $z_1^{down}$, $z_2^{down}$, $z^{up}$, the parent of $z_1^{down}$ is necessarily $z_2^{down}$ or $z^{up}$, and the parent of $z_2^{down}$ is necessarily $z_1^{down}$ or $z^{up}$. In all cases, $z^{up}$ is on both trajectories from $z_1^{down}$ and $z_2^{down}$.

Now let us define
\begin{eqnarray*}
\kappa_i:=\max\{l \in \llbracket 0,k_i(h) \rrbracket,~A^l(Z_i)(\eta)=A^l(Z_i)(\eta')\} \cup \{k_i(h)\},
\end{eqnarray*}
and 
\begin{eqnarray*}
z_i^{sep}:=A^{\kappa_i}(Z_i).
\end{eqnarray*}
We show that the new parent of $z_i^{sep}$ is one of the three points $z_1^{down},z_2^{down}$ or $z^{up}$ (that is, $A(z_i^{sep})(\eta') \in \{z_1^{down},z_2^{down},z^{up} \}$). This implies that the trajectory from $Z_i$ contains $z^{up}$ for the realisation $\eta'$.

Suppose $\kappa_i<k_i(h)$ (that is, the change inside $K(M,h,\delta)$ affects the trajectory from $Z_i$ before the highest point below the level $h$, it is Case 1 in Figure \ref{Fig:coalescence}). Set
\begin{eqnarray*}
(x_i^{sep+},y_i^{sep+})=A(z_i^{sep})(\eta).
\end{eqnarray*}
By assumption $y_i^{sep+}<e^h$. Since $\eta \in \Cyl(A,h,\delta)$, $\|x_i^{sep+}\| \le Ae^h$. Thus, since $\eta \in \EmptySlice(A,h,\delta)$, $y_i^{sep+}<e^{h-\delta}$. Therefore $A(z_i^{sep})(\eta) \in \eta'$, so the new parent of $z_i^{sep}$ is (strictly) closer to $z_i^{sep}$ than its previous parent; in particular the new parent cannot be in $\eta$. The only possibility is that the new parent of $z_i^{sep}$ is one of the three points $z_1^{down},z_2^{down}$ or $z^{up}$ (i.e. $A(z_i^{sep})(\eta') \in \{z_1^{down},z_2^{down},z^{up}\}$).

Consider now the remaining case (Case 2 in Figure \ref{Fig:coalescence}), $\kappa_i= k_i(h)$ (that is, the trajectory is unchanged up to $A^{k_i(h)}(Z_i)$). In this case, $z_i^{sep}=A^{k_i(h)}(Z_i)$. Since $\eta \in \EmptySlice(A,h,\delta)$, and since $\|x_i\| \le Ae^h$ (recall that $(x_i,y_i)$ are coordinates of $A^{k_i(h)}(Z_i)$), $y_i \le e^{h-\delta}$. Therefore, $z_i^{down}$ is higher than $z_i^{sep}$. Then it suffices to show that $z_i^{sep}$ is closer to $z_i^{down}$ than to its previous parent, i.e. $d(z_i^{sep},z_i^{down})<d(z_i^{sep},A(z_i^{sep})(\eta))$. Indeed, it implies that the new parent of $z_i^{sep}$ cannot be in $\eta$, so it is necessarily $z_1^{down}$, $z_2^{down}$ or $z^{up}$. It is the case if
\begin{eqnarray*}
B_i^{down} \subset B^{sep}:=B_H(z_i^{sep},d(z_i^{sep},A(z_i^{sep})(\eta))).
\end{eqnarray*}
The inclusion follows from the construction of $B_i^{down}$. Indeed, by construction, the center of $B_i^{down}$ and $z_i^{sep}$ have the same abscissa, so $B_i^{down}$ and $B^{sep}$ are centered at the same abscissa. Since the height of $A(z_i^{sep})(\eta))$ is larger than $h$, the top (i.e. highest point) of the ball $B^{sep}$ has height larger than $h$. Moreover, its bottom
(i.e. lowest point) has height smaller than $h-\delta$ since $z_i^{sep}$ has height smaller than $h-\delta$. On the other hand, the top of $B_i^{down}$ has height $e^h$ and its bottom has height $e^{h-\delta}$. It follows that $B_i^{down} \subset B^{sep}$. This proves Claim \ref{Cl:forcecoalescence}.

\begin{figure}[!h]
    \centering
    \begin{tikzpicture}[scale=1,decoration={brace}]
    \draw (0,0)--(6,0)--(6,6)--(0,6)--(0,0);
    \draw (1,5.75) circle (0.25);
    \draw (5,5.75) circle (0.25);
    \draw (3,6.25) circle (0.25);
    \draw (-1.5,5.5)--(7.5,5.5);
    \draw (-1.5,5.5) arc (200:-20:4.78);
    \draw [dashed] (0,0)--(-3,0) node[left] {$e^0$};
    \draw [dashed] (0,5.5)--(-3,5.5) node[left] {$e^{h-\delta}$};
    \draw [dashed] (0,6)--(-3,6) node[left] {$e^h$};
    \draw [decorate] (8.1,11.8)--(8.1,5.5);
    \draw (8.1,8.65) node[right] {$K(M,h,\delta),$};
    \draw (8.1,8.15) node[right] {\color{red} no points of $\eta'_{in}$};
    \draw (8.1,7.65) node[right] {\color{red} appart from};
    \draw (8.1,7.15) node[right] {\color{red} $z_1^{down}$, $z_2^{down}$ and $z^{up}$};
    \draw [decorate](6,6)--(6,5.5);
    \draw (6,5.75) node[right] {$\Slice(A,h,\delta)$, \color{blue}no points of $\eta$};
    \draw (3,0.1)--(3,-0.1) node[below] {$0$};
    \draw[<->] (3,-0.1)--(6,-0.1);
    \draw (4.5,-0.15) node[below] {$Ae^h$};
    \draw [decorate] (0.75,6)--(1.25,6);
    \draw (1,6) node[above] {$B_1^{down}$};
    \draw [decorate] (4.75,6)--(5.25,6);
    \draw (5,6) node[above] {$B_2^{down}$};
    \draw [decorate] (2.75,6.5)--(3.25,6.5);
    \draw (3,6.5) node[above] {$B^{up}$};
    \filldraw [red] (1,5.9) circle(1pt) node[right] {$z_1^{down}$};
    \filldraw [red] (5.05,5.72) circle(1pt) node[right] {$z_2^{down}$};
    \filldraw [red] (3.06,6.3) circle(1pt) node[right] {$z^{up}$};

    \draw[blue] (2,-0.5)--(3,0.5)--(2.5,1.5)--(3.5,2.5)--(0.5,3.6)--(4,4.25)--(-0.25,6.375)--(5,8);
    \draw[red] (2.1,-0.5)--(3.1,0.5)--(2.6,1.5)--(3.7,2.5)--(0.7,3.6)--(4.2,4.25)--(1,5.9)--(3.06,6.3);
    \draw[red,dashed] (0.7,3.6)--(1,5.9);

    \draw[blue]  (2,-0.5) node[right] {$Z_1$};
    \draw[blue] (4,4.25) node[right] {$Z_{k_i(h)}$};
    
    \filldraw [blue]  (2,-0.5) circle(1pt);
    \filldraw [blue] (3,0.5) circle(1pt);
    \filldraw [blue] (2.5,1.5) circle(1pt);
    \filldraw [blue] (3.5,2.5) circle(1pt);
    \filldraw [blue] (0.5,3.6) circle(1pt);
    \filldraw [blue] (4,4.25) circle(1pt);
    \filldraw [blue] (-0.25,6.375) circle(1pt);
    \filldraw [blue] (5,8) circle(1pt);
    \filldraw [blue] (6,7) circle(1pt);
    \filldraw [blue] (1,9.5) circle(1pt);
    \filldraw [blue] (3,11.5) circle(1pt);
    
    \end{tikzpicture}
    \caption{If we replace the point process \textcolor{blue}{$\eta$} inside $K(M,h,\delta)$ by some \textcolor{red}{$\eta_{in}'$} $\in \FC(\eta)$, then, for $i=1,2$, the new trajectory (in \textcolor{red}{red}) from $Z_i$ is forced to go through $z_{up}$. Either it follows the previous trajectory (in \textcolor{blue}{blue}) up to $Z_{k_i(h)}$ which is connected to $z_1^{down}$, $z_2^{down}$ or $z_{up}$ (Case 2, represented in full line), or it takes a short cut (Case 1, represented in dashed line).}
    \label{Fig:coalescence}
\end{figure}
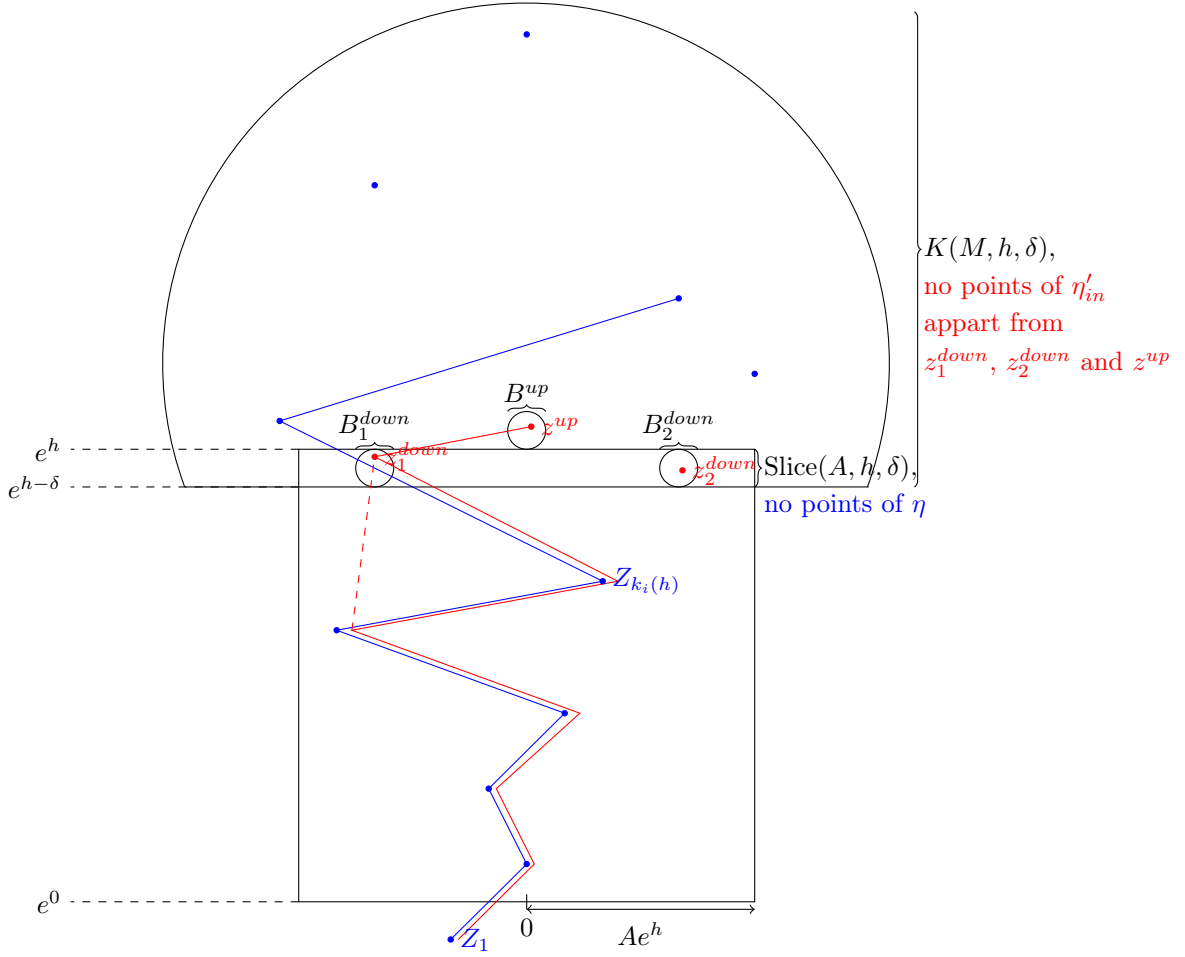

We finally prove Claim \ref{Cl:infeta}. The event $FC(\eta)$ can be realised in two different ways: (i) there is exactly one point of $N_{in}$ in $B^{up}$, exactly one point in $B_1^{down} \cap B_2^{down}$ and no other points in $K(M,h,\delta)$, or (ii) there is exactly one point in $B^{up}$, exactly one point in $B_1^{down} \backslash B_2^{down}$, exactly one point in $B_2^{down} \backslash B_1^{down}$, and no other points in $K(M,h,\delta)$. It is not hard to see that, whatever $\eta \in \Cyl(A,h) \cap \EmptySlice(A,h,\delta)$, one of the two possibilities occurs with uniformly lower-bounded probability, depending whether $B_1^{down} \cap B_2^{down}$ is large or not. The details are given in \cite[Section 7.7]{version_longue}.

\section{The bi-infinite branches and their asymptotic directions}
\label{S:proofofinfinitebranches}

In this section we prove Theorem \ref{Thm:biinfinite}.

\subsection{Notations and sketch of the proof}

Let us first introduce some notations. Recall that $\Bi$ is the set of functions $f:\mathbb{R} \to \mathbb{R}^d$ that encode a bi-infinite branch (see Definition \ref{Def:biinfinite}).

Let us denote by $\mathcal{I}$ the set of (abscissas of) points at infinity in $\mathbb{R}^d \times \{0\}$ that are the limit of at least one infinite branch in the direction of the past:
\begin{eqnarray*}
\mathcal{I}:=\{x \in \mathbb{R}^d,~\exists f \in \Bi,~ \lim_{t \to -\infty}f(t)=x\}.
\end{eqnarray*}

\begin{definition}
Let $t \in \mathbb{R}$ and $x \in \mathcal{L}_t$. We call the \emph{cell} of $x$, denoted by $\Psi_t(x)$, the set of abscissas $x'$ of points at infinity in $\mathbb{R}^d \times \{0\}$ such that there exists a infinite branch in the direction of past starting from $(x,e^t)$ that converges to $(x',0)$:
\begin{eqnarray}
\label{E:defofpsi}
\Psi_t(x):=\{x' \in \mathbb{R}^d,~\exists f \in \Bi,~f(t)=x \text{ and } \lim_{t \to -\infty}f(t)=x'\}.
\end{eqnarray}
\end{definition}

Thus $\mathcal{I}=\bigcup_{x \in \mathcal{L}_0}\Psi_t(x)$. In Step 1, we show that every infinite branch in the direction of the past converges to a point at infinity (point (ii) in Theorem \ref{Thm:biinfinite}), which is a direct consequence of Theorem \ref{Thm:controlbackward}. Then we show in Step 2 that the DSF is \emph{straight} with probability $1$ (Proposition \ref{Prop:straightness}). Recall that the Maximal Backward Deviation ($\MBD$) has been defined in Definition \ref{Def:defofcfd}. 

\begin{proposition}
\label{Prop:straightness}
\tvc{For $n\in \N$ and $x'\in \mathcal{L}_{-n}$, denote by $\MBD_{-\infty}^{-n}(x'):=\liminf_{t \to \infty} \MBD_{-t}^{-n}(x')$. Then, with probability $1$:
\begin{eqnarray}
\label{E:straightness}
\forall x \in \mathcal{L}_0,~\lim_{\substack{n \to \infty,\\n \in \mathbb{N}}}\max_{x' \in \mathcal{D}_{-n}^0(x)}  \MBD_{-\infty}^{-n}(x')=0.
\end{eqnarray}}
\end{proposition}

The above property (\ref{E:straightness}) is called the \emph{straightness} property.

The rest \tvc{of the} proof is organized as follows. In Step 3, we show that the cells $\Psi_t(x)$ are closed and we check measurably conditions on $\Psi_t(x)$ and $\mathcal{I}$. In Step 4, we use a second moment technique to show that there exists infinitely many bi-infinite branches (the point (i)). It will follow that $\mathcal{I}$ is dense in $\mathbb{R}^d$.

Then we show that $\mathcal{I}$ is a closed subset of $\mathbb{R}^d$ (Step 5). To do this, it is sufficient to show that the family of cells $\{\Psi_0(x),~x \in \mathcal{L}_0\}$ is locally finite, that is, every ball $B \subset \mathbb{R}^d$ intersects finitely many cells. Thus it follows that $\mathcal{I}$ is a dense closed subset of $\mathbb{R}^d$, therefore $\mathcal{I}=\mathbb{R}^d$. It proves (iii).

In Step 6, we prove (iv). The uniqueness follows from coalescence.

\subsection{Step 1: proof of (ii)}

By Theorem \ref{Thm:controlbackward} applied with $p=1$ and Fatou Lemma:
\begin{eqnarray}
\label{E:biinfinitestep2eq1}
\mathbb{E}_{0 \in \mathcal{L}_0} \left[ \MBD_{-\infty}^0(x) \right] \le \liminf_{h \to \infty} \mathbb{E}_{0 \in \mathcal{L}_0} \left[\MBD_{-h}^0(0)\right] \le \limsup_{h \to \infty} \mathbb{E}_{0 \in \mathcal{L}_0} \left[\MBD_{-h}^0(0)\right] <\infty.
\end{eqnarray}
Then, almost surely, for all $x \in \mathcal{L}_0$, $\MBD_{-\infty}^0(x)<\infty$. For any $f \in \Bi$, and any $h \ge 0$,
\begin{eqnarray*}
\int_{-h}^0 \|f'(t)\|~dt \le \MBD_{-h}^0(f(0)),
\end{eqnarray*}
thus
\begin{eqnarray*}
\int_{-\infty}^0 \|f'(t)\|~dt \le \liminf_{h \to \infty}\MBD_{-h}^0(f(0))=\MBD_{-\infty}^0(f(0))<\infty \text{ by } (\ref{E:biinfinitestep2eq1}),
\end{eqnarray*}
so $\lim_{t \to -\infty} f(t)$ exists. Then any bi-infinite branch admits an asymptotic direction toward the past.

\subsection{Step 2: proof of straightness}

The proof of straightness is based on Theorem \ref{Thm:controlbackward}. It is equivalent to prove the following statement:

\begin{eqnarray*}
\mathbb{P}_{0 \in \mathcal{L}_0} \left[ \lim_{\substack{n \to \infty,\\n \in \mathbb{N}}}\max_{x \in \mathcal{D}_{-n}^0(0)} \MBD_{-\infty}^{-n}(x) \right]=1.
\end{eqnarray*}

Let $n \in \mathbb{N}$. Consider the weight function
\begin{eqnarray*}
w(x,\eta):=\mathbf{1}_{x \in \mathcal{L}_{-n}} \MBD_{-\infty}^{-n}(x)^{2d}.
\end{eqnarray*}
Proposition \ref{Lem:MTlemma1} applied to $w$ with $t_1=-n$ and $t_2=0$ gives,
\begin{eqnarray*}
\mathbb{E}_{0 \in \mathcal{L}_{-n}}\left[ \MBD_{-\infty}^{-n}(0)^{2d} \right]=e^{-dn}\mathbb{E}_{0 \in \mathcal{L}_0}\left[ \sum_{x \in \mathcal{D}_{-n}^0(0)} \MBD_{-\infty}^{-n}(x)^{2d} \right].
\end{eqnarray*}
Thus
\begin{eqnarray}
\label{E:biinfeq1}
&&\mathbb{E}_{0 \in \mathcal{L}_0}\left[ \max_{x \in \mathcal{D}_{-n}^0(0)} \MBD_{-\infty}^{-n}(x)^{2d} \right] \le \mathbb{E}_{0 \in \mathcal{L}_0}\left[ \sum_{x \in \mathcal{D}_{-n}^0(0)} \MBD_{-\infty}^{-n}(x)^{2d} \right]\nonumber\\
&&=e^{dn}\mathbb{E}_{0 \in \mathcal{L}_{-n}}\left[ \MBD_{-\infty}^{-n}(0)^{2d} \right]=e^{-dn}\mathbb{E}_{0 \in \mathcal{L}_0}\left[ \MBD_{-\infty}^{0}(0)^{2d} \right].
\end{eqnarray}
where invariance by dilation (Lemma \ref{Lem:scaleinv}) was used in the last equality (with $t=-n$ and $t'=0$).

By Theorem \ref{Thm:controlbackward} applied to any $p \ge 1$ and Fatou Lemma,
\begin{eqnarray}
\label{E:lasteqnstep1}
\mathbb{E}_{0 \in \mathcal{L}_0}\left[ \MBD_{-\infty}^{0}(0)^p \right]=\mathbb{E}_{0 \in \mathcal{L}_0}\left[ \liminf_{t \to \infty}\MBD_{-t}^{0}(0)^p \right]\le \liminf_{t \to \infty} \mathbb{E}_{0 \in \mathcal{L}_0}\left[ \MBD_{-t}^0(0)^p \right]<\infty.
\end{eqnarray}

Thus, by taking $p=2d$,
\begin{eqnarray*}
\mathbb{E}_{0 \in \mathcal{L}_0}\left[ \sum_{n \in \mathbb{N}} \max_{x \in \mathcal{D}_{-n}^0(0)} \MBD_{-\infty}^{-n}(x)^{2d}\right]\overset{(\ref{E:biinfeq1})}{\le} \sum_{n \in \mathbb{N}} e^{-dn} \mathbb{E}_{0 \in \mathcal{L}_0}\left[ \MBD_{-\infty}^{0}(0)^{2d} \right] \overset{(\ref{E:lasteqnstep1})}{<}\infty.
\end{eqnarray*}
It follows that $\sum_{n \in \mathbb{N}} \max_{x \in \mathcal{D}_{-n}^0(0)}\MBD_{-\infty}^{-n}(x)^{2d}<\infty$ $\mathbb{P}_{0 \in \mathcal{L}_0}$-a.s., this implies Proposition \ref{Prop:straightness}.

\subsection{Step 3: the cells $\Psi_t(x)$ are closed and measurability conditions}

\begin{definition}
A point $(x,e^t) \in \DSF$ is said to be $\emph{connected to infinity}$ if for all $t' \le t$, $\mathcal{D}_{t'}^t(x) \neq \emptyset$. We denote by $\DSF^\infty \subset \DSF$ the set of points that are connected to infinity.
\end{definition}

For $t \in \mathbb{R}$, $x_0 \in \mathcal{L}_t$, we define the random subset of descendants of $(x_0,e^t)$:
\begin{eqnarray*}
D_t(x_0):=\{(x,e^{t'}) \in \DSF^\infty,~x \in \mathcal{D}_{t'}^t(x_0)\} \subset \DSF^\infty.
\end{eqnarray*}

The facts that $\Psi_t(x_0)$ is closed for all $t \in \mathbb{R}$ and $x_0 \in \mathcal{L}_t$ a.s. will be deduced from the following lemma:

\begin{lemma}
\label{Lem:measurlemma}
The following occurs outside a set of probability zero: for all $t \in \mathbb{R}$, $x_0 \in \mathcal{L}_t$ and $x \in \mathbb{R}^d$, $x \in \Psi_t(x_0)$ if and only if $(x,0) \in \overline{D_t(x_0)}$ (where $\overline{\cdot}$ denotes the closure operator).
\end{lemma}

Lemma \ref{Lem:measurlemma} implies that, outside a set of probability zero, for all $t \in \mathbb{R}$ and $x_0 \in \mathcal{L}_t$, $\Psi_t(x_0)=\overline{D_t(x_0)} \cap (\mathbb{R}^d \times \{0\})$ is closed in $\mathbb{R}^d$. It can also be deduced from Lemma \ref{Lem:measurlemma} that the maps
\begin{eqnarray*}
\Phi_{t}:&\mathbb{R}^d \times \mathbb{R}^d \times \mathcal{N}_S &\to \mathbb{R}\nonumber\\
&(x,x_0,\eta) &\mapsto \mathbf{1}_{x_0 \in \mathcal{L}_t(\eta)}\mathbf{1}_{x \in \Psi_{t}(x_0)(\eta)},
\end{eqnarray*}
for any $t \in \mathbb{R}$, and
\begin{eqnarray*}
\Phi':&\mathbb{R}^d \times \mathcal{N}_S &\to \mathbb{R} \nonumber\\
&(x,\eta) &\mapsto \mathbf{1}_{x \in \mathcal{I}(\eta)},
\end{eqnarray*}
are measurable if $\mathcal{N}_S$ is equipped with the completed $\sigma$-algebra (for the probability measure $\mathbb{P}$), which will be required in the following steps. The details concerning these measurability conditions are given in \cite[Section 8.4]{version_longue}.

\begin{proof}[Proof of Lemma \ref{Lem:measurlemma}] Let $t \in \mathbb{R}$ and $x_0 \in \mathcal{L}_t$. It is clear that $x \in \Psi_t(x_0)$ implies that  $(x,0) \in \overline{D_t(x_0)}$. Let us suppose $(x,0) \in \overline{D_t(x_0)}$ and we show that $x \in \Psi_t(x_0)$. Let us construct a sequence $(x'_n) \in (\mathbb{R}^d)^\mathbb{N}$ inductively such that, at each step $n \in \mathbb{N}$, $x'_n \in \mathcal{L}_{t-n}$, $x'_{n+1} \in \mathcal{D}_{t-n-1}^{t-n}(x'_n)$ and $(x,0) \in \overline{D_{t-n}(x_n)}$.

We set $x'_0=x_0$. Let $n \in \mathbb{N}$ and suppose that $x'_n$ has been constructed. Since
\begin{eqnarray*}
D_{t-n}(x'_n)=\left(\mathbb{R}^d \times (0,e^{t-n-1}] \right) \cap \left(\bigcup_{x'' \in \mathcal{D}_{t-n-1}^{t-n}(x'_n)} D_{t-n-1}(x'') \right)
\end{eqnarray*}
and $\#\mathcal{D}_{t-n-1}^{t-n}(x)<\infty$ with probability $1$ (Corollary \ref{Cor:expnumdesc}), it follows that
\begin{eqnarray}
\label{E:measureqn2}
\overline{D}_{t-n}(x'_n) \cap (\mathbb{R}^d \times \{0\})=\bigcup_{x'' \in \mathcal{D}_{t-n-1}^{t-n}(x'_n)} \left( \overline{D}_{t-n-1}(x'') \cap (\mathbb{R}^d \times \{0\})\right).
\end{eqnarray}

Since $(x,0) \in \overline{D}_{t-n}(x'_n)$, by (\ref{E:measureqn2}) it is possible to choose $x'_{n+1} \in \mathcal{D}_{t-n-1}^{t-n}(x'_n)$ such that $(x,0) \in \overline{D}_{t-n-1}(x'_{n+1})$.

This construction defines a sequence $(x'_n)_{n \in \mathbb{N}}$ such that, for all $n \in \mathbb{N}$, $(x'_n,0) \in \overline{D_{t-n}(x'_n)}$. The sequence of points $(x'_n,e^{t-n})$ naturally defines an infinite branch toward the past and it remains to show that this branch converges to $(x,0)$ toward the past. By Step 1, this branch converges to some point at infinity, thus it suffices to show that $x'_n \to x$ as $n \to \infty$. 

Let us assume for the moment that
\begin{eqnarray}
\label{E:measureqn3}
\forall n \in \mathbb{N},~\|x'_n-x\| \le \MBD_{-\infty}^{t-n}(x'_n).
\end{eqnarray}

Then, by the straightness property (Proposition \ref{Prop:straightness}),
\begin{eqnarray*}
\|x'_n-x\| \le \MBD_{-\infty}^{t-n}(x'_n) \le \max_{x'' \in \mathcal{D}_{t-n}^t(x_0)} \MBD_{-\infty}^{t-n}(x'') \to 0 \text{ as } n \to \infty.
\end{eqnarray*}

It remains to prove (\ref{E:measureqn3}). Let $n \in \mathbb{N}$. Let $\varepsilon>0$. Since $x \in \overline{D_{t-n}(x'_n)}$, there exists $t_2 \le t-n$ and $x_2 \in \mathcal{D}_{t_2}^{t-n}(x'_n)$ such that $(x_2,t_2) \in \DSF^\infty$ and $\|x_2-x\|<\varepsilon$. Let $t_3 \le t_2$. Since $(x_2,t_2) \in \DSF^\infty$, there exists some $x_3 \in \mathcal{D}_{t_3}^{t_2}(x_2)$. Thus,
\begin{eqnarray*}
&\|x'_n-x\| &\le \|x'_n-x_2\|+\|x_2-x\| \le  \CFD_{t_2}^{t-n}(x_2)+\varepsilon \nonumber\\
&&\le \CFD_{t_3}^{t-n}(x_3)+\varepsilon \le \MBD_{t_3}^{t-n}(x'_n)+\varepsilon.
\end{eqnarray*}
Thus $\|x'_n-x\| \le \MBD_{t_3}^{t-n}(x'_n)+\varepsilon$ for any $t_3$ small enough, then
\begin{eqnarray*}
\|x'_n-x\| \le \liminf_{t' \to \infty}\MBD_{t'}^{t-n}+\varepsilon=\MBD_{-\infty}^{t-n}(x'_n)+\varepsilon.
\end{eqnarray*}
Since this is true for any $\varepsilon>0$, we obtain (\ref{E:measureqn3}), this completes the proof of Lemma \ref{Lem:measurlemma}.
\end{proof}

\subsection{Step 4: $\mathcal{I}$ is nonempty and dense in $\mathbb{R}^d$}

The main part of the proof consists in proving that $\mathcal{I} \neq \emptyset$ (i.e. the point (i) in Theorem \ref{Thm:biinfinite}). The density will follow easily. The proof is based on a second moment method, Theorem \ref{Thm:controlforward} and Lemma \ref{Lem:technical}.
For $t \ge 0$, let us define the level $t$-association function $f_t$ as follows:
\begin{eqnarray*}
f_t(x,\eta):=\mathcal{A}_0^t\left(\argmin_{x' \in \mathcal{L}_0(\eta)}\|x'-x\|\right)(\eta),
\end{eqnarray*}
for any $x \in \mathbb{R}^d$ and $\eta \in \mathcal{N}_S$. That is, we consider the point $x'$ of $\mathcal{L}_0(\eta)$ the closest to $x$ and we follow the trajectory from $x'$ up to the level $t$.
We apply a second moment method on $V_t(0):=\Leb(\Lambda_{f_t}(0))$ (recall that $\Lambda_{f_t}(\cdot)$ is defined in Definition  \ref{Def:cellofpoint}). First, Corollary \ref{Cor:expectedvolcell} gives that $\mathbb{E}_{0 \in \mathcal{L}_t}[V_t(0)]=\alpha_0^{-1}e^{dt}$, where $\alpha_0$ has been defined in Proposition \ref{Prop:finiteintensity}. We now apply Theorem \ref{Thm:controlforward} to upper-bound $\mathbb{E}_{0 \in \mathcal{L}_t}[V_t(0)^2]$.

Let $X_0:=\argmin_{x \in \mathcal{L}_0}\|x\|$ be the point of $\mathcal{L}_0$ the closest to $0$. By Proposition \ref{Lem:closest}, $\mathbb{E}[\|X_0\|^p]<\infty$. Then (i) in Theorem \ref{Thm:controlforward} applied to $X_0$ with $p=d$ gives,
\begin{eqnarray}
\label{E:bininfeq1}
\limsup_{t \to \infty} ~e^{-dt}\mathbb{E}\left[\|\mathcal{A}_0^t(X_0)\|^d \right] \le \limsup_{t \to \infty} \mathbb{E}\left[ \left(\frac{\|X_0\|+\CFD_0^t(X_0)}{e^t}\right)^d \right]<\infty.
\end{eqnarray}
Applying Proposition \ref{Lem:technical} to $f_t$ with  $p=d$, we obtain
\begin{eqnarray}
\label{E:bininfeq2}
\mathbb{E}_{0 \in \mathcal{L}_t}[V_t(0)^2] \le C_{d,d}e^{dt}\mathbb{E}[\|f_t(0)\|^d]=C_{d,d}e^{dt}\mathbb{E}[\|\mathcal{A}_0^t(X_0)\|^d].
\end{eqnarray}
Then
\begin{eqnarray}
\label{E:bininffinitelimsup}
\limsup_{t \to \infty} ~e^{-2dt}\mathbb{E}_{\tvc{0 \in \mathcal{L}_t}}[V_t(0)^2] \overset{(\ref{E:bininfeq2})}{\le}\limsup_{t \to \infty} C_{d,d}e^{-dt}\mathbb{E}[\|\mathcal{A}_0^t(\tvc{X_0})\|^d] \overset{(\ref{E:bininfeq1})}{<}\infty.
\end{eqnarray}
By scale invariance (Lemma \ref{Lem:scaleinv}) and Cauchy-Schwarz,
\begin{eqnarray*}
\mathbb{P}_{0 \in \mathcal{L}_0}[\mathcal{D}_{-t}^0(0) \neq \emptyset]=\mathbb{P}_{0 \in \mathcal{L}_t}[\mathcal{D}_0^t(0) \neq \emptyset]=\mathbb{P}_{0 \in \mathcal{L}_t}[V_t(0)>0] \ge \frac{\mathbb{E}_{0 \in \mathcal{L}_t}[V_t(0)]^2}{\mathbb{E}_{0 \in \mathcal{L}_t}[V_t(0)^2]}.
\end{eqnarray*}
Thus
\begin{eqnarray*}
&\mathbb{P}_{0 \in \mathcal{L}_0}[\forall t \ge 0,~\mathcal{D}_{-t}^0(0) \neq \emptyset]
&=\lim_{t \to \infty} \mathbb{P}_{0 \in \mathcal{L}_0}[\mathcal{D}_{-t}^0(0) \neq \emptyset] \ge \liminf_{t \to \infty} \frac{e^{-2dt}\mathbb{E}_{0 \in \mathcal{L}_t}[V_t(0)]^2}{e^{-2dt}\mathbb{E}_{0 \in \mathcal{L}_t}[V_t(0)^2]} \nonumber\\
&&=\frac{\alpha_0^{-2}}{\underset{t \to \infty}{\limsup}~e^{-2dt}\mathbb{E}_{0 \in \mathcal{L}_t}[V_t(0)^2]}>0 \text{ by } (\ref{E:bininffinitelimsup}).
\end{eqnarray*}
Note that, if the DSF has finite degree (which happens with probability 1), then a point $(x,e^0)$ with $x \in \mathcal{L}_0$ belongs to an infinite branch if and only if for all $t \ge 0$, $\mathcal{D}_{-t}^0(x) \neq \emptyset$. Thus 
\begin{eqnarray*}
\mathbb{P}_{0 \in \mathcal{L}_0}[0 \text{ belongs to a bi-infinite branch}]>0.
\end{eqnarray*}
It easily follows that there exists a bi-infinite branch with positive probability, and bi-infinite branches converge in the direction of the past by Step 1, so $\mathbb{P}[\mathcal{I} \neq \emptyset]>0$. Since the event $\mathcal{I} \neq \emptyset$ is translation invariant, it implies $\mathbb{P}[\mathcal{I} \neq \emptyset]=1$ by ergodicity.

We move on to show that $\mathcal{I}$ is dense in $\mathbb{R}^d$ almost surely. Let $x \in \mathbb{R}^d$. By translation invariance, $\mathcal{I}\overset{d}{=}T_x\mathcal{I}$. Moreover, for $r,r'>0$, by dilation invariance, $\mathcal{I} \overset{d}{=} D_{r/r'}\mathcal{I}$, \tvc{where $D_{r/r'} :  x\in \mathbb{R}^d\mapsto x\times r/r'$ is a dilation of $\mathbb{R}^d$}. Thus, since $\mathcal{I} \neq \emptyset$ a.s.,
\begin{eqnarray*}
&\mathbb{P}[\mathcal{I} \cap B(x,r) \neq \emptyset]&=\mathbb{P}[T_x\mathcal{I} \cap B(x,r) \neq \emptyset]=\mathbb{P}[\mathcal{I} \cap B(0,r) \neq \emptyset] \\
&&=\mathbb{P}[D_{r/r'}\mathcal{I} \cap B(0,r) \neq \emptyset]=\mathbb{P}[\mathcal{I} \cap B(0,r') \neq \emptyset] \to 1 \text{ as } r' \to \infty, \nonumber
\end{eqnarray*}
thus $\mathbb{P}[\mathcal{I} \cap B(x,r) \neq \emptyset]=1$. Since $\mathbb{R}^d$ admits a countable basis, it follows that $\mathcal{I}$ is dense in $\mathbb{R}^d$ almost surely.

\subsection{Step 5: $\mathcal{I}$ is closed in $\mathbb{R}^d$}

Since $\Psi_0(x)$ is closed in $\mathbb{R}^d$ for all $x \in \mathcal{L}_0$ by Step 3, \tvc{it is sufficient to show that almost surely, for any ball $B \subset \mathbb{R}^d$ of radius $1$, $B \cap \Psi_0(x) \neq \emptyset$ for finitely many $x \in \mathcal{L}_0$.} Let $B \subset \mathbb{R}^d$ be some ball of radius $1$, it will be shown that $B$ intersects finitely many cells a.s. The conclusion will immediately follow since $\mathbb{R}^d$ admits a countable basis. By translation invariance it is enough to consider $B(0,e^0)$.

For $t \in \mathbb{R}$ and $x \in \mathcal{L}_t$, we define the radius of the cell $\Psi_t(x)$, denoted by $\Rad_t(x)$, as
\begin{eqnarray*}
\Rad_t(x):=\sup_{x' \in \Psi_t(x)} \|x'-x\|.
\end{eqnarray*}
with the convention $\sup \emptyset=0$, where $\Psi_t(x)$ is defined in (\ref{E:defofpsi}). We now show that $\Rad_0(x) \le \MBD_{-\infty}^0(x)$. Let $x' \in \Psi_0(x)$. There exists $f \in \Bi$ such that $f(0)=x$ and $\lim_{t \to -\infty} f(t)=x'$. Thus
\begin{eqnarray*}
\|x'-x\|=\left\|\int_{-\infty}^0 f'(t)~dt\right\| \le \int_{-\infty}^0 \|f'(t)\|~dt \le \liminf_{t \to -\infty} \MBD_t^0(x)=\MBD_{-\infty}^0(x).
\end{eqnarray*}
Since this is true for each $x' \in \Psi_0(x)$, $\Rad_0(x) \le \MBD_{-\infty}^0(x)$. Thus, by (\ref{E:lasteqnstep1}), for any $p \ge 1$, $\mathbb{E}_{0 \in \mathcal{L}_0}[\Rad_0(0)^p]<\infty$.
For $x \in \mathcal{L}_0$, we now define the \emph{augmented cell} of $x$ by the set of points $x'$ that are at distance at most 1 from $\Psi_0(x)$:
\begin{eqnarray*}
\Psi'_0(x):=\{x' \in \mathbb{R}^d,~\exists x'' \in \Psi_0(x),\|x''-x'\|<1\}.
\end{eqnarray*}
Note that $\Psi_0(x) \cap B(0,e^0) \neq \emptyset$ if and only if $0 \in \Psi'_0(x)$, this is the reason why $\Psi'_0(x)$ has been introduced. Thus what we want to show is that $0 \in \Psi_0'(x)$ for finitely many $x \in \mathcal{L}_0$. It is done by the Mass Transport Principle. From each $x \in \mathcal{L}_0$, we transport a unit mass from $x$ to each unit volume of $\Psi_0'(x)$. It corresponds to the measure $\pi$ defined as
\begin{eqnarray*}
\pi(E):=\mathbb{E}\left[\sum_{x \in \mathcal{L}_0} \int_{\Psi_0'(x)} \mathbf{1}_{(x,x') \in E} ~dx'\right].
\end{eqnarray*}
for all $E \subset \mathbb{R}^d \times \mathbb{R}^d$.
Let $A$ be a nonempty open subset of $\mathbb{R}^d$. On the one hand,
\begin{eqnarray*}
\pi(A \times \mathbb{R}^d)=\mathbb{E}\left[ \sum_{x \in \mathcal{L}_0 \cap A} \Leb(\Psi'_0(x)) \right]=\alpha_0\Leb(A)\mathbb{E}_{0 \in \mathcal{L}_0}[\Leb(\Psi'_0(0))],
\end{eqnarray*}
where Lemma \ref{Lem:technicallemmageneral} is used in the second equality. On the other hand,
\begin{eqnarray*}
&\pi(\mathbb{R}^d \times A)&=\mathbb{E}\left[ \sum_{x \in \mathcal{L}_0} \int_{\Psi'_0(x)} \mathbf{1}_{x'  \in A}~dx' \right] =\mathbb{E}\left[\sum_{x \in \mathcal{L}_0} \int_{A} \mathbf{1}_{x' \in \Psi'_0(x)} ~dx'\right] \\
&&=\mathbb{E}\left[\int_A \#\{x \in \mathcal{L}_0,~x' \in \Psi'_0(x)\} ~dx' \right]=\int_A \mathbb{E}\left[ \#\{x \in \mathcal{L}_0,~x' \in \Psi'_0(x)\} \right]~dx' \nonumber\\
&&=\int_A \mathbb{E}\left[ \#\{x \in \mathcal{L}_0,~0 \in \Psi'_0(x)\} \right]~dx'=\Leb(A) \mathbb{E}\left[ \#\{x \in \mathcal{L}_0,~0 \in \Psi'_0(x)\} \right]. \nonumber
\end{eqnarray*}
Fubini was used in second, third and fourth equality and translation invariance was used in the fifth equality. Thus, since $\pi$ is diagonally invariant, the Mass Transport Principle gives,
\begin{eqnarray*}
\mathbb{E}\left[ \#\{x \in \mathcal{L}_0,~0 \in \Psi'_0(x)\} \right]=\alpha_0\mathbb{E}_{0 \in \mathcal{L}_0}[\Leb(\Psi'_0(0))].
\end{eqnarray*}
Denoting by $\vartheta(d)$ the volume of the unit ball in $\mathbb{R}^d$, we have
\begin{eqnarray*}
\mathbb{E}_{0 \in \mathcal{L}_0}[\Leb(\Psi'_0(0))] \le \mathbb{E}_{0 \in \mathcal{L}_0}[\vartheta(d) (\Rad_0(0)+1)^d] \le 2^{d-1}\vartheta(d)\left(\mathbb{E}_{0 \in \mathcal{L}_0}[\Rad_0(0)^d]+1 \right)<\infty,
\end{eqnarray*}
it follows that $\mathbb{E}\left[ \#\{x \in \mathcal{L}_0,~0 \in \Psi'_0(x)\} \right]<\infty$, this proves that the family $\{\Psi_0(x),~x \in \mathcal{L}_0\}$ is locally finite almost surely.

Therefore, $\mathcal{I}$ is dense and closed in $\mathbb{R}^d$, thus $\mathcal{I}=\mathbb{R}^d$. This proves (iii) in Theorem \ref{Thm:biinfinite}.

\subsection{Step 6 proof of (iv)}

Let us call $\mathcal{I}'$ the set of (abscissas of) points in $\mathbb{R}^d \times \{0\}$ which are the limit in the direction of past of at least two bi-infinite branches:
\begin{eqnarray*}
\mathcal{I}':=\{x \in \mathbb{R}^d,\exists f_1,f_2 \in \Bi,~f_1 \neq f_2 \text{ and }\lim_{t \to -\infty}f_1(t)=\lim_{t \to -\infty}f_2(t)=x\}.
\end{eqnarray*}

The proof that $\Phi':(x,\eta) \mapsto \mathbf{1}_{x \in \mathcal{I}(\eta)}$ is measurable, done in Step 3, can be easily adapted to show that $(x,\eta) \mapsto \mathbf{1}_{x \in \mathcal{I}'(\eta)}$ is also measurable. By translation invariance and Fubini,
\begin{eqnarray*}
\mathbb{E}[\Leb(\mathcal{I}')]=\mathbb{E}\left[ \int_{\mathbb{R}^d} \mathbf{1}_{x \in \mathcal{I}'}~dx \right]=\int_{\mathbb{R}^d}\mathbb{P}[x \in \mathcal{I}']~dx=\int_{\mathbb{R}^d}\mathbb{P}[0 \in \mathcal{I}']~dx=\infty\mathbb{P}[0 \in \mathcal{I}'].
\end{eqnarray*}

Thus, in order to show that $\Leb(\mathcal{I}')=0$ a.s., we will prove that $\mathbb{P}[0 \in \mathcal{I}']=0$.

Consider the set of bi-infinite branches that converges to $(0,0)$ in the direction of the past. For $t \in \mathbb{R}$, let $\mathcal{P}(t)$ be the set of $t$-level points through which these branches pass:
\begin{eqnarray*}
\mathcal{P}(t):=\{f(t),~f \in \Bi \text{ and } \tvc{\lim_{t \to -\infty}f(t)}=0\}=\{x \in \mathcal{L}_t,~0 \in \Psi_t(x)\}.
\end{eqnarray*}

We define the \emph{coalescing time of $0$}, denoted by $\tau_0$, as the first time $t$ for which all branches converging to $(0,0)$ in the direction of the past have coalesced:
\begin{eqnarray*}
\tau_0:=\inf\{t \in \mathbb{R},~\#\mathcal{P}(t)=1\} \in \mathbb{R} \cup \{-\infty,+\infty\}.
\end{eqnarray*}

Let us show that $\tau_0<+\infty$ a.s. It has been shown in Step 3 that the family of cells $\{\Psi_0(x),~x \in \mathcal{L}_0\}$ is locally finite, so $\#\mathcal{P}(0)<\infty$ a.s. (and it is also true that $\#\mathcal{P}(t)<\infty$ for all $t$). By coalescence (Theorem \ref{Thm:coalescence}), there exists a.s. some $t \ge 0$ such that a.s. (and it is also true that $\#\mathcal{P}(t)<\infty$ for all $t$). By coalescence (Theorem \ref{Thm:coalescence}ll trajectories starting from the points $\{(x,e^0),~x \in \mathcal{P}(0)\}$ coalesce before time $t$. For such a $t$, $\#\mathcal{P}(t)=1$, therefore $\tau_0<\infty$ a.s.

By dilation invariance, for all $t \in \mathbb{R}$, $\tau_0 \overset{d}{=}\tau_0+t$, therefore the only possibility is that $\tau_0=-\infty$ a.s. This implies that $\#\mathcal{P}(t)=1$ for all $t \in \mathbb{R}$ a.s., so there exists a unique $f \in \Bi$ such that $\lim_{t \to -\infty} f(t)=0$. This shows that $\mathbb{P}[0 \in \mathcal{I}']=0$ so $\mathcal{I}'$ has measure zero almost surely.

We move on to show that $\mathcal{I}'$ is dense in $\mathbb{R}^d$. We first show that $\mathcal{I}' \neq \emptyset$ a.s. Let us suppose that $\mathcal{I}'=\emptyset$ with positive probability. On the event $\{\mathcal{I}'=\emptyset\}$, the cells $\{\Psi_0(x),~x \in \mathcal{L}_0\}$ are pairwise disjoint. So for all $x \in \mathcal{L}_0$,

\begin{eqnarray*}
\Psi_0(x)^c=\bigcup_{\substack{x' \in \mathcal{L}_0 \\ x' \neq x }} \Psi_0(x').
\end{eqnarray*}

Since the cells $\Psi_0(x)$ are closed in $\mathbb{R}^d$ (Step 3) and the family $\{\Psi_0(x)~x \in \mathcal{L}_0\}$ is locally finite (Step 5), both $\Psi_0(x)$ and $\Psi_0(x)^c$ must be closed in $\mathbb{R}^d$. By connectivity, this implies that $\Psi_0(x)$ is $\emptyset$ or $\mathbb{R}^d$ and there is unique $x \in \mathbb{R}^d$ such that $\Psi_0(x)=\mathbb{R}^d$. Then, conditioning to the event $\{\mathcal{I}'=\emptyset\}$, the law of the unique random $X \in \mathcal{L}_0$ such that $\Psi_0(X)=\mathbb{R}^d$ must be translation invariant, which is impossible. Therefore $\mathbb{P}[\mathcal{I}'=\emptyset]=0$.

We now show that $\mathcal{I}'$ is dense in $\mathbb{R}^d$ by the same argument that have been use to show that $\mathcal{I}$ is dense. For any $x \in \mathbb{R}^d$ and $0<\varepsilon<R<\infty$, by translation and dilation invariance,
\begin{eqnarray*}
\mathbb{P}[\mathcal{I}' \cap B(x,\varepsilon)]=\mathbb{P}[\mathcal{I}' \cap B(0,R)] \to 1 \text{ as } R \to \infty,
\end{eqnarray*}
so $\mathbb{P}[\mathcal{I}' \cap B(x,\varepsilon)]=1$. Since $\mathbb{R}^d$ admits a countable basis, we can conclude that $\mathcal{I}'$ is dense in $\mathbb{R}^d$ almost surely. 

The last point is to show that $\mathcal{I}'$ is countable in the bi-dimensional case ($d=1$). Note that, for $x \in \mathbb{R}^d$, $x \in \mathcal{I}'$ if and only if there exists some level $t \in \mathbb{R}$ and two points $x_1,x_2 \in \mathcal{L}_t$ with $x_1 \neq x_2$ such that $x \in \Psi_t(x_1) \cap \Psi_t(x_2)$. Moreover the level $t$ can be chosen rational without loss of generality. Thus it suffices to show that, for a given level $t \in \mathbb{Q}$, $\cup_{x_1,x_2 \in \mathcal{L}_t,~x_1 \neq x_2} (\Psi_t(x_1) \cap \Psi_t(x_2))$ is countable. Let us consider the set $L_t^\infty:=\{x \in \mathcal{L}_t,~\Psi_t(x) \neq \emptyset\}$. Since it is a discrete subset of $\mathbb{R}$, let us index its elements by $\mathbb{Z}$ in the ascending order: $\mathcal{L}_t^\infty=\{x_n,~n \in \mathbb{Z}\}$. It has been shown that, for $n \in \mathbb{Z}$, $\Psi_t(x_n) \subset \mathbb{R}$ is closed (Step 3) and bounded (Step 5); moreover it has to be connected by planarity. Thus $\Psi_t(x_n)$ is a segment (eventually reduced to a single point); let us write $\Psi_t(x_n)=[a_n,b_n]$ for all $n \in \mathbb{Z}$. Again by planarity, $b_n \le a_{n+1}$ for all $n \in \mathbb{N}$ (else a trajectory from $(b_n,0)$ should cross a trajectory from $(a_{n+1},0)$). Moreover, since the segments $[a_n,b_n]$ cover $\mathbb{R}$, $b_n \le a_{n+1}$ so $a_n=b_{n+1}$ for all $n \in \mathbb{Z}$. Finally, the set of points in $\mathbb{R}$ belonging to two different cells $[a_n,b_n]$ are exactly the set of extremities $\{a_n,~n \in \mathbb{Z}\}$, so it is countable. This completes the proof.

\bigbreak

We can wonder what are the possible numbers of bi-infinite branches sharing a same asymptotic direction toward the past. This question is unsolved, but we can give the following conjecture:

\begin{conj}
\label{conj:infinitebranches}
Almost surely, the maximal number of bi-infinite branches sharing a same asymptotic direction toward the past is $d+1$. That is,
\begin{eqnarray*}
\max_{x \in \mathbb{R}^d} \#\{ f \in \Bi,~\lim_{t \to -\infty}f(t)=x\}=d+1.
\end{eqnarray*}
\end{conj}

The intuition behind this conjecture can be explained as follows. Let us consider the family of cells $\{\Psi_t(x),~x \in \mathcal{L}_t)$ for a given level $t \in \mathbb{R}$. They cover $\mathbb{R}^d$ and they do not overlap except for boundaries. A boundary point shared by $k$ cells corresponds to an asymptotic direction with $k$ branches that have not coalesced at level $t$. It is reasonable to expect that it exists $d+1$ cells sharing a same boundary point, but that it does not exist $d+2$ cells overlapping at a same point. If this is true for every level $t \in \mathbb{R}$, it implies the existence of $d+1$ branches sharing a same asymptotic direction but the non-existence of $d+2$ such branches.

\section{Appendix : first properties of the hyperbolic DSF}
\label{S:firstprop}

In this section, we show Proposition \ref{Prop:firstprop}.

\subsection{The edges never cross}
Let us show that the DSF is non-crossing a.s. We first run out the case $d \ge 2$. Almost surely, $N$ does not contain four coplanary points, so two edges never cross.

In the following, we suppose $d=1$. Recall that $\pi_y:(x,y) \mapsto y$ is the projection on the $y$-coordinate. Let $z_1,z_2 \in \mathcal{N}_S$ and suppose that $[z_1,A(z_1)]_{eucl} \cap [z_2, A(z_2)]_{eucl} \neq \emptyset$. We denote by $P_{eucl}$ the intersection point of $[z_1,A(z_1)]_{eucl}$ and $[z_2, A(z_2)]_{eucl}$. Let us suppose that there are no two points $z_1$, $z_2$ with $\pi_y(z_1)=\pi_y(z_2)$ (this happens with probability $0$). We will prove the following:

\begin{claim}
\label{Cl:geodesicscross}
The geodesics $[z_1,A(z_1)]$ and $[z_2,A(z_2)]$ meet at one point $P_{hyp}$.
\end{claim}

We suppose Claim \ref{Cl:geodesicscross} for the moment. We have $\pi_y(A(z_1))>\pi_y(P_{eucl})>\pi_y(z_2)$, thus by definition of the parent, $d(z_2,A(z_2))<d(z_2,A(z_1))$. Then
\begin{eqnarray}
&d(z_2,P_{hyp})+d(P_{hyp},A(z_2))&=d(z_2,A(z_2))<d(z_2,A(z_1)) \nonumber\\
&&\le d(z_2,P_{hyp})+d(P_{hyp},A(z_1)),
\end{eqnarray}
so $d(P_{hyp},A(z_2))<d(P_{hyp},A(z_1))$. On the other hand, interchanging $z_1$ and $z_2$ in the previous calculation leads to $d(P_{hyp},A(z_1))<d(P_{hyp},A(z_2))$. This is a contradiction. Therefore $[z_1,A(z_1)]_{eucl} \cap [z_2,A(z_2)]_{eucl}=\emptyset$.

It remains to show Claim \ref{Cl:geodesicscross}. For $i=1,2$, consider the simple closed curve supported on $[z_i,A(z_i)] \cup [z_i,A(z_i)]_{eucl}$. Let us denote by $R_i$ the region of $H$ inside this closed curve. We now show that $R_i$ contains no point of $N$. Both $[z_i,A(z_i)]_{eucl}$ and $[z_i,A(z_i)]$ are contained in $B_H(z_i,d(z_i,A(z_i)))$ since $B_H(z_i,d(z_i,A(z_i)))$ is a Euclidean ball so it is convex for both Hyperbolic and Euclidean metrics. Moreover, $\pi_y(A(z_i))>\pi_y(z_i)$, so both $[z_i,A(z_i)]$ and $[z_i,A(z_i)]_{eucl}$ are contained in the upper-half plane $\mathbb{R}^d \times (\pi_y(z_i),\infty)$. Thus, both $[z_i,A(z_i)]$ and $[z_i,A(z_i)]_{eucl}$ are contained in $B^+(z_i)$. By simple connexity, $R_i \subset B^+(z_i)$. Thus, since $N \cap B^+(z_i)=\emptyset$, $R_i$ contains no points of $N$.

By assumption $[z_1,A(z_1)]_{eucl}$ crosses $[z_2,A(z_2)]_{eucl}$ exactly once, and none of the extremities $z_1$ and $A(z_1)$ belong to $R_2$. Thus $[z_1,A(z_1)]_{eucl}$ should cross $[z_2,A(z_2)]$ exactly once. Now, consider $[z_2,A(z_2)]$. None of the extremities $z_2$ and $A(z_2)$ belong to $R_1$, so by the same argument, $[z_2,A(z_2)]$ crosses $[z_1,A(z_1)]$ exactly once. This proves Claim \ref{Cl:geodesicscross} and achieves the proof of Proposition \ref{Prop:noncrossing}.

\subsection{The DSF has finite degree}

We move on to show that the DSF is locally finite a.s. Fix the origin $z_0:=(0,e^0)$. Consider $N'=N \cup \{z_0\}$ and consider the DSF on $N'$. Since $N$ is a Poisson Point Process, $N'$ has same law as the Palm version of $N$ conditioned that $z_0 \in N$. The origin $z_0$ has one parent almost surely, so it has to be shown that $z_0$ has finitely many sons almost surely. We apply Campbell formula \cite{stochastic}. Consider the function
\begin{eqnarray}
F: \mathcal{N}_S \times \mathbb{H}^d &\rightarrow& \mathbb{R}_+ \nonumber\\
(\eta,z) &\mapsto& \mathbf{1}_{B_+(z,d(z,0)) \cap N=\emptyset}
\end{eqnarray}

For $z \in N$, if $z$ is a son of $z_0$ then $B_+(z,d(z,z_0))=\emptyset$ so $F(N \backslash \{z\},z)=1$. Therefore,
\begin{eqnarray}
\label{E:finitedegreeeq1}
&\mathbb{E}\left[ \#\{z \in N,~(z,z_0) \in \vec{E} \} \right] &\le \mathbb{E}\left[ \sum_{z \in N} F(N \backslash \{z\},z) \right] \nonumber\\
&&=\int_{\mathbb{H}^d} \mathbb{E}\left[ F(N,z) \right]~dz \nonumber\\
&&=\int_{\mathbb{H}^d} \exp\big(-\lambda\mu(B^+(z,d(z,z_0)))\big)~dz, \nonumber\\
&&=\int_{\mathbb{H}^d} \exp\big(-\lambda\mu(B^+(0,d(z,z_0))) \big)~dz,
\end{eqnarray}
where Campbell formula was used in the first equality. The last inequality holds since, for all $\rho>0$ $B^+(z,\rho)$ have same volume as $B^+(z_0,\rho)$ by isometry invariance. We now rewrite the integral above using the following coordinates transformation formula:

\begin{lemma}
\label{Lem:polarcoord}
Let $f:\mathbb{R}_+ \to \mathbb{R}_+$. Then
\begin{eqnarray}
\int_{\mathbb{H}^d} f(d(z,z_0))~dz=\int_{\mathbb{R}_+} s(\rho)f(\rho)~d\rho
\end{eqnarray}
$s:\mathbb{R}_+ \to \mathbb{R}_+$ is some function. This function verify $s(\rho) \sim \beta e^{d\rho}$ when $\rho \to \infty$ for some constant $\beta>0$. 
\end{lemma}

The proof of lemma \ref{Lem:polarcoord} is given in the Appendix. This formula applied to $f(\rho)=\mu(B^+(z_0,\rho))$ and (\ref{E:finitedegreeeq1}) lead to:
\begin{eqnarray}
&\mathbb{E}\left[ \#\{z \in N,~(z,z_0) \in \vec{E} \} \right] \le \int_{\mathbb{R}_+} s(\rho) e^{-\lambda\mu(B^+(z_0,\rho))})~d\rho.
\end{eqnarray}

In order to show that the right-hand side is finite, we need to lower-bound $\mu(B^+(z_0,\rho))$. Suppose for the moment that, for all $\rho$ large enough,
\begin{eqnarray}
\label{E:lowbplus}
\mu(B^+(z_0,\rho)) \ge e^{d\rho/3},
\end{eqnarray}
Then
\begin{eqnarray}
&\mathbb{E}\left[ \#\{z \in N,~(z,z_0) \in \vec{E} \} \right] &\le \int_{\mathbb{R}_+} s(\rho) e^{-\lambda\mu(B^+(z_0,\rho))})~d\rho \nonumber\\
&&\le \int_{\mathbb{R}_+} s(\rho)\exp(-\lambda e^{d\rho/3})~d\rho \nonumber\\
&&<\infty
\end{eqnarray}
since $s(\rho) \sim \beta \exp(d\rho) \ll \exp(\lambda e^{d\rho/3})$. Thus $0$ has a finite number of sons almost surely, this shows that the DSF is locally finite almost surely.

It remains to show (\ref{E:lowbplus}). Let $\rho>0$. Consider the cylinder
\begin{eqnarray}
\mathcal{C}_\rho:=B_{\mathbb{R}^d}(0,e^{\frac{2}{5}\rho}) \times [1,e^\rho-e^{-\rho}].
\end{eqnarray}

The claim is that, when $\rho$ is large enough, $\mathcal{C}_\rho \subset B^+(0,\rho)$. Indeed, by the discussion below Corollary \ref{Cor:distvert}, it follows that the Euclidean center of $B_H(z_0,\rho)$ is $(e^\rho+e^{-\rho})/2$, thus by reflectional symmetry with respect to the hyperplane $\mathbb{R}^d \times \{(e^\rho+e^{-\rho})/2\}$, it suffices to show that $B_{\mathbb{R}^d}(0,e^{\frac{2}{5}\rho}) \subset B_H(0,\rho)$ for $\rho$ large enough. It follows from Corollary 2 in the Supplementary materials that, for $r$ large enough, $d((0,e^0),(x,e^0)) \le 5/2\ln(r)$ for all $x \in B_{\mathbb{R}^d}(0,r)$, thus, for $\rho$ large enough $B_{\mathbb{R}^d}(0,e^{\frac{2}{5}\rho}) \subset B_H(z_0,\rho)$ and the claim is proved.

Finally, we can easily compute the volume of $\mathcal{C}_\rho$:
\begin{eqnarray}
\mu(\mathcal{C}_\rho)=\eta e^{\frac{2}{5}d\rho} \int_1^{e^\rho-e^{-\rho}} \frac{dz}{z^{d+1}}\sim\frac{\eta}{d} e^{\frac{2}{5}d\rho} \text{ when } \rho \to \infty.
\end{eqnarray}
where $\eta$ denotes the volume of the (Euclidean) unit ball in $\mathbb{R}^d$.
Thus, for $\rho$ large enough,
\begin{eqnarray}
\mu(B^+(z_0,\rho)) \ge \mu(\mathcal{C}_\rho) \ge \frac{\eta}{2d} e^{\frac{2}{5}d\rho} \ge e^{d\rho/3},
\end{eqnarray}
this achieves the proof.

\subsection{Controlling the number of points at a given level}
\label{S:nop}

We finally prove Proposition \ref{Prop:nop}. By the dilation invariance property of model, it is enough to show it for $t=0$. Let $R>0$. We will in fact prove that $(\#\mathcal{L}_0 \cap B(0,R))^p$ admits exponential moments. Let $n \in \mathbb{N}$, and $L>0$ depending on $n$ that will be chosen later. Let us partition $\#\mathcal{L}_0 \cap B(0,R)$ in two sets:
\begin{eqnarray*}
E_{\le L}:=\{x \in \#\mathcal{L}_0 \cap B(0,R),~d((x,e^0)_\downarrow, (0,e^0)) \le L\}, \nonumber\\
E_{>L}:=\{x \in \#\mathcal{L}_0 \cap B(0,R),~d((x,e^0)_\downarrow, (0,e^0))>L\}.
\end{eqnarray*}
We have
\begin{eqnarray}
\label{E:majorprobaterm0}
\mathbb{P}[\#\mathcal{L}_0 \cap B(0,R) \ge n] \le \mathbb{P}[E_{\le L} \ge n]+\mathbb{P}[E_{>L} \neq \emptyset].
\end{eqnarray}
Then we will upperbound the two terms of this sum.

\bigbreak

\paragraph{Step 1:} we upperbound $\mathbb{P}[E_{\le L} \ge n]$.

Clearly, $\#E_{\le L} \le \#(N \cap B_H((0,e^0),L)$. Let us denote by $\mathcal{V}$ the Hyperbolic volume of $\#(N \cap B_H((0,e^0),L)$. We use the following lemma to estimate $\mathcal{V}$:

\begin{lemma}[Volume of a Hyperbolic ball]
\label{Lem:volball}
For any $z_0 \in \mathbb{H}^{d+1}$,
\begin{eqnarray*}
\mu(B_{(z_0,\rho)}) \sim \frac{\mathcal{S}(d)}{d2^d} e^{d\rho}\text{ when } \rho \to \infty.
\end{eqnarray*}
\end{lemma}

This Lemma follows from Lemma \ref{Lem:polarcoord} applied to $f=\mathbf{1}_{[0,\rho]}$ and easy computations. Then, when $L \to \infty$, large enough, $\mathcal{V} \sim \mathcal{S}(d)/(d2^d) e^{dL}=O(e^{dL})$. So $\#N \cap B_H((0,e^0),L)$ is distributed according to a Poisson law of parameter $\lambda\mathcal{V} \le Ce^{dL}$ for some constant $C$ large enough. We use the following Chernoff bound \cite{boundPoisson}:

\begin{lemma}[Chernoff bound for a Poisson distribution]
If $X$ is distributed according to a Poisson low of parameter $\alpha>0$, then, for $n \ge \alpha$,
\begin{eqnarray*}
\mathbb{P}[X \ge n] \le \frac{e^{-\alpha}(e\alpha)^n}{n^n}.
\end{eqnarray*}
\end{lemma}

See \cite{boundPoisson} for a proof. Applying this bound to $\#N \cap B_H((0,e^0),L)$ leads to:
\begin{eqnarray}
\label{E:boundterm1}
\mathbb{P}[\#E_{\le L} \ge n] \le \mathbb{P}[\#(N \cap B_H((0,e^0),L) \ge n] \le \frac{\exp(-Ce^{dL})(C e^{dL+1})^n}{n^n}
\end{eqnarray}
if $n \ge Ce^{dL}$ and for $L$ large enough.

\bigbreak

\paragraph{Step 2:} we upperbound $\mathbb{P}[E_{> L} \neq \emptyset]$.

For $x \in E_{>L}$, by triangular inequality and Corollary 1 in the Supplementary materials,
\begin{eqnarray*}
&d((x,e^0)_\downarrow,(x,e^0)_\uparrow) &\ge d((x,e^0),(x,e^0)_\downarrow) \nonumber\\
&&\ge d((x,e^0)_\downarrow,(0,e^0))-d((x,e^0),(0,e^0)) \nonumber\\
\end{eqnarray*}
The second part of Corollary 2 in the Supplementary materials gives that $d((x,e^0),(0,e^0)) \le \|x\| \le R$. Then  $d((x,e^0)_\downarrow,(x,e^0)_\uparrow) \ge d((x,e^0)_\downarrow,(0,e^0))-R$.
Thus $N \cap B^+((x,e^0)_\downarrow,d((x,e^0)_\downarrow,(0,e^0))-R)=\emptyset$.
Let us define the set
\begin{eqnarray*}
E'_{>L}:=\{z \in N,~d(z,(0,e^0))>L \text{ and } N \cap B^+(z,d(z,(0,e^0))-R)=\emptyset \}.
\end{eqnarray*}
Then $\#E_{>L} \le \#E'_{L}$. Consider the function $f:\mathcal{N}_S \times H \to \mathbb{R}_+$ defined as
\begin{eqnarray*}
f(z,\eta):=\mathbf{1}_{d(z,(0,e^0))>L}\mathbf{1}_{B^+(z,d(z,(0,e^0))-R)=\emptyset}.
\end{eqnarray*}
By Campbell formula \cite{stochastic} and Fubini,
\begin{eqnarray}
\label{E:applycampbell}
\mathbb{E}\left[ \#E'_{>L} \right]=\mathbb{E}\left[ \sum_{z \in N} f(z,N \backslash \{z\}) \right]=\mathbb{E}\left[ \int_H f(z,N)~dz \right]=\int_H \mathbb{P}[z \in E']~dz.
\end{eqnarray}
For $z \in H$ such that $d(z,(0,e^0)) \le L$, $\mathbb{P}[z \in E'_{>L}]=0$. If $d(z,(0,e^0))>L$, then
\begin{eqnarray}
\label{E:boundproba}
\mathbb{P}[z \in E'_{>L}]=\exp(-\lambda \mu(B^+(z,d(z,(0,e^0))-R))) \overset{(\ref{E:lowbplus})}{\le} \exp\left(-\lambda e^{d (d(z,(0,e^0))-R)/3}\right).
\end{eqnarray}
Thus, using the change of coordinates formula (Lemma \ref{Lem:polarcoord}),
\begin{eqnarray*}
&\mathbb{E}\left[ \#E'_{>L} \right] &\overset{(\ref{E:applycampbell}),(\ref{E:boundproba})}{\le} \int_H \exp\left(-\lambda e^{d (d(z,(0,e^0))-R)/3}\right)~dz \nonumber\\
&&=\mathcal{S}(d) \int_L^\infty \sinh(\rho)^d \exp\left(-\lambda e^{d(\rho-R)/3} \right)~d\rho \nonumber\\
&&\le \mathcal{S}(d) \int_L^\infty \exp\left(d\rho-\lambda e^{d(\rho-R)/3} \right)~d\rho \text{ for $L$ large enough}.
\end{eqnarray*}
For $L$ large enough, since $d\rho-\lambda e^{d(\rho-R)/3} \le -e^{\rho/4}$,
\begin{eqnarray}
\label{E:boundterm2}
&\mathbb{E}\left[ \#E'_{>L} \right] &\le \int_L^\infty e^{-e^{\rho/4}} ~d\rho \le \int_L^\infty e^{\rho/4}e^{-e^{\rho/4}}=\left[ -4e^{-e^{\rho/4}}\right]_L^\infty=4e^{-e^{L/4}}.
\end{eqnarray}

\bigbreak

\paragraph{Step 3:} conclusion.
We now combine upperbounds obtained in Step 1 and Step 2. Let us take 
\begin{eqnarray*}
L=\frac{1}{d}\ln\left(\frac{n}{2C}\right),
\end{eqnarray*}
then $Ce^{dL}=n/2$. Consider $n$ large enough such that upperbounds (\ref{E:boundterm1}) and (\ref{E:boundterm2}) are satisfied. Then
\begin{eqnarray*}
&\mathbb{P}[\#\mathcal{L}_0 \cap B(0,R) \ge n] &\overset{(\ref{E:majorprobaterm0})}{\le} \mathbb{P}[E_{\le L} \ge n]+\mathbb{P}[E_{>L} \neq \emptyset]\nonumber\\
&&\overset{(\ref{E:boundterm1}),(\ref{E:boundterm2})}{\le} \frac{\exp(-n/2)(en/2)^n}{n^n}+4\exp\left(-e^{1/(4d)\ln(n/(2C))}\right) \nonumber\\
&&=\left(\frac{e^{1/2}}{n}\right)^n+4\exp\left(-\left(\frac{n}{2C}\right)^{\frac{1}{4d}} \right) \nonumber\\
&& \le e^{-n^{1/(5d)}} \text{ for $n$ large enough}.
\end{eqnarray*}
Therefore $\#\mathcal{L}_0 \cap B(0,R) \in L^p$ for all $p \ge 1$, this achives the proof of Proposition \ref{Prop:nop}. 


\end{document}